\documentclass[11pt,a4paper,oneside,french,english]{amsart}
\usepackage{lmodern}

\usepackage[T1]{fontenc}
\usepackage[latin9]{inputenc}
\usepackage{verbatim}
\usepackage{amsthm}
\usepackage{amstext}
\usepackage{amssymb}
\usepackage{graphicx}
\usepackage{esint}

\makeatletter

\numberwithin{equation}{section}
\numberwithin{figure}{section}
\theoremstyle{plain}
\newtheorem{thm}{\protect\theoremname}[section]
  \theoremstyle{definition}
  \newtheorem{example}[thm]{\protect\examplename}
  \theoremstyle{plain}
  \newtheorem{prop}[thm]{\protect\propositionname}
  \theoremstyle{plain}
  \newtheorem{cor}[thm]{\protect\corollaryname}
  \theoremstyle{definition}
  \newtheorem{rem}[thm]{\protect\remarkname}
  \theoremstyle{definition}
  \newtheorem{defn}[thm]{\protect\definitionname}
  \theoremstyle{plain}
  \newtheorem{lem}[thm]{\protect\lemmaname}

\@ifundefined{date}{}{\date{}}
\makeatother

\usepackage{babel}
\addto\extrasfrench{%
   
   \providecommand{\fg}{\ifdim\lastskip>\z@\unskip\fi~\frqq}
}

  \addto\captionsenglish{\renewcommand{\corollaryname}{Corollary}}
  \addto\captionsenglish{\renewcommand{\definitionname}{Definition}}
  \addto\captionsenglish{\renewcommand{\examplename}{Example}}
  \addto\captionsenglish{\renewcommand{\lemmaname}{Lemma}}
  \addto\captionsenglish{\renewcommand{\propositionname}{Proposition}}
  \addto\captionsenglish{\renewcommand{\remarkname}{Remark}}
  \addto\captionsenglish{\renewcommand{\theoremname}{Theorem}}
  \addto\captionsfrench{\renewcommand{\corollaryname}{Corollaire}}
  \addto\captionsfrench{\renewcommand{\definitionname}{Definition}}
  \addto\captionsfrench{\renewcommand{\examplename}{Exemple}}
  \addto\captionsfrench{\renewcommand{\lemmaname}{Lemme}}
  \addto\captionsfrench{\renewcommand{\propositionname}{Proposition}}
  \addto\captionsfrench{\renewcommand{\remarkname}{Remarque}}
  \addto\captionsfrench{\renewcommand{\theoremname}{Theoreme}}
  \providecommand{\corollaryname}{Corollary}
  \providecommand{\definitionname}{Definition}
  \providecommand{\examplename}{Example}
  \providecommand{\lemmaname}{Lemma}
  \providecommand{\propositionname}{Proposition}
  \providecommand{\remarkname}{Remark}
\providecommand{\theoremname}{Theorem}

\begin{document}

\title{Symplectic embeddings of $4$-dimensional ellipsoids into cubes}

\author{David Frenkel and Dorothee M\"uller}

\thanks{Partially supported by SNF grant 200020-132000.}

\begin{abstract}
Recently, McDuff and Schlenk determined in \cite{MS} the function
$c_{EB}(a)$ whose value at $a$ is the infimum of the size of a $4$-ball
into which the ellipsoid $E(1,a)$ symplectically embeds (here, $a\geqslant1$
is the ratio of the area of the large axis to that of the smaller
axis of the ellipsoid). In this paper we look at embeddings into four-dimensional
cubes instead, and determine the function $c_{EC}(a)$ whose value
at $a$ is the infimum of the size of a $4$-cube $C^{4}(A)=D^{2}(A)\times D^{2}(A)$
into which the ellipsoid $E(1,a)$ symplectically embeds (where $D^{2}(A)$
denotes the disc in $\mathbb{R}^{2}$ of area $A$). As in the case
of embeddings into balls, the structure of the graph of $c_{EC}(a)$
is very rich: for $a$ less than the square~$\sigma^{2}$ of the
silver ratio $\sigma:=1+\sqrt{2}$, the function $c_{EC}(a)$ turns
out to be piecewise linear, with an infinite staircase converging
to $(\sigma^{2},\sqrt{\sigma^{2}/2})$. This staircase is determined
by Pell numbers. On the interval $\left[\sigma^{2},7\frac{1}{32}\right]$,
the function $c_{EC}(a)$ coincides with the volume constraint $\sqrt{\frac{a}{2}}$
except on seven disjoint intervals, where $c$ is piecewise linear.
Finally, for $a\geqslant7\frac{1}{32}$, the functions $c_{EC}(a)$
and $\sqrt{\frac{a}{2}}$ are equal.

For the proof, we first translate the embedding problem $E(1,a)\hookrightarrow C^{4}(A)$
to a certain ball packing problem of the ball $B^{4}(2A)$. This embedding
problem is then solved by adapting the method from \cite{MS}, which
finds all exceptional spheres in blow-ups of the complex projective
plane that provide an embedding obstruction.

We also prove that the ellipsoid $E(1,a)$ symplectically embeds into
the cube $C^{4}(A)$ if and only if $E(1,a)$ symplectically embeds
into the elllipsoid $E(A,2A)$. Our embedding function $c_{EC}(a)$
thus also describes the smallest dilate of $E(1,2)$ into which $E(1,a)$
symplectically embeds.
\end{abstract}
\maketitle
\tableofcontents{}

\section{Introduction}

\subsection{Statement of the result}

Let $\left(\mathbb{R}^{4},\omega\right)$ be the Euclidean $4$-dimensional
space endowed with the canonical symplectic form $\omega=dx_{1}\wedge dy_{1}+dx_{2}\wedge dy_{2}$.
Any open subset of $\mathbb{R}^{4}$ is also endowed with $\omega$.
Simple examples are the symplectic cylinders $Z(a):=D^{2}(a)\times\mathbb{R}^{2}$
(where $D^{2}(a)$ is the open disc of area $a$), the open symplectic
ellipsoids
\[
E(a_{1},a_{2})\,=\,\left\{ \left(x_{1},y_{1},x_{2},y_{2}\right)\in\mathbb{R}^{4}\,:\,\frac{\pi\left(x_{1}^{2}+y_{1}^{2}\right)}{a_{1}}+\frac{\pi\left(x_{2}^{2}+y_{2}^{2}\right)}{a_{2}}<1\right\} ,
\]
and the open polydiscs $P\left(a_{1},a_{2}\right):=D^{2}\left(a_{1}\right)\times D^{2}\left(a_{2}\right)$.
We denote the open ball $E(a,a)$ (of radius $\sqrt{a/\pi}$) by $B(a)$
and the open cube $P(a,a)$ by~$C(a)$. Since $D^{2}(a)$ is symplectomorphic
to an open square, $D^{2}(a)\times D^{2}(a)$ is indeed symplectomorphic
to a cube.

Given two open subsets $U$ and $V$, we say that a smooth embedding
$\varphi\colon U\hookrightarrow V$ is a \emph{symplectic embedding}
if $\varphi$ preserves $\omega$, that is, if $\varphi^{*}\omega=\omega$.
In the sequel, we will write $\varphi\colon U\overset{s}{\hookrightarrow}V$
for such an embedding. Since symplectic embeddings are volume preserving,
a necessary condition for the existence of a symplectic embedding
$U\overset{s}{\hookrightarrow}V$ is, of course, $\textrm{Vol}(U)\leqslant\textrm{Vol}(V)$,
where $\textrm{Vol}(U):=\frac{1}{2}\intop_{U}\omega\wedge\omega$.
For volume preserving embeddings, this is the only condition (see
e.g.~\cite{S1}). For symplectic embeddings, however, the situation
is very different, as was detected by Gromov in~\cite{G}. Among
many other things, he proved
\begin{example}
(Gromov's nonsqueezing Theorem) There exists a symplectic embedding
of the ball $B(a)$ into the cylinder $Z(A)$ if and only if $a\leqslant A$.
\end{example}
Notice that the volume of the cylinder $Z(A)$ is infinite, and that
for any $a$ the ball $B(a)$ embeds by a linear volume preserving
embedding into $Z(A)$. Similarly, we also have
\begin{example}
There exists a symplectic embedding of the ball $B(a)$ into the cube
$C(A)$ if and only if $a\leqslant A$.
\end{example}
The above results show that symplectic embeddings are much more special
and in some sense {}``more rigid'' than volume preserving embeddings.
A next step was to understand this rigidity better. One way of doing
this is to fix a domain $V\subset\mathbb{R}^{4}$ of finite volume,
and to try to determine for each $k\in\mathbb{N}$ the $k$\emph{-th
packing number}
\[
p_{k}(V)\,:=\,\sup\left\{ \frac{k\,\textrm{Vol}\left(B(a)\right)}{\textrm{Vol}(V)}\,:\,\bigsqcup_{k}B(a)\;\overset{s}{\hookrightarrow}\; V\right\} .
\]
Here, $\bigsqcup_{k}B(a)$ is the disjoint union of $k$ equal balls
$B(a)$. It follows from Darboux's Theorem that always $p_{k}(V)>0$.
If $p_{k}(V)=1$, one says that $V$ admits a \emph{full packing}
by $k$ balls, and if $p_{k}(V)<1$, one says that there is a \emph{packing
obstruction}. Again, it is known that if we would consider volume
preserving embeddings instead, then all packing numbers would always
be~1.

In imporant work by Gromov~\cite{G}, McDuff-Polterovich~\cite{MP}
and Biran~\cite{B} all the packing numbers of the $4$-ball $B$
and the $4$-cube $C$ were determined. The result for $C$ is

\begin{center} \renewcommand{\arraystretch}{1.5} \begin{tabular}{ccccccccc} \hline  $k$ & $1$ & $2$ & $3$ & $4$ & $5$ & $6$ & $7$ & $\geqslant8$\tabularnewline \hline  \noalign{\vskip1bp} $p_{k}$ & $\frac{1}{2}$ & $1$ & $\frac{2}{3}$ & $\frac{8}{9}$ & $\frac{9}{10}$ & $\frac{48}{49}$ & $\frac{224}{225}$ & $1$\tabularnewline[1bp] \hline  \end{tabular}\medskip{}
\par\end{center}%

\noindent This result shows that, while there is symplectic rigidity
for many small $k$, there is no rigidity at all for large $k$.

In order to better understand these numbers, we look at a problem
that interpolates the above problem of packing by $k$ equal balls.
For $0<a_{1}\leqslant a_{2}$, consider the ellipsoid $E\left(a_{1},a_{2}\right)$
defined above, and look for the smallest cube~$C(A)$ into which
$E\left(a_{1},a_{2}\right)$ symplectically embeds. Since $E(a_{1},a_{2})\overset{s}{\hookrightarrow}C(A)$
if and only if $E(\lambda a_{1},\lambda a_{2})\overset{s}{\hookrightarrow}C(\lambda A)$,
we can always assume that $a_{1}=1$, and therefore study the \emph{embedding
capacity function}
\[
c_{EC}(a)\,:=\,\inf\left\{ A\,:\, E(1,a)\;\overset{s}{\hookrightarrow}\; C(A)\right\} 
\]
on the interval $\left[1,\infty\right[$. It is clear that $c$ is
a continuous and nondecreasing function. Since symplectic embeddings
are volume preserving and the volumes of $E(1,a)$ and $C(A)$ are
$\frac{1}{2}a$ and $A^{2}$ respectively, we must have the lower
bound
\[
\sqrt{\frac{a}{2}}\leqslant c(a).
\]
It is not hard to see that $\bigsqcup_{k}B(1)\overset{s}{\hookrightarrow}E(1,k)$.
Therefore, $\bigsqcup_{k}B(1)\overset{s}{\hookrightarrow}C(A)$ whenever
$E(1,k)\overset{s}{\hookrightarrow}C(A)$. In~\cite{M2}, McDuff
has shown that the converse is also true! Our ellipsoid embedding
problem therefore indeed interpolates the problem of packing by $k$
equal balls, and we get
\[
p_{k}(C)=\frac{\textrm{Vol}\left(E(1,k)\right)}{\textrm{Vol}\left(C\left(c_{EC}(k)\right)\right)}=\frac{\frac{k}{2}}{\left(c_{EC}(k)\right)^{2}}.
\]

First upper estimates for the function $c_{EC}(a)$ were obtained
in Chapter~4.4 of~\cite{S2} by explicit embeddings of ellipsoids
into a cube. These upper estimates also suggested that symplectic
rigidity for the problem $E(1,a)\overset{s}{\hookrightarrow}C(A)$
should disappear for large $a$.

In this paper, we completely determine the function $c(a):=c_{EC}(a)$.
In order to state our main theorem, we introduce two sequences of
integers: the \emph{Pell numbers} $P_{n}$ and the \emph{half companion
Pell numbers} $H_{n}$, which are defined by the recurrence relations
\[
\begin{alignedat}{5}P_{0} & = & \:0,\, & P_{1} & = & \:1,\quad & P_{n} & = & 2P_{n-1}+P_{n-2},\\
H_{0} & = & \:1,\, & H_{1} & = & \:1,\quad & H_{n} & = & \:2H_{n-1}+H_{n-2}.
\end{alignedat}
\]
Thus, $P_{2}=2,$ $P_{3}=5$, $P_{4}=12$, $P_{5}=29,\ldots$ and
$H_{2}=3$, $H_{3}=7$, $H_{4}=17$, $H_{5}=41,\ldots$. The two sequences
$\left(\alpha_{n}\right)_{n\geqslant0}$ and $\left(\beta_{n}\right)_{n\geqslant0}$
are then defined by
\begin{flalign*}
\alpha_{n}:=\left\{ \begin{aligned}\frac{2P_{n+1}^{2}}{H_{n}^{2}} & \qquad\textrm{if }n\textrm{ is even},\\
\frac{H_{n+1}^{2}}{2P_{n}^{2}} & \qquad\textrm{if }n\textrm{ is odd};
\end{aligned}
\right. & \qquad\beta_{n}:=\left\{ \begin{aligned}\frac{H_{n+2}}{H_{n}} & \qquad\textrm{if }n\textrm{ is even},\\
\frac{P_{n+2}}{P_{n}} & \qquad\textrm{if }n\textrm{ is odd}.
\end{aligned}
\right.
\end{flalign*}
The first terms in these sequences are
\[
\alpha_{0}=2<\beta_{0}=3<\alpha_{1}=\frac{9}{2}<\beta_{1}=5<\alpha_{2}=\frac{50}{9}<\beta_{2}=\frac{17}{3}<\ldots.
\]
More generally, for all $n\geqslant0$,
\[
\ldots<\alpha_{n}<\beta_{n}<\alpha_{n+1}<\beta_{n+1}<\ldots,
\]
and both sequences converge to $\sigma^{2}=3+2\sqrt{2}\cong5.83$,
which is the square of the silver ration $\sigma:=1+\sqrt{2}$.
\begin{thm}
\label{thm:statement of the result}\renewcommand{\labelenumi}{(\roman{enumi})}
\begin{enumerate} \item On the interval $\left[1,\sigma^{2}\right]$, \[ c(a)=\left\{ \begin{aligned}1\quad\; & \qquad\textrm{if }a\in\left[1,2\right],\\ \frac{1}{\sqrt{2\alpha_{n}}}\, a & \qquad\textrm{if }a\in\left[\alpha_{n},\beta_{n}\right],\\ \sqrt{\frac{\alpha_{n+1}}{2}} & \qquad\textrm{if }a\in\left[\beta_{n},\alpha_{n+1}\right], \end{aligned} \right. \]  for all $n\geqslant0$ (see Figure \ref{fig:c on [1,sigma^2]}).\medskip{}
\item On the interval $\left[\sigma^{2},7\frac{1}{32}\right]$ we have $c(a)=\sqrt{\frac{a}{2}}$ except on seven disjoint intervals, where $c$ is piecewise linear (see Figure \ref{fig:c on [sigma^2,7+1/32]}).\medskip{}
\item For $a\geqslant7\frac{1}{32}$ we have $c(a)=\sqrt{\frac{a}{2}}$.\end{enumerate}%

\end{thm}
The proof of (i) is given in Corollary \ref{cor:c(a) on [1,sigma^2]},
a more detailed statement as well as the proof of (ii) are given in
Theorem \ref{thm:c(a) on [sigma^2,7+1/32]}, while the proof of (iii)
is given in Lemma \ref{lem:c(a)=00003Dsqrt(a/2) for a>8} and Proposition
\ref{prop:c(a)=00003Dsqrt(a/2) for a>7+1/32}.

A similar result has been previously obtained by McDuff-Schlenk in~\cite{MS}
for the embedding problem $E(1,a)\overset{s}{\hookrightarrow}B(A)$.
These two results show that the structure of symplectic rigidity can
be very rich.

\selectlanguage{french}%
\begin{figure}
\begin{centering}
\includegraphics{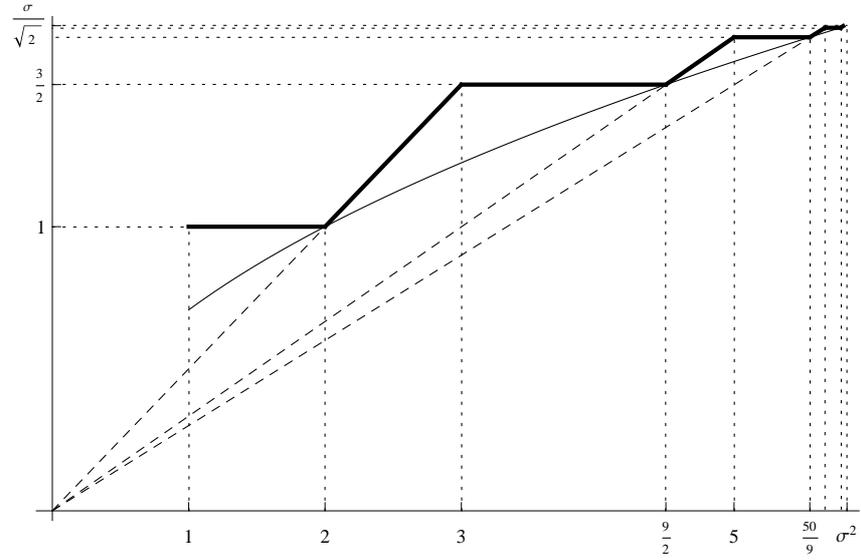}
\par\end{centering}

\caption{\label{fig:c on [1,sigma^2]}The graph of $c$ on the interval $\left[1,\sigma^{2}\right]$}
\end{figure}

\begin{figure}
\begin{centering}
\includegraphics[scale=0.8]{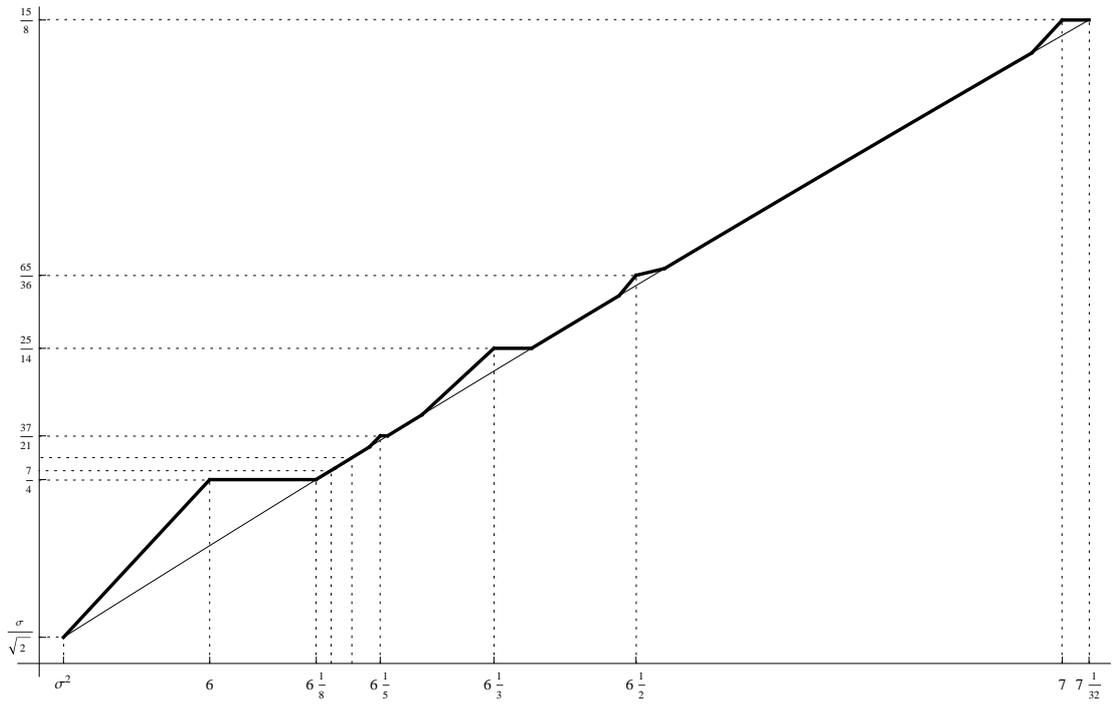}
\par\end{centering}

\caption{\label{fig:c on [sigma^2,7+1/32]}The graph of $c$ on the interval
$\left[\sigma^{2},7\frac{1}{32}\right]$}
\end{figure}

\selectlanguage{english}%

\subsection{Relations to ECH-capacities}

There is a more combinatorial (but non-explicit) way of describing
the embedding function $c_{EC}(a)$. Indeed, in~\cite{H1}, Hutchings
used his embedded contact homology to construct for each domain $U\subset\mathbb{R}^{4}$
a sequence of symplectic capacities $c_{ECH}^{k}(U)$, which for the
ellipsoid $E(a,b)$ and the polydisc $P(a,b)$ are as follows.

Form the sequence $N_{E}(a,b)$ by arranging all numbers of the form
$ma+nb$ with $m,n\geqslant0$, in nondecreasing order (with multiplicities).
Then for $k=1,2,\ldots$, the $k$-th ECH-capacity $c_{ECH}^{k}\left(E(a,b)\right)$
is the $k$-th entry of $N_{E}(a,b)$. For instance, $c_{ECH}\left(E(1,1)\right)=\left\{ 1,1,2,2,2,3,3,3,3,4,\ldots\right\} $.

Moreover, for polydiscs,
\[
c_{ECH}^{k}\left(P(a,b)\right)\,=\,\min\left\{ am+bn\,:\, m,n\in\mathbb{N},\,(m+1)(n+1)\geqslant k+1\right\} .
\]

There exists a canonical way to decompose an ellipsoid $E(a,b)$ with
$\frac{a}{b}$ rational into a finite disjoint union of balls $B(a,b):=\sqcup_{i}B\left(w_{i}\right)$
with weights~$w_{i}$ related to the continued fraction expansion
of $\frac{a}{b}$. We shall explain this decomposition in more detail
and prove the following proposition in the next section.
\begin{prop}
\label{prop:E(a,b) -> P(c,d) iff B(a,b)... -> B(c+d)}Let $a,b,c,d>0$
with $\frac{a}{b}$ rational. Then there exists a symplectic embedding
$E(a,b)\,\hookrightarrow\, P(c,d)$ if and only if there exists a
symplectic embedding
\[
B(a,b)\bigsqcup B(c)\bigsqcup B(d)\;\hookrightarrow\; B(c+d).
\]

\end{prop}
Hutchings showed in Corollary 11 of~\cite{H2} how Proposition \ref{prop:E(a,b) -> P(c,d) iff B(a,b)... -> B(c+d)}
implies that ECH-capacities form a complete set of invariants for
the problem of symplectically embedding an ellipsoid into a polydisc:
\begin{cor}
\label{cor:ECH capacities are sharp for EP}There exists a symplectic
embedding $E(a,b)\,\hookrightarrow\, P(c,d)$ if and only if $c_{ECH}^{k}\left(E(a,b)\right)\leqslant c_{ECH}^{k}\left(P(c,d)\right)$
for all $k\geqslant1$.
\end{cor}
\noindent It seems to be hard to derive Theorem~\ref{thm:statement of the result}
from Corollary~\ref{cor:ECH capacities are sharp for EP} or vice-versa.\medskip{}

As a further corollary we obtain
\begin{cor}
The ellipsoid $E(1,a)$ symplectically embeds into the cube $C(A)$
if and only if $E(1,a)$ symplectically embeds into the ellipsoid
$E(A,2A)$.\end{cor}
\begin{proof}
By Corollary \ref{cor:ECH capacities are sharp for EP}, $E(1,a)$
symplectically embeds into $C(A)$ if and only if $c_{ECH}^{k}\left(E(1,a)\right)\leqslant c_{ECH}^{k}\left(C(A)\right)$
for all $k\geqslant1$. By McDuff's proof of the Hofer Conjecture~\cite{M3},
$E(1,a)$ symplectically embeds into $E(A,2A)$ if and only if $c_{ECH}^{k}\left(E(1,a)\right)\leqslant c_{ECH}^{k}\left(E(A,2A)\right)$
for all $k\geqslant1$. The corollary now follows from the remark
on page 8098 in~\cite{H2}, that says that for all $k\geqslant1$
\begin{equation}
c_{ECH}^{k}\left(E(1,2)\right)=c_{ECH}^{k}\left(C(1)\right).\label{eq:identity ECH}
\end{equation}
For the easy proof, we refer to Section \ref{sec:proof of proposition}.\end{proof}
\begin{rem}
Recall that the ECH capacities of $B(1)$ and $C(1)$ (or $E(1,2)$)
are
\[
\begin{aligned}c_{ECH}\left(B(1)\right) & =\left(0^{\times1},1^{\times2},2^{\times3},3^{\times4},4^{\times5},5^{\times6},6^{\times7},7^{\times8},8^{\times9},9^{\times10},\ldots\right),\\
c_{ECH}\left(C(1)\right) & =\left(0^{\times1},1^{\times1},2^{\times2},3^{\times2},4^{\times3},5^{\times3},6^{\times4},7^{\times4},8^{\times5},9^{\times5},\ldots\right).
\end{aligned}
\]
One sees that the sequence $c_{ECH}\left(C(1)\right)$ is obtained
from $c_{ECH}\left(B(1)\right)$ by some sort of doubling. This is
reminiscent to the doubling in the denition of the Pell numbers: The
Fibonacci and Pell numbers are defined recursively by
\[
F_{n+1}=F_{n}+F_{n-1},\qquad P_{n+1}=2P_{n}+P_{n-1},
\]
and while the Fibonacci numbers determine the infinite stairs of the
function $c_{EB}(a)$ for $a\leqslant\tau^{4}$ (with $\tau$ the
golden ratio, see~\cite{MS}), the Pell numbers determine the infinite
stairs of the function $c_{EC}(a)$ for $a\leqslant\sigma^{2}$. This
reminiscence may, however, be a coincidence. Indeed, for the ellipsoid
$E(1,3)$ the sequence
\[
c_{ECH}\left(E(1,3\right)=\left(0^{\times1},1^{\times1},2^{\times1},3^{\times2},4^{\times2},5^{\times2},6^{\times3},7^{\times3},8^{\times3},9^{\times4},\ldots\right)
\]
is obtained from $c_{ECH}(B(1))$ by some sort of trippling, but the
beginning of the function describing the embedding problem $E(1,a)\overset{s}{\hookrightarrow}E(A,3A)$
seems not to be given in terms of numbers defined by $G_{n+1}=3G_{n}+G_{n-1}$.
\end{rem}
\noindent \textbf{Acknowledgments.} We wish to sincerely thank Felix
Schlenk for his precious help and encouragement during the whole project
and R\'egis Straubhaar for his help with all the computer issues.

\section{Proof of Proposition \ref{prop:E(a,b) -> P(c,d) iff B(a,b)... -> B(c+d)}
and equalities (\ref{eq:identity ECH})\label{sec:proof of proposition}}

\noindent In Section \ref{sub:Decomposing-an-ellipsoid}, we explain
the canonical decomposition of $E(1,a)$ with $a\in\mathbb{Q}$ into
a disjoint union of balls. We then prove Proposition \ref{prop:E(a,b) -> P(c,d) iff B(a,b)... -> B(c+d)}
in Sections \ref{sub:Representations-of-balls} and \ref{sub:Proof-of-Proposition},
and in Section \ref{sub:Proof-of-equalities} we prove equalities
(\ref{eq:identity ECH}).

\subsection{Decomposing an ellipsoid into a disjoint union of balls\label{sub:Decomposing-an-ellipsoid}}

In \cite{M2}, McDuff showed the following theorem.
\begin{thm}
\emph{(McDuff \cite{M2}) }Let $a,b>0$ be two rational numbers. Then,
there exists a finite sequence $\left(w_{1},\ldots,w_{M}\right)$
of rational numbers such that the closed ellipsoid $\overline{E}(a,b)$
symplectically embeds into the ball $B(A)$ if and only if the disjoint
union of balls $\sqcup_{i}\overline{B}\left(w_{i}\right)$ symplectically
embed into $B(A)$.
\end{thm}
The disjoint union $\sqcup_{i}\overline{B}\left(w_{i}\right)$ is
then denoted by $\overline{B}(a,b)$. Following \cite{MS}, we will
now explain one way to compute the weights $w_{1},\ldots,w_{M}$ in
this decomposition. Notice that in~\cite{M2}, the weights of the
balls $B\left(w_{i}\right)$ are defined in a slightly different way.
The proof that these weights agree with the weight expansion of $a$
defined now can be found in the Appendix of \cite{MS}.
\begin{defn}
Let $a=\frac{p}{q}\geqslant1$ be a rational number written in lowest
terms. The \emph{weight expansion} of $a$ is the finite sequence
$w(a):=\left(w_{1},\ldots,w_{M}\right)$ defined recursively by\end{defn}
\begin{itemize}
\item $w_{1}=1$, and $w_{n}\geqslant w_{n+1}>0$ for all $n$;
\item if $w_{i}>w_{i+1}=\ldots=w_{n}$ (where we set $w_{0}:=a$), then
\[
w_{n+1}=\left\{ \begin{array}{cl}
w_{n} & \textrm{if }w_{i+1}+\ldots+w_{n+1}=(n-i+1)w_{i+1}\leqslant w_{i},\\
w_{i}-(n-i)w_{i+1} & \textrm{otherwise};
\end{array}\right.
\]

\item the sequence stops at $w_{n}$ if the above formula gives $w_{n+1}=0$.\end{itemize}
\begin{rem}
If we regard this weight expansion as consisting of $N+1$ blocks
on which the $w_{i}$ are constant, that is
\[
w(a)=(\underset{l_{0}}{\underbrace{1,\ldots,1},\,}\underset{l_{1}}{\underbrace{x_{1},\ldots,x_{1}}},\ldots,\underset{l_{N}}{\underbrace{x_{N},\ldots,x_{N}}})=\left(1^{\times l_{0}},x_{1}^{\times l_{1}},\ldots,x_{N}^{\times l_{N}}\right),
\]
then $x_{1}=a-l_{0}<1$, and if we set $x_{0}=1$, then for all $2\leqslant i\leqslant N$,
$x_{i}=x_{i-2}-l_{i-1}x_{i-1}$. Moreover, the lengths of the blocks
give the continued fraction of $a$ since
\[
a=l_{0}+\frac{1}{l_{1}+\frac{{\textstyle 1}}{{\textstyle l_{2}+\frac{{\textstyle 1}}{{\textstyle \ddots+\frac{{\textstyle 1}}{{\textstyle l_{N}}}}}}}}=:\left[l_{0};l_{1},\ldots,l_{N}\right].
\]
\end{rem}
\begin{example}
The weight expansion of $\frac{25}{9}$ is $\left(1^{\times2},\frac{7}{9},\frac{2}{9}^{\times3},\frac{1}{9}^{\times2}\right)$.
The continued fraction expansion of $\frac{25}{9}$ is thus $\left[2;1,3,2\right]$.
Notice that we also have
\[
\frac{25}{9}=2\cdot1^{2}+\left(\frac{7}{9}\right)^{2}+3\cdot\left(\frac{2}{9}\right)^{2}+2\cdot\left(\frac{1}{9}\right)^{2}.
\]
This is no accident and is best explained geometrically as in Figure
\ref{fig:weight expansion}. The general result is stated in the next
lemma.

\selectlanguage{french}%
\begin{figure}
\begin{centering}
\includegraphics[scale=0.3]{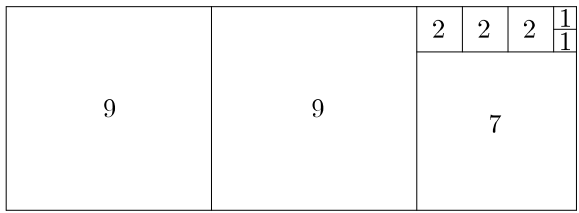}
\par\end{centering}

\caption{\label{fig:weight expansion}The weight expansion of $\frac{25}{9}$.}
\end{figure}
\end{example}
\begin{lem}
\label{lem:weight expansion} \emph{(McDuff-Schlenk \cite{MS}, Lemma
1.2.6)}\renewcommand{\labelenumi}{(\roman{enumi})}
Let $a=\frac{p}{q}\geqslant1$ be a rational number with $p,q$ relatively prime, and let $w:=w(a)=\left(w_{1},...,w_{M}\right)$ be its weight expansion. Then
\begin{enumerate}
\item $w_{M}=\frac{1}{q}$;
\item $\sum w_{i}^{2}=\left\langle w,w\right\rangle =a$;\smallskip{}
\item $\sum w_{i}=a+1-\frac{1}{q}$.
\end{enumerate} %

\end{lem}

\subsection{Representations of balls and polydiscs\label{sub:Representations-of-balls}}

In the proof of Proposition \ref{prop:E(a,b) -> P(c,d) iff B(a,b)... -> B(c+d)},
we shall use certain ways of representing open and closed balls and
open polydiscs. Recall that $B(a)$ is the open ball in~$\mathbb{R}^{4}$
of capacity $a=\pi r^{2}$, and that $P(a,b)=D^{2}(a)\times D^{2}(b)$,
where $D^{2}(a)$ is the open disc in~$\mathbb{R}^{2}$ of area~$a$.

\subsubsection{Representations as products\label{sub:Representations-as-products}}

Denote by $\square(a,b)$ the open square $\left]0,a\right[\times\left]0,b\right[$
in $\mathbb{R}^{2}$. Since $D^{2}(a)$ is symplectomorphic to the
open square $\left]0,a\right[\times\left]0,1\right[$, the polydisc
$P(a,b)$ is symplectomorphic to
\[
\square(a,b)\times\square(1,1)\,\subset\,\mathbb{R}^{2}(x)\times\mathbb{R}^{2}(y).
\]
Next, consider the simplex
\[
\triangle(a)\,:=\,\left\{ (x_{1},x_{2})\in\mathbb{R}^{2}(x)\;:\;0<x_{1},x_{2}\;,\; x_{1}+x_{2}<a\right\} .
\]
Then $B(a)$ is symplectomorphic to the product
\[
\triangle(a)\times\square(1,1)\,\subset\,\mathbb{R}^{2}(x)\times\mathbb{R}^{2}(y),
\]
see~\cite{T} and Remark 9.3.1 of~\cite{S2}.

\subsubsection{Representations by the Delzant polytope\label{sub:Representations-Delzant}}

As before, denote by $\omega_{SF}$ the Study-Fubini form on the complex
projective plane~$\mathbb{C}P^{2}$, normalized by $\int_{\mathbb{C}P^{1}}\omega_{SF}=1$.
We write $\mathbb{C}P^{2}(a)$ for $\left(\mathbb{C}P^{2},a\,\omega_{SF}\right)$.
Its affine part $\mathbb{C}P^{2}\setminus\mathbb{C}P^{1}$ is symplectorphic
to the open ball~$B(a)$. (Indeed, for $a=\pi$, the embedding
\[
z=(z_{1},z_{2})\,\mapsto\,\left[z_{1}:z_{2}:\sqrt{1-|z|^{2}}\right]
\]
is symplectic.)

The image of the moment map of the usual $T^{2}$-action on $\mathbb{C}P^{2}(a)$
is the closed triangle~$\overline{\triangle(a)}$. For $b<a$, the
preimage of $\overline{\triangle(b)}\subset\overline{\triangle(a)}$
is symplectomorphic to $\overline{B}(b)$. By precomposing the torus
action with suitable linear torus automorphisms, one sees that also
the closed triangles based at the other two corners of $\overline{\triangle(a)}$
correspond to closed balls in~$\mathbb{C}P^{2}(a)$. We refer to~\cite{K}
for details.

The image of the moment map of the usual $T^{2}$-action on~$\mathbb{C}^{2}$
maps the polydisc $P(c,d)$ to the rectangle $\left[0,c\right[\times\left[0,d\right[\subset\mathbb{R}^{2}(x)$.

\selectlanguage{french}%
\begin{figure}
\begin{centering}
\includegraphics[scale=0.4]{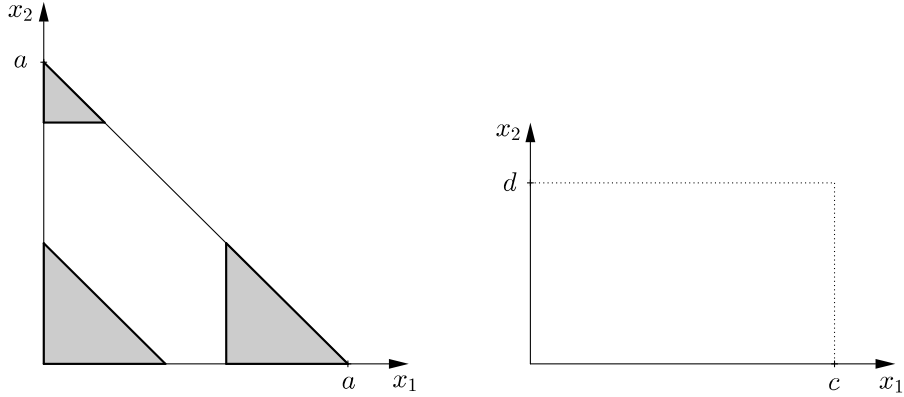}
\par\end{centering}

\caption{\label{fig:0}Closed balls in $\mathbb{C}P^{2}(a)$ and the moment
image of $P(c,d)$}
\end{figure}

\selectlanguage{english}%

\subsection{Proof of Proposition \ref{prop:E(a,b) -> P(c,d) iff B(a,b)... -> B(c+d)}\label{sub:Proof-of-Proposition}}

Let now $a,b,c,d>0$ with $\frac{a}{b}$ rational. We need to show
that
\[
E(a,b)\;\overset{s}{\hookrightarrow}\; P(c,d)\;\Longleftrightarrow\, B(a,b)\sqcup B(c)\sqcup B(d)\;\overset{s}{\hookrightarrow}\; B(c+d).
\]

\noindent ''$\Longrightarrow$'': By decomposing $E(a,b)$ into
balls as before, we find that $B(a,b)\,\overset{s}{\hookrightarrow}\, P(c,d)$,
(see also~\cite{M2}). Fix $\varepsilon>0$. Then also $(1-\varepsilon)\overline{B}(a,b)\,\overset{s}{\hookrightarrow}\, P(c,d)$.
Now represent the open balls $B(c),B(d),B(c+d)$ and the polydisc
as in Section~\ref{sub:Representations-as-products} above. We then
read off from Figure \ref{fig:1} that
\[
(1-\varepsilon)\overline{B}(a,b)\sqcup B(c)\sqcup B(d)\;\overset{s}{\hookrightarrow}\; B(c+d).
\]
This holds for every $\varepsilon>0$. In view of~\cite{M1} we then
also find a symplectic embedding $B(a,b)\sqcup B(c)\sqcup B(d)\,\overset{s}{\hookrightarrow}\, B(c+d)$.\foreignlanguage{french}{}
\begin{figure}
\selectlanguage{french}%
\begin{centering}
\includegraphics[scale=0.5]{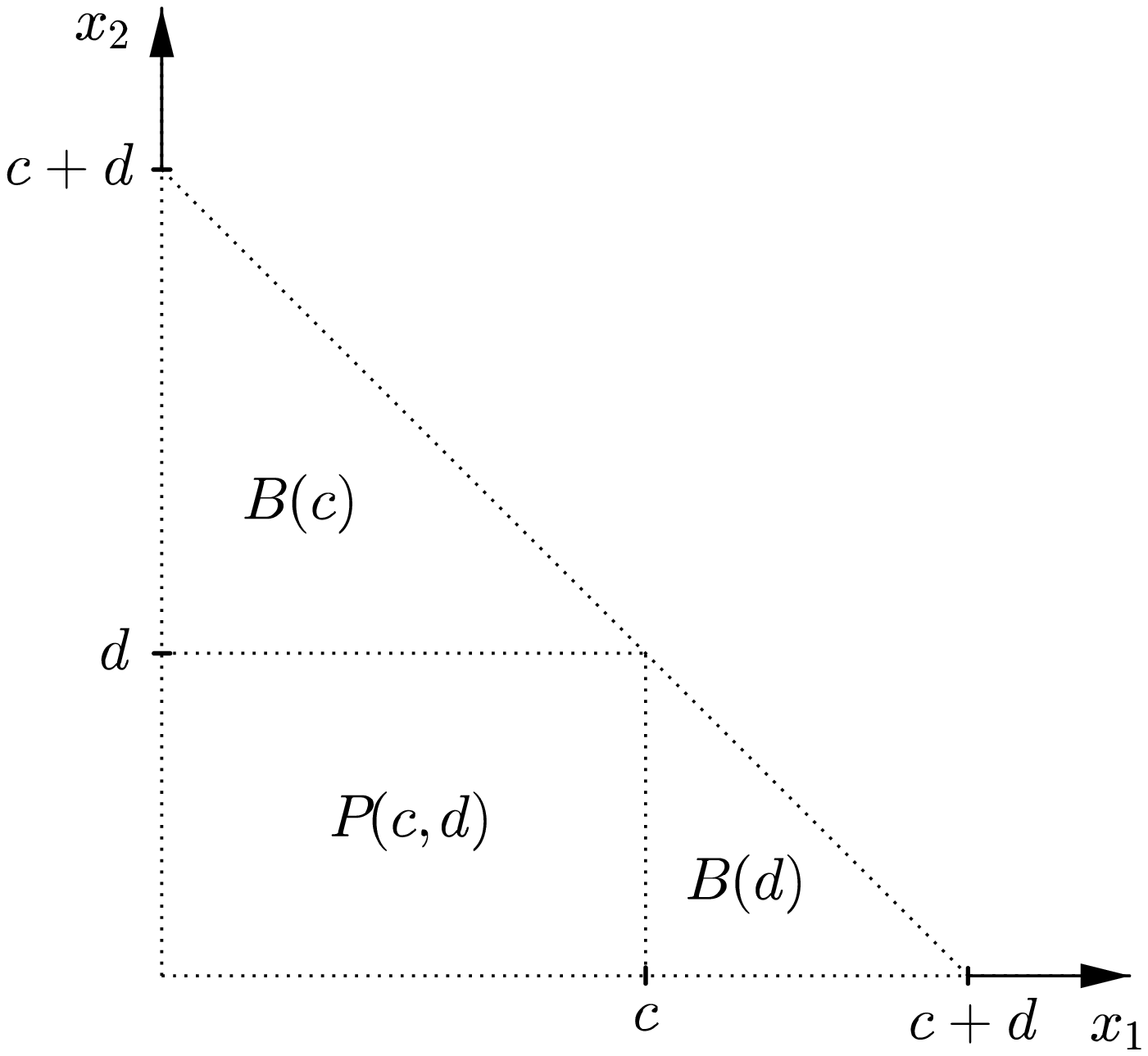}
\par\end{centering}

\caption{\label{fig:1}\foreignlanguage{english}{$(1-\varepsilon)\overline{B}(a,b)\sqcup B(c)\sqcup B(d)\;\protect\overset{s}{\hookrightarrow}\; B(c+d)$}}
\selectlanguage{english}
\end{figure}

{}``$\Longleftarrow":$ Assume now that $B(a,b)\sqcup B(c)\sqcup B(d)\,\overset{s}{\hookrightarrow}\, B(c+d)$.
Fix $\varepsilon>0$. Then
\[
(1-\varepsilon)\overline{B}(a,b)\sqcup\overline{B}(c-\varepsilon)\sqcup\overline{B}(d-\varepsilon)\;\overset{s}{\hookrightarrow}\;\mathbb{C}P^{2}(c+d).
\]
According to~\cite{M1}, the space of symplectic embeddings of $\overline{B}(c-\varepsilon)\sqcup\overline{B}(d-\varepsilon)$
into~$\mathbb{C}P^{2}(c+d)$ is connected. Any such isotopy extends
to an ambient symplectic isotopy of~$\mathbb{C}P^{2}(c+d)$. In view
of this and by Section~\ref{sub:Representations-Delzant} we can
thus assume that the balls $\overline{B}(c-\varepsilon)$ and $\overline{B}(d-\varepsilon)$
lie in $\mathbb{C}P^{2}(c+d)$ as shown in Figure \ref{fig:2}.

\selectlanguage{french}%
\begin{figure}
\begin{centering}
\includegraphics[scale=0.6]{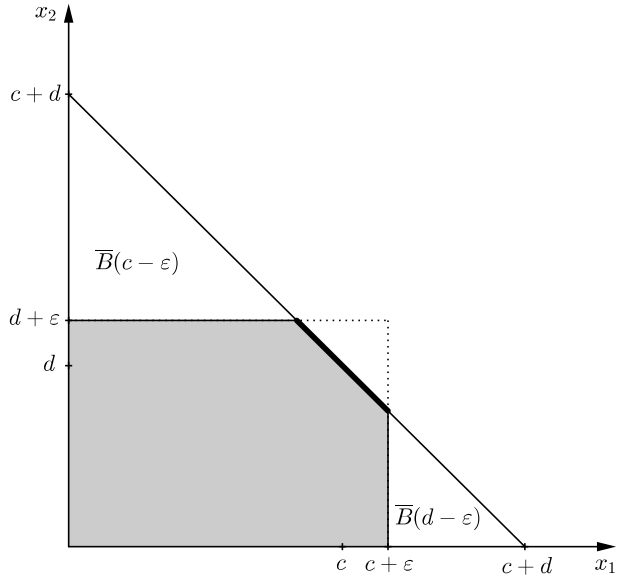}
\par\end{centering}

\caption{\label{fig:2}\foreignlanguage{english}{How $\overline{B}(c-\varepsilon)$
and $\overline{B}(d-\varepsilon)$ lie in $\mathbb{C}P^{2}(c+d)$}}
\end{figure}
\foreignlanguage{english}{The image of the balls $(1-\varepsilon)\overline{B}(a,b)$
must then lie over the gray shaded closed region. However, since the
balls $\overline{B}(c-\varepsilon)$ and $\overline{B}(d-\varepsilon)$
are closed, the image of $(1-\varepsilon)\overline{B}(a,b)$ cannot
touch the upper horizontal or the right vertical boundary of the gray
shaded region. Moreover, according to Remark 2.1.E of~\cite{MP}
we can assume that this image lies in the affine part of $\mathbb{C}P^{2}(c+d)$,
i.e., the image of the balls $(1-\varepsilon)\overline{B}(a,b)$ lies
over the gray region deprived from the dark segment, and hence, by
Section~\ref{sub:Representations-Delzant}, in $P(c+\varepsilon,d+\varepsilon)$.
We may suppose from the start that $c,d\geqslant1$. Then $P(c+\varepsilon,d+\varepsilon)\subset(1+\varepsilon)P(c,d)$.
We have thus found a symplectic embedding $(1-\varepsilon)\overline{B}(a,b)\,\overset{s}{\hookrightarrow}\,(1+\varepsilon)P(c,d)$.
It is shown in Theorem 1.5 of~\cite{M2} that then also $(1-\varepsilon)\overline{E}(a,b)\,\overset{s}{\hookrightarrow}\,(1+\varepsilon)P(c,d)$.
Hence
\[
\frac{1-\varepsilon}{1+\varepsilon}\,\overline{E}(a,b)\;\overset{s}{\hookrightarrow}\; P(c,d).
\]
It now follows again from~\cite{M2} that $E(a,b)\,\overset{s}{\hookrightarrow}\, P(c,d)$.
(To be precise,~\cite{M2} considers embeddings of ellipsoids into
open balls; however, the same arguments work for embeddings of ellipsoids
into polydiscs.)\hfill{}$\square$}

\selectlanguage{english}%

\subsection{Proof of equalities (\ref{eq:identity ECH})\label{sub:Proof-of-equalities}}
\begin{lem}
For all $k\geqslant1$,
\[
c_{ECH}^{k}\left(E(1,2)\right)=c_{ECH}^{k}\left(C(1)\right).
\]
\end{lem}
\begin{proof}
We will prove that $c_{ECH}^{k}\left(E(1,2)\right)$ and $c_{ECH}^{k}\left(C(1)\right)$
are both equal to the unique integer $d$ such that
\[
\left\lfloor \frac{d+1}{2}\right\rfloor \left\lceil \frac{d+1}{2}\right\rceil <k\leqslant\left\lfloor \frac{d+2}{2}\right\rfloor \left\lceil \frac{d+2}{2}\right\rceil .
\]
For $c_{ECH}^{k}\left(E(1,2)\right)$, this follows from the fact
that the number
\[
\sharp\left\{ (m,n)\in\mathbb{N}_{0}^{2}:m+2n\leqslant d\right\} 
\]
of pairs of nonnegative integers $(m,n)$ such that $m+2n\leqslant d$
is equal to $\left\lfloor \frac{d+2}{2}\right\rfloor \left\lceil \frac{d+2}{2}\right\rceil $.
This, in turn, can easily be deduced from the identities 
\[
\sharp\left\{ (m,n)\in\mathbb{N}_{0}^{2}:m+2n=2l\right\} =\sharp\left\{ (m,n)\in\mathbb{N}_{0}^{2}:m+2n=2l+1\right\} =l+1.
\]
On the other hand, we have by definition that
\[
c_{ECH}^{k}\left(C(1)\right)=c_{ECH}^{k}\left(P(1,1)\right)=\min\left\{ m+n:(m+1)(n+1)\geqslant k+1\right\} .
\]
Fix a nonnegative integer $k$. Let $m_{0},n_{0}\in\mathbb{N}_{0}$
be two nonnegative integers such that
\[
m_{0}+n_{0}=\min\left\{ m+n:(m+1)(n+1)\geqslant k+1\right\} .
\]
Without loss of generality, $m_{0}\geqslant n_{0}$. Moreover, we
can always take $m_{0},n_{0}$ such that $m_{0}-n_{0}\in\left\{ 0,1\right\} $.
Indeed, assume that $m_{0}=n_{0}+c$ with $c\geqslant2$. Then for
$m_{0}'=m_{0}-1$ and $n_{0}'=n_{0}+1$, we get
\[
\begin{aligned}\left(m_{0}'+1\right)\left(n_{0}'+1\right) & =m_{0}\left(n_{0}+2\right)=\left(n_{0}+c\right)\left(n_{0}+2\right)=n_{0}^{2}+(c+2)n_{0}+2c\\
 & >n_{0}^{2}+(c+2)n_{0}+c+1=\left(n_{0}+c+1\right)\left(n_{0}+1\right)\\
 & =\left(m_{0}+1\right)\left(n_{0}+1\right)\geqslant k+1.
\end{aligned}
\]
Thus $\left(m_{0}',n_{0}'\right)$ also realizes the minimum. Now,
if $m_{0}+n_{0}$ is even, then $m_{0}=n_{0}$ and we have to show
that
\[
\begin{aligned}\left\lfloor \frac{2m_{0}+1}{2}\right\rfloor \left\lceil \frac{2m_{0}+1}{2}\right\rceil  & =m_{0}\left(m_{0}+1\right)<k\\
 & \leqslant\left(m_{0}+1\right)^{2}=\left\lfloor \frac{2m_{0}+2}{2}\right\rfloor \left\lceil \frac{2m_{0}+2}{2}\right\rceil .
\end{aligned}
\]
The first inequality follows from the minimality of $m_{0}+n_{0}$
while the second one follows from the fact that $\left(m_{0}+1\right)\left(n_{0}+1\right)\geqslant k+1$.
The case $m_{0}+n_{0}$ odd is treated similarly.
\end{proof}

\section{Reduction to a constraint function given by exceptional spheres\label{sec:Method-of-proof}}

\noindent In this section we explain how the function $c(a)$ can
be described by the volume constraint $\sqrt{\frac{a}{2}}$ and the
constraints coming from certain exceptional spheres in blow-ups of
$\mathbb{C}P^{2}$. Since the function $c$ is continuous, it suffices
to determine $c$ for each rational $a\geqslant1$. The starting point
is the following lemma, which is a special case of Proposition \ref{prop:E(a,b) -> P(c,d) iff B(a,b)... -> B(c+d)}.
\begin{lem}
\label{lem:E(1,a) embeds into C(A) iff balls embed in balls}Let $a\geqslant1$
be a rational number with weight expansion $w(a)=\left(w_{1},\ldots,w_{M}\right)$
and $A>0$. Then the ellipsoid $E(1,a)$ embeds symplectically into
the cube $C(A)$ if and only if there is a symplectic embedding
\[
B(A)\sqcup B(A)\sqcup_{i}B\left(w_{i}\right)\overset{s}{\hookrightarrow}B(2A).
\]

\end{lem}
With this lemma, we have converted the problem of embedding an ellipsoid
into a cube to the problem of embedding a disjoint union of balls
into a ball. In \cite{MP}, the problem of embedding $k$ disjoint
balls into a ball was reduced to the question of understanding the
symplectic cone of the $k$-fold blow-up $X_{k}$ of~$\mathbb{C}P^{2}$.
Let $L:=\left[\mathbb{C}P^{1}\right]\in H_{2}\left(X_{k},\mathbb{Z}\right)$
be the class of a line, let $E_{1},\ldots,E_{k}\in H_{2}\left(X_{k},\mathbb{Z}\right)$
be the homology classes of the exceptional divisors, and denote by
$l,e_{1},\ldots,e_{k}\in H^{2}\left(X_{k},\mathbb{Z}\right)$ their
Poincar\'e duals. Let $-K:=3L-\sum E_{i}$ be the anti-canonical divisor
of $X_{k}$, and define the corresponding \emph{symplectic cone} $\mathcal{C}_{K}\left(X_{k}\right)\subset H^{2}\left(X_{k},\mathbb{Z}\right)$
as the set of classes represented by symplectic forms~$\omega$ with
first Chern class $c_{1}\left(M,\omega\right)=-K$.
\begin{thm}
\emph{\label{thm:McDuff-Polterovich}(McDuff-Polterovich \cite{MP})}
The union $\sqcup_{i=1}^{k}\overline{B}\left(w_{i}\right)$ embeds
into the ball $B(\mu)$ or into $\mathbb{C}P^{2}(\mu)$ if and only
if $\mu\;\! l-\sum w_{i}e_{i}\in\mathcal{C}_{K}\left(X_{k}\right)$.
\end{thm}
To understand $\mathcal{C}_{K}\left(X_{k}\right)$, we define as in
\cite{MS} the following set $\mathcal{E}_{k}\subset H_{2}\left(X_{k}\right)$.
\begin{defn}
$\mathcal{E}_{k}$ is the set consisting of $(0;-1,0,\ldots,0)$ and
of all tuples $(d;m):=\left(d;m_{1},\ldots,m_{k}\right)$ with $d\geqslant0$
and $m_{1}\geqslant\ldots\geqslant m_{k}\geqslant0$ such that the
class $E_{(d;m)}:=dL-\sum m_{i}E_{i}\in H_{2}\left(X_{k}\right)$
is represented in $X_{k}$ by a symplectically embedded sphere of
self-intersection $-1$.
\end{defn}
We will often write $\mathcal{E}$ instead of $\mathcal{E}_{k}$ if
there is no danger of confusion. We then have the following description
of $\mathcal{C}_{K}\left(X_{k}\right)$.
\begin{prop}
\emph{\label{prop:Li-Li}(Li-Li \cite{LiLi}, Li-Liu \cite{LiLiu})}
\[
\mathcal{C}_{K}\left(X_{k}\right)=\left\{ \alpha\in H^{2}\left(X_{k}\right):\alpha^{2}>0,\,\alpha\left(E\right)>0,\forall E\in\mathcal{E}_{k}\right\} .
\]

\end{prop}
In order to give a characterization of the set $\mathcal{E}_{k}$,
we need the following definition as in \cite{MS}.
\begin{defn}
\label{def: Cremona move}A tuple $(d;m):=\left(d;m_{1},\ldots m_{k}\right)$
is said to be \emph{ordered} if the~$m_{i}$ are in nonincreasing
order. The \emph{Cremona transform} of an ordered tuple $(d;m)$ is
\[
\left(2d-m_{1}-m_{2}-m_{3};d-m_{2}-m_{3},d-m_{1}-m_{3},d-m_{1}-m_{2},m_{4},\ldots,m_{k}\right).
\]
A \emph{Cremona move} of a tuple $(d;m)$ is the composition of the
Cremona transform of $(d;m)$ with any permutation of the new obtained
vector $m$.\end{defn}
\begin{prop}
\emph{\label{prop:characterization of E_k in MS}(McDuff-Schlenk \cite{MS},
Proposition 1.2.12 and Remark 3.3.1)}

\renewcommand{\labelenumi}{(\roman{enumi})}
\begin{enumerate}
\item All $(d;m)\in\mathcal{E}_{k}$ satisfy the two Diophantine equations \[ \begin{aligned}\sum m_{i} & =3d-1,\\ \sum m_{i}^{2} & =d^{2}+1. \end{aligned} \]
\item For all distinct $\left(d;m\right),(d';m')\in\mathcal{E}_{k}$ we have \[ \sum m_{i}m_{i}'\leqslant dd'. \]
\item A tuple $(d;m)$ belongs to $\mathcal{E}_{k}$ if and only if $(d;m)$ satisfies the Diophantine equations in (i) and $(d;m)$ can be reduced to $(0;-1,0,\ldots,0)$ by repeated Cremona moves.
\end{enumerate} %
\end{prop}
\begin{rem}
\label{rmk:change of basis}Working directly with Lemma \ref{lem:E(1,a) embeds into C(A) iff balls embed in balls},
Theorem \ref{thm:McDuff-Polterovich} and Proposition \ref{prop:Li-Li}
we find, as in \cite{MS}, that the only constraints for an embedding
$E(1,a)\,\overset{s}{\hookrightarrow}\, C(A)$ are $A\geqslant\sqrt{\frac{a}{2}}$
and, for each class $(d;m)\in\mathcal{E}_{k}$,
\begin{equation}
2Ad\;\geqslant\;\left(m_{1}+m_{2}\right)\, A+\left\langle \left(m_{3},\ldots,m_{k}\right),\: w(a)\right\rangle .\label{eq:AA}
\end{equation}
One can start from here and use Proposition \ref{prop:characterization of E_k in MS}
to prove Theorem \ref{thm:statement of the result}. The analysis
becomes, however, rather awkward, since the unknown $A$ appears on
both sides of (\ref{eq:AA}).

To improve the situation, we shall apply a base change of $H_{2}\left(X_{k}\right)$,
and express the elements of $\mathcal{E}$ in a new basis. Consider
the product $S^{2}\times S^{2}$ (whose affine part is a cube), and
form the $M$-fold (topological) blow-up $X_{M}\left(S^{2}\times S^{2}\right)$.
A basis of $H_{2}\left(X_{M}\left(S^{2}\times S^{2}\right)\right)$
is given by $S_{1},S_{2},F_{1},\ldots,F_{M}$, where $S_{1}:=\left[S^{2}\times\left\{ \textrm{point}\right\} \right]$,
$S_{2}:=\left[\left\{ \textrm{point}\right\} \times S^{2}\right]$
and $F_{1},\ldots,F_{M}$ are the classes of the exceptional divisors.

Notice that there is a diffeomorphism $\varphi\colon X_{M}\left(S^{2}\times S^{2}\right)\rightarrow X_{M+1}\left(\mathbb{C}P^{2}\right)$
such that the induced map in homology is
\[
\begin{array}{cccc}
\varphi_{*}\colon & H_{2}\left(X_{M}\left(S^{2}\times S^{2}\right)\right) & \longrightarrow & H_{2}\left(X_{M+1}\left(\mathbb{C}P^{2}\right)\right)\\
 & \begin{alignedat}{1}S_{1}\vphantom{E_{1}^{1^{1^{1^{1^{1}}}}}}\\
S_{2}\\
F_{1}\\
F_{i}
\end{alignedat}
 & \begin{alignedat}{1}\vphantom{E_{1}^{1^{1^{1^{1^{1}}}}}}\longmapsto\vphantom{E_{1}}\\
\vphantom{E_{1}}\longmapsto\vphantom{E_{1}}\\
\vphantom{E_{1}}\longmapsto\vphantom{E_{1}}\\
\vphantom{E_{1}}\longmapsto\vphantom{E_{1}}
\end{alignedat}
 & \begin{alignedat}{2}L & -E_{1} &  & \vphantom{E_{1}^{1^{1^{1^{1^{1}}}}}}\\
L &  & -E_{2}\\
L & -E_{1} & -E_{2}\\
 &  &  & E_{i+1}.
\end{alignedat}
\end{array}
\]
The existence of such a $\varphi$ is clear from a moment map picture
such as Figure~\ref{fig:2} above. With respect to the new basis
$S_{1},S_{2},F_{1},\ldots,F_{M}$ we write an element of $H_{2}\left(X_{M}\left(S^{2}\times S^{2}\right)\right)$
as $\left(d,e;m_{1},\ldots,m_{M}\right)$. Then
\[
\varphi_{*}(d,e;m)\;=\;\left(d+e-m_{1};d-m_{1},f-m_{1},m_{2},\ldots,m_{M}\right).
\]
In the new basis, the constraint given by a class in $\mathcal{E}$
can be written in a more useful form:\end{rem}
\begin{prop}
\label{prop:characterization of E_M}

\renewcommand{\labelenumi}{(\roman{enumi})}
\begin{enumerate} \item All $(d,e;m)\in\mathcal{E}_{M}$ satisfy the two Diophantine equations \[ \begin{aligned}\sum m_{i} & =2(d+e)-1,\\ \sum m_{i}^{2} & =2de+1. \end{aligned} \]
\item For all distinct $(d,e;m),(d',e';m')\in\mathcal{E}_{M}$, we have \[ \sum m_{i}m_{i}'\leqslant de'+d'e. \]
\item A tuple $(d,e;m)$ belongs to $\mathcal{E}_{M}$ if and only if $(d,e;m)$ satisfies the Diophantine equations of (i) and its image under $\varphi_{*}$ can be reduced to $(0;-1,0,\ldots,0)$ by repeated Cremona moves.
\end{enumerate}%
\end{prop}
\begin{proof}
Let $E\in\mathcal{E}$. The two identities in Proposition \ref{prop:characterization of E_k in MS}\,(i)
correspond to $c_{1}(E)=1$ and $E\cdot E=-1$. For $E=d\, S_{1}+e\, S_{2}-\sum m_{i}F_{i}$
these identities become
\[
\begin{aligned}c_{1}\left(E\right) & =2d+2e-\sum m_{i}=1,\\
E\cdot E & =-\sum m_{i}^{2}+2de=-1,
\end{aligned}
\]
proving (i). Assertion (ii) of Proposition \ref{prop:characterization of E_k in MS}
corresponds to positivity of intersection of $J$-holomorphic spheres
representing $E,E'\in\mathcal{E}$. For distinct elements $E=(d,e;m)$
and $E'=(d',e';m')$ in $\mathcal{E}$ we thus have
\[
E\cdot E'=de'+ed'-\sum m_{i}m_{i}'\geqslant0,
\]
proving (ii). Assertion (iii) holds since $\varphi_{*}$ is a base
change.
\end{proof}
In the sequel, given two vectors $m$ and $w$ of length $M$, we
will denote by $\left\langle m,w\right\rangle =\sum_{i=1}^{M}m_{i}w_{i}$
the Euclidean scalar product in $\mathbb{R}^{M}$. Notice that we
will also use this notation for vectors $m$ and $w$ of different
lengths, meaning the Euclidean scalar product of the two vectors after
adding enough zeros at the end of the shorter one.
\begin{prop}
\label{prop:characterization of c(a)}Let $a\geqslant1$ be a rational
number with weight expansion $w(a)=\left(w_{1},\ldots,w_{M}\right)$.
For $(d,e;m)\in\mathcal{E}$, define the constraint 
\[
\mu(d,e;m)(a):=\frac{\left\langle m,w(a)\right\rangle }{d+e}.
\]
Then
\[
c(a)=\sup_{(d,e;m)\in\mathcal{E}}\left\{ \sqrt{\frac{a}{2}},\mu(d,e;m)(a)\right\} .
\]
\end{prop}
\begin{proof}
By Lemma \ref{lem:E(1,a) embeds into C(A) iff balls embed in balls},
$E(1,a)\overset{s}{\hookrightarrow}C(A)$ if and only if
\[
B(A)\sqcup B(A)\sqcup_{i}B\left(w_{i}\right)\overset{s}{\hookrightarrow}B(2A).
\]
By Theorem \ref{thm:McDuff-Polterovich}, this is true if and only
if
\begin{equation}
\left(2A\right)l-A\;\! e_{1}-A\;\! e_{2}-\sum_{i=1}^{M}w_{i}\;\! e_{i+2}\;\in\;\mathcal{C}_{K}.\label{eq:AAA}
\end{equation}
Denote by $s_{1},s_{2},f_{1},\ldots,f_{M}$ the Poincar\'e duals of
$S_{1},S_{2},F_{1},\ldots,F_{M}$. The base change in cohomology is
then
\[
\begin{array}{cccc}
\varphi^{*}\colon & H^{2}\left(X_{M+1}\left(\mathbb{C}P^{2}\right)\right) & \longrightarrow & H^{2}\left(X_{M}\left(S^{2}\times S^{2}\right)\right)\\
 & \begin{alignedat}{1}l\vphantom{f_{1}^{1^{1^{1^{1^{1}}}}}}\\
e_{1}\\
e_{2}\\
e_{i}
\end{alignedat}
 & \begin{alignedat}{1}\vphantom{f_{1}^{1^{1^{1^{1^{1}}}}}}\longmapsto\vphantom{f_{1}}\\
\vphantom{f_{1}}\longmapsto\vphantom{f_{1}}\\
\vphantom{f_{1}}\longmapsto\vphantom{f_{1}}\\
\vphantom{f_{1}}\longmapsto\vphantom{f_{1}}
\end{alignedat}
 & \begin{alignedat}{2}s_{1}+ & s_{2} & -f_{1} & \vphantom{f_{1}^{1^{1^{1^{1^{1}}}}}}\\
 & s_{2} & -f_{1}\\
s_{1}\hphantom{+} &  & -f_{1}\\
 &  &  & f_{i-1}.
\end{alignedat}
\end{array}
\]
In this new basis of $H^{2}\left(X_{M}\left(S^{2}\times S^{2}\right)\right)$,
(\ref{eq:AAA}) therefore becomes 
\begin{equation}
A\:\! s_{1}+A\:\! s_{2}-\sum_{i=1}^{M}w_{i}\:\! f_{i+1}\;\in\;\mathcal{C}_{K}.\label{eq:A}
\end{equation}
In view of Proposition \ref{prop:Li-Li}, (\ref{eq:A}) translates
to the conditions that for all $E:=(d,e;m)\in\mathcal{E}$, we have
$2A^{2}-\sum w_{i}^{2}>0$ and 
\[
A\:\! s_{1}(E)+A\:\! s_{2}(E)-\sum w_{i}\:\! f_{i}(E)\;=\;(d+e)A-\sum m_{i}w_{i}\;>\;0.
\]
Recall from Lemma \ref{lem:weight expansion}\,(ii) that $\sum w_{i}^{2}=a$.
We conclude that $E(1,a)\overset{s}{\hookrightarrow}C(A)$ if and
only if $A>\sqrt{\frac{a}{2}}$ and for all $(d,e;m)\in\mathcal{E}$
\[
A\;>\;\frac{\sum m_{i}w_{i}}{d+e}=\frac{\left\langle m,w(a)\right\rangle }{d+e}.
\]
This proves the proposition.\end{proof}
\begin{rem}
By the symmetry between $d$ and $e$ in the formula for $\mu(d,e;m)(a)$,
we can assume that all elements $(d,e;m)\in\mathcal{E}$ have $d\geqslant e$.
We will use this convention troughout the paper.
\end{rem}
The rest of this paper is devoted to the analysis of the constraints
given in Proposition \ref{prop:characterization of c(a)}. This analysis
follows the one in \cite{MS}. However, several modifications are
necessary.

\section{Basic observations}
\begin{lem}
\label{lem:c(a)=00003Dsqrt(a/2) for a>8}For all $a\geqslant8$, $c(a)=\sqrt{\frac{a}{2}}$.\end{lem}
\begin{proof}
For all $(d,e;m)\in\mathcal{E}$, we have by definition of $\mu$
and Proposition \ref{prop:characterization of E_M}\,(i) 
\[
\mu(d,e;m)(a):=\frac{\left\langle m,w(a)\right\rangle }{d+e}\leqslant\frac{\sum m_{i}}{d+e}=\frac{2(d+e)-1}{d+e}<2\leqslant\sqrt{\frac{a}{2}}
\]
for all $a\geqslant8$. Therefore, by Proposition \ref{prop:characterization of c(a)},
$c(a)=\sqrt{\frac{a}{2}}$ for all $a\geqslant8$.\end{proof}
\begin{lem}
\label{lem:scaling property of c}The function $c$ has the following
scaling property: for all $\lambda\geqslant1$,
\[
\frac{c(\lambda a)}{\lambda a}\leqslant\frac{c(a)}{a}.
\]
\end{lem}
\begin{proof}
By definition of $c$, $E(1,a)\overset{s}{\hookrightarrow}C\left(c(a)+\varepsilon\right)$
for all $\varepsilon>0$. Since $E(1,a)$ symplectically embeds into
$C(A)$ if and only if $E(\lambda,\lambda a)$ symplectically embeds
into $C(\lambda A)$, this is equivalent to $E(\lambda,\lambda a)\overset{s}{\hookrightarrow}C\left(\lambda c(a)+\varepsilon\right)$
for all $\varepsilon>0$. Since $E(1,\lambda a)\subset E(\lambda,\lambda a)$
when $\lambda>1$, this implies that 
\[
E(1,\lambda a)\overset{s}{\hookrightarrow}C\left(\lambda c(a)+\varepsilon\right)
\]
for all $\varepsilon>0$. Thus
\[
c(\lambda a):=\inf\left\{ A:E(1,\lambda a)\overset{s}{\hookrightarrow}C(A)\right\} \leqslant\lambda c(a)=\lambda a\frac{c(a)}{a},
\]
as claimed.\end{proof}
\begin{lem}
\label{lem:finite list of elements in E_M for M leq 7}For $M\leqslant7$,
the sets $\mathcal{E}_{M}$ are finite and the only elements are
\[
\begin{array}{ccccc}
\left(0,0;-1\right), & \left(1,0;1\right), & \left(1,1;1^{\times3}\right), & \left(2,1;1^{\times5}\right), & \left(2,2;2,1^{\times5}\right),\\
\left(3,1;1^{\times7}\right), & \left(3,2;2^{\times2},1^{\times5}\right), & \left(3,3;2^{\times4},1^{\times3}\right), & \left(4,3;2^{\times6},1\right), & \left(4,4;3,2^{\times6}\right).
\end{array}
\]
\end{lem}
\begin{proof}
By Proposition \ref{prop:characterization of E_M}\,(i), we have
for a class $(d,e;m)\in\mathcal{E}$ with $l(m)=k$,
\[
\left(2(d+e)-1\right)^{2}=\left(\sum_{i=1}^{k}m_{i}\right)^{2}\leqslant k\sum_{i=1}^{k}m_{i}^{2}=k(2de+1),
\]
which is equivalent to
\[
4\left(d^{2}+e^{2}\right)+8de-4(d+e)+1\leqslant2kde+k,
\]
and to
\[
d^{2}+e^{2}\leqslant\frac{k-4}{2}de+d+e+\frac{k-1}{4}.
\]
Now, if $k\leqslant7$, then
\[
d^{2}+e^{2}\leqslant\frac{3}{2}de+d+e+\frac{3}{2}\leqslant\frac{3}{4}\left(d^{2}+e^{2}\right)+d+e+\frac{3}{2},
\]
using the fact that $2de\leqslant d^{2}+e^{2}$. This last inequality
is equivalent to
\[
d^{2}-4d+e^{2}-4e\leqslant6
\]
and finally to
\[
(d-2)^{2}+(e-2)^{2}\leqslant14,
\]
which shows that $d,e\leqslant5$. This shows that the sets $\mathcal{E}_{M}$
are finite for $M\leqslant7$. To find the list of classes given above,
it suffices to compute the solutions to the Diophantine equations
of Proposition \ref{prop:characterization of E_M}\,(i) having $l(m)\leqslant7$
and $d,e\leqslant5$, and to show that they reduce to $(0,-1)$ by
Cremona moves, which is the case.\end{proof}
\begin{defn}
A class $(d,e;m)\in\mathcal{E}$ is said to be \emph{obstructive}
if there exists a rational number $a\geqslant1$ such that $\mu(d,e;m)(a)>\sqrt{\frac{a}{2}}$.\end{defn}
\begin{lem}
\label{lem:e=00003Dd or e=00003Dd-1}Let $(d,e;m)\in\mathcal{E}$
be an obstructive class. Then either $e=d$, or $e=d-1$.\end{lem}
\begin{proof}
Suppose by contradiction that there exists a class $(d,e;m)\in\mathcal{E}$
obstructive at some point $a>1$ such that $d=e+k$ with $k\geqslant2$.
Then, using Proposition \ref{prop:characterization of E_M}\,(i)
and Lemma \ref{lem:weight expansion}\,(ii), we obtain
\[
\begin{aligned}\sqrt{\frac{a}{2}} & <\mu(d,e;m)(a)=\frac{\left\langle m,w(a)\right\rangle }{d+e}\leqslant\frac{\left\Vert m\right\Vert \left\Vert w(a)\right\Vert }{d+e}=\frac{\sqrt{2de+1}\sqrt{a}}{d+e}\\
 & =\frac{\sqrt{2(e+k)e+1}\sqrt{a}}{2e+k}=\frac{\sqrt{4e^{2}+4ke+2}\sqrt{a}}{\sqrt{2}(2e+k)}<\frac{\sqrt{4e^{2}+4ke+k^{2}}\sqrt{a}}{\sqrt{2}(2e+k)}\\
 & =\sqrt{\frac{a}{2}},
\end{aligned}
\]
 which is a contradiction.\end{proof}
\begin{rem}
This lemma will be very useful in the sequel, because whenever we
will have to prove some properties of obstructive classes, it will
be sufficient to prove them for classes of the form $(d,d;m)$ or
$(d+\frac{1}{2},d-\frac{1}{2};m)$ only. Since this will happen many
times, we will not explicitly refer to this lemma each time.\end{rem}
\begin{defn}
We define the \emph{error vector} of a class $(d,e;m)$ at a point
$a$ as the vector $\varepsilon:=\varepsilon\left((d,e;m),a\right)$
defined by the equation
\[
m=\frac{d+e}{\sqrt{2a}}\, w(a)+\varepsilon.
\]
\end{defn}
\begin{lem}
\label{lem:first properties of mu}\renewcommand{\labelenumi}{(\roman{enumi})}
Let $a=\frac{p}{q}\geqslant1$ be a rational number with weight expansion $w(a)$ and let $(d,e;m)\in\mathcal{E}$. Then
\begin{enumerate}
\item $\mu(d,e;m)(a)\leqslant\frac{\sqrt{2de+1}\sqrt{a}}{d+e}$. In particular, \[ \mu(d,d;m)(a)\leqslant{\textstyle \sqrt{1+\frac{1}{2d^{2}}}\sqrt{\frac{a}{2}}}\quad\textrm{and}\quad\mu(d+{\textstyle \frac{1}{2}},d-{\textstyle \frac{1}{2}};m)(a)\leqslant{\textstyle \sqrt{1+\frac{1}{4d^{2}}}\sqrt{\frac{a}{2}}}, \]
\item $\mu(d,e;m)(a)>\sqrt{\frac{a}{2}}$ if and only if $\left\langle \varepsilon,w(a)\right\rangle >0$,
\item If $\mu(d,d;m)(a)>\sqrt{\frac{a}{2}}$ (resp. $\mu(d+\frac{1}{2},d-\frac{1}{2};m)(a)>\sqrt{\frac{a}{2}}$), then $\left\langle \varepsilon,\varepsilon\right\rangle <1$ (resp. $\left\langle \varepsilon,\varepsilon\right\rangle <\frac{1}{2}$),
\item $-\sum_{i=1}^{M}\varepsilon_{i}=\frac{d+e}{\sqrt{2a}}\left(y(a)-\frac{1}{q}\right)+1$, where $y(a):=a+1-2\sqrt{2a}$.
\end{enumerate}%
\end{lem}
\begin{proof}
(i) By the Cauchy-Schwarz inequality, Proposition \ref{prop:characterization of E_M}\,(i)
and Lemma~\ref{lem:weight expansion}, we have
\[
(d+e)\mu(d,e;m)(a)=\left\langle m,w(a)\right\rangle \leqslant\left\Vert m\right\Vert \left\Vert w(a)\right\Vert =\sqrt{2de+1}\sqrt{a}.
\]
In the case of a class $(d,d;m)$ we find that
\[
\mu(d,d;m)(a)\leqslant\frac{\sqrt{2d^{2}+1}\sqrt{a}}{2d}=\sqrt{\frac{2d^{2}+1}{2d^{2}}}\sqrt{\frac{a}{2}}=\sqrt{1+\frac{1}{2d^{2}}}\sqrt{\frac{a}{2}},
\]
and in the case of a class $(d+\frac{1}{2},d-\frac{1}{2};m)$ that
\[
\mu(d+\frac{1}{2},d-\frac{1}{2};m)(a)\leqslant\frac{\sqrt{2d^{2}+\frac{1}{2}}\sqrt{a}}{2d}=\sqrt{1+\frac{1}{4d^{2}}}\sqrt{\frac{a}{2}}.
\]

(ii) Since
\[
\begin{aligned}\left\langle \varepsilon,w(a)\right\rangle  & =\left\langle m-\frac{d+e}{\sqrt{2a}}w(a),w(a)\right\rangle =\left\langle m,w(a)\right\rangle -\frac{d+e}{\sqrt{2a}}\left\Vert w(a)\right\Vert ^{2}\\
 & =\left\langle m,w(a)\right\rangle -\left(d+e\right)\sqrt{\frac{a}{2}},
\end{aligned}
\]
 we see that $\left\langle \varepsilon,w(a)\right\rangle >0$ if and
only if $\mu(d,e;m)(a)=\frac{\left\langle m,w(a)\right\rangle }{d+e}>\sqrt{\frac{a}{2}}$.\medskip{}

(iii) For a class $(d,d;m)$,
\[
\begin{aligned}2d^{2}+1 & =\left\langle m,m\right\rangle =\left\langle \sqrt{\frac{2}{a}}\, dw(a)+\varepsilon,\sqrt{\frac{2}{a}}\, dw(a)+\varepsilon\right\rangle \\
 & =2d^{2}+\frac{2\sqrt{2}}{\sqrt{a}}\, d\left\langle w(a),\varepsilon\right\rangle +\left\langle \varepsilon,\varepsilon\right\rangle 
\end{aligned}
\]
shows that if $\mu(d,d;m)(a)>\sqrt{\frac{a}{2}}$, then by (ii) $\left\langle \varepsilon,\varepsilon\right\rangle <1$.
Similarly, for a class $(d+\frac{1}{2},d-\frac{1}{2};m)$,
\[
2d^{2}+\frac{1}{2}=2d^{2}+\frac{2\sqrt{2}}{\sqrt{a}}d\left\langle w(a),\varepsilon\right\rangle +\left\langle \varepsilon,\varepsilon\right\rangle 
\]
shows that if $\mu(d+\frac{1}{2},d-\frac{1}{2};m)(a)>\sqrt{\frac{a}{2}}$,
then $\left\langle \varepsilon,\varepsilon\right\rangle <\frac{1}{2}$.\medskip{}

(iv) By Proposition \ref{prop:characterization of E_M}\,(i) and
Lemma \ref{lem:weight expansion}, we see that
\[
2(d+e)-1=\sum_{i=1}^{M}m_{i}=\frac{d+e}{\sqrt{2a}}\left(a+1-\frac{1}{q}\right)+\sum_{i=1}^{M}\varepsilon_{i}.
\]
Thus
\[
-\sum_{i=1}^{M}\varepsilon_{i}=\frac{d+e}{\sqrt{2a}}\left(a+1-2\sqrt{2a}\right)-\frac{d+e}{q\sqrt{2a}}+1,
\]
from which the result follows.\end{proof}
\begin{cor}
\label{cor:set of mu st mu=00003Dc is finite}\renewcommand{\labelenumi}{(\roman{enumi})}
Suppose that $c(a)>\sqrt{\frac{a}{2}}$ for some rational $a\geqslant1$. Then
\begin{enumerate}
\item There exist classes $(d,e;m),(d',e';m')\in\mathcal{E}$ (possibly equal) and $\varepsilon>0$ such that \[ c(z)=\left\{ \begin{array}{ll} \mu(d,e;m) & \textrm{if }z\in\left]a-\varepsilon,a\right],\\ \mu(d',e';m') & \textrm{if }z\in\left[a,a+\varepsilon\right[. \end{array}\right. \]
\item The set of classes $(d,e;m)\in\mathcal{E}$ such that $\mu(d,e;m)(a)=c(a)$ is finite.
\item For each of the intervals of (i), there exist rational coefficients $\alpha,\beta\geqslant0$ such that $c(z)=\alpha+\beta z$.
\end{enumerate}%
\end{cor}
\begin{proof}
Since $c(a)>\sqrt{\frac{a}{2}}$, there exists $D\in\mathbb{N}$ such
that $c(a)>\sqrt{1+\frac{1}{D^{2}}}\sqrt{\frac{a}{2}}$. Since $c$
is continuous, there exists $\varepsilon>0$ such that $c(z)>\sqrt{1+\frac{1}{D^{2}}}\sqrt{\frac{z}{2}}$
for all $z\in\left]a-\varepsilon,a+\varepsilon\right[$. Now, if $\mu(d,e;m)(z)>\sqrt{1+\frac{1}{D^{2}}}\sqrt{\frac{z}{2}}$,
the inequalities of Lemma \ref{lem:first properties of mu}\,(i)
imply that $d\leqslant D$. There are thus only finitely many classes
with $\mu(d,e;m)(z)>\sqrt{1+\frac{1}{D^{2}}}\sqrt{\frac{z}{2}}$,
and for all $z\in\left]a-\varepsilon,a+\varepsilon\right[$, $c(z)$
is the supremum of $\mu(d,e;m)(z)$ taken over finitely many classes.
This proves (i) and~(ii). To prove (iii), notice first that the constraints
$\mu(d,e;m)$ are piecewise linear functions. Indeed, let $w(a)=\left(w_{1}(a),w_{2}(a),\ldots\right)$
be the weight expansion of $a$, where the $w_{i}$ are seen as functions
of $a$. Then the $w_{i}$ are piecewise linear functions and so is
$\mu(d,e;m)=\frac{\left\langle m,w\right\rangle }{d+e}$. We can thus
write $c(z)=\alpha+\beta z$ for $z$ belonging to one of the intervals
of (i). Now, since $c$ is nondecreasing, $\beta\geqslant0$, and
by the scaling property of Lemma \ref{lem:scaling property of c},
$\alpha\geqslant0$.\end{proof}
\begin{defn}
A class $(d,e;m)\in\mathcal{E}$ is called perfect if there exists
$b>1$ and $\kappa>0$ such that $m=\kappa\, w(b)$, that is, such
that the vector $m$ is a multiple of the weight expansion of $b$.\end{defn}
\begin{lem}
\label{lem:perfect elements give the constraint}Let $(d,e;m)\in\mathcal{E}$
be a perfect class for some $b>1$ with $d=e$ or $d=e+1$. Then $c(b)=\mu(d,e;m)(b)>\sqrt{\frac{b}{2}}$
and $(d,e;m)$ is the only class such that $\mu(d,e;m)(b)=c(b)$.\end{lem}
\begin{proof}
We first treat the case $d=e$. Let $(d,d;m)\in\mathcal{E}$ be a
perfect class: $m=\kappa\, w(b)$ for some $b>1$. By Proposition
\ref{prop:characterization of E_M}\,(i),
\[
2d^{2}<2d^{2}+1=\left\langle m,m\right\rangle =\kappa^{2}\left\langle w(b),w(b)\right\rangle =\kappa^{2}b,
\]
from which we deduce that $d<\kappa\sqrt{\frac{b}{2}}$. Then
\[
\mu(d,d;m)(b)=\frac{\left\langle m,w(b)\right\rangle }{2d}=\frac{\kappa b}{2d}>\sqrt{\frac{b}{2}}.
\]
This shows that $(d,d;m)$ is obstructive at $b$. But then $(d,d;m)$
is the only obstructive class at $b$. Indeed, if $(d',e';m')\in\mathcal{E}$
is a class different of $(d,d;m)$, by positivity of intersections
(Proposition \ref{prop:characterization of E_M}\,(ii)),
\[
\kappa\left\langle m',w(b)\right\rangle =\sum m_{i}m_{i}'\leqslant d(d'+e').
\]
Thus
\[
\mu(d',e';m')(b)=\frac{\left\langle m',w(b)\right\rangle }{d'+e'}\leqslant\frac{d\left\langle m',w(b)\right\rangle }{\kappa\left\langle m',w(b)\right\rangle }=\frac{d}{\kappa}<\sqrt{\frac{b}{2}}.
\]
Consider now a class of the form $\left(d+\frac{1}{2},d-\frac{1}{2};m\right)$.
By Proposition \ref{prop:characterization of E_M}\,(i), we have
\[
2d^{2}<2d^{2}+\frac{1}{2}=\left\langle m,m\right\rangle =\kappa^{2}\left\langle w(b),w(b)\right\rangle =\kappa^{2}b,
\]
and thus $d<\kappa\sqrt{\frac{b}{2}}$ as in the case of a class $(d,d;m)$.
The rest of the proof is then identical.\end{proof}
\begin{defn}
Define the \emph{length of a vector} $m$, denoted by $l(m)$, as
the number of positive entries in $m$, and denote by $l(a)$ the
\emph{length of the weight expansion} $w(a)$ of $a$.\end{defn}
\begin{lem}
\label{lem:there exists a unique a_0 center of the class}Let $(d,e;m)\in\mathcal{E}$
be an obstructive class. Let $I$ be a maximal nonempty open interval
on which $\mu(d,e;m)(a)>\sqrt{\frac{a}{2}}$. Then there exists a
unique $a_{0}\in I$ such that $l\left(a_{0}\right)=l(m)$. Moreover
for all $a\in I$, $l(a)\geqslant l\left(a_{0}\right)$.\end{lem}
\begin{proof}
Let us first prove that for all $a\in I$, $l(a)\geqslant l(m)$.
If $l(a)<l(m)$, then, by Proposition \ref{prop:characterization of E_M}\,(i),
\[
\sum_{i=1}^{l(a)}m_{i}^{2}<2de+1.
\]
Thus
\[
\mu(d,e;m)(a)=\frac{\left\langle m,w(a)\right\rangle }{d+e}\leqslant\frac{\sqrt{\sum_{i=1}^{l(a)}m_{i}^{2}}\left\Vert w(a)\right\Vert }{d+e}\leqslant\frac{\sqrt{4de}}{d+e}\sqrt{\frac{a}{2}}\leqslant\sqrt{\frac{a}{2}},
\]
and so $a\notin I$. Let us now prove the existence of an $a_{0}$
with $l\left(a_{0}\right)=l(m)$. Let $w(a)=\left(w_{1}(a),w_{2}(a),\ldots\right)$
be the weight expansion of $a$, where the $w_{i}$ are again seen
as functions of $a$. The $w_{i}$ are piecewise linear functions
and are linear on intervals that do not contain elements $a'$ with
$l\left(a'\right)\leqslant i$. Hence if all $a\in I$ would have
$l(a)>l(m)$, then the $l(m)$ first $w_{i}$ would be linear, and
$\mu(d,e;m)$ also. But this is impossible since $\sqrt{\frac{a}{2}}$
is concave. Thus there exists $a{}_{0}\in I$ with $l\left(a_{0}\right)=l(m)$.
The proof of uniqueness of $a_{0}$ follows from the fact that if
$a<b$ and $l(a)=l(b)$, then there exists $c\in\left]a,b\right[$
such that $l(c)<l(a)$.\end{proof}
\begin{lem}
\label{lem:parallel blocks}Let $(d,e;m)\in\mathcal{E}$ be such that
$\mu(d,e;m)(a)>\sqrt{\frac{a}{2}}$. Let $J=\left\{ k,...,k+s-1\right\} $
be a block of $s\geqslant2$ consecutive integers such that the $w_{i}(a)$
are equal for all $i\in J$. Then we have the three following possibilities
\[
\begin{array}{ll}
1. & m_{k}=\ldots=m_{k+s-1}\\
2. & m_{k}-1=m_{k+1}=\ldots=m_{k+s-1}\\
3. & m_{k}=\ldots=m_{k+s-2}=m_{k+s-1}+1.
\end{array}
\]
Moreover, there is at most one block of length $s\geqslant2$ where
the $m_{i}$ are not all equal, and if such a block $J$ exists, then
$\sum_{i\in J}\varepsilon_{i}^{2}\geqslant\frac{s-1}{s}$.
\end{lem}
The proof is similar to the proof of Lemma 2.1.7 in \cite{MS}.
\begin{cor}
\label{cor:an obstructive class (d,d-1;m) has the same blcoks}If
a class of the form $(d+\frac{1}{2},d-\frac{1}{2};m)$ is obstructive,
then the $m_{i}$ are constant on each block.\end{cor}
\begin{proof}
Suppose there exists a block $J$ of length $s\geqslant2$ on which
the $w_{i}(a)$ are not all equal. Then by Lemma \ref{lem:parallel blocks},
$\sum_{i\in J}\varepsilon_{i}^{2}\geqslant\frac{s-1}{s}\geqslant\frac{1}{2}$,
which contradicts Lemma~\ref{lem:first properties of mu}\,(iii),
which states that $\sum\varepsilon_{i}^{2}<\frac{1}{2}$.\end{proof}
\begin{lem}
\label{lem:parallel blocks 2}\renewcommand{\labelenumi}{(\roman{enumi})}
Let $(d,e;m)\in\mathcal{E}$ be an obstructive class at some $a\in\mathbb{Q}$ with $l(a)=l(m)$. Let $w_{k+1},\ldots,w_{k+s}$ be a block which is not the first block of $w(a)$.
\begin{enumerate}
\item If the block is not the last one, then \[ \left|m_{k}-\left(m_{k+1}+\ldots+m_{k+s+1}\right)\right|<\sqrt{s+2}. \] If the block is the last one, then \[ \left|m_{k}-\left(m_{k+1}+\ldots+m_{k+s}\right)\right|<\sqrt{s+1}. \]
\item It is always the case that \[ m_{k}-\sum_{i=k+1}^{M}m_{i}<\sqrt{M-k+1}. \]
\end{enumerate} %

\end{lem}
The proof is similar to the one of Lemma 2.1.8 in \cite{MS}.
\begin{prop}
\label{prop:structure of the obstructions mu}Let $(d,e;m)\in\mathcal{E}$
be an obstructive class at a point $a=:\frac{p}{q}\in\mathbb{Q}$
written in lowest terms with $l(a)=l(m)$. Let $m_{M}$ be the last
nonzero entry of the vector $m$ and let $I$ be the maximal open
interval containing $a$ such that $\mu(d,e;m)(a)>\sqrt{\frac{a}{2}}$.
Then there exist integers $A<p$ and $B<(m_{M}+1)q$ such that
\[
(d+e)\mu(d,e;m)(z)=\left\{ \begin{array}{cl}
A+Bz & \textrm{if }z\leqslant a,z\in I,\\
(A+m_{M}p)+(B-m_{M}q)z & \textrm{if }z\geqslant a,z\in I.
\end{array}\right.
\]

\end{prop}
Again, the proof is similar to the one of Proposition 2.3.2 in \cite{MS}.

\section{The interval $\left[1,\sigma^{2}\right]$}

\noindent The goal of this section is to prove part (i) of Theorem
\ref{thm:statement of the result}.

\subsection{Preliminaries}

Let us first recall that the \emph{Pell numbers} $P_{n}$ and the
\emph{half companion Pell numbers}~$H_{n}$ are defined by the recurrence
relations
\[
\begin{alignedat}{5}P_{0} & = & \:0,\; & P_{1} & = & \:1,\quad & P_{n} & = & 2P_{n-1}+P_{n-2},\\
H_{0} & = & \:1,\; & H_{1} & = & \:1,\quad & H_{n} & = & \:2H_{n-1}+H_{n-2},
\end{alignedat}
\]
respectively. It is then easy to see that
\[
H_{n}=P_{n}+P_{n-1}.
\]
Using this, we define the sequence $\left(\alpha_{n}\right)_{n\geqslant0}$
by
\[
\alpha_{n}:=\left\{ \begin{aligned}\frac{2P_{n+1}^{2}}{H_{n}^{2}}=:\frac{p_{n}}{q_{n}} & \qquad\textrm{if }n\textrm{ is even},\\
\frac{H_{n+1}^{2}}{2P_{n}^{2}}=:\frac{p_{n}}{q_{n}} & \qquad\textrm{if }n\textrm{ is odd}.
\end{aligned}
\right.
\]
Set $W'\left(\alpha_{n}\right)=q_{n}w\left(\alpha_{n}\right)$ with
adding an extra $1$ at the end. Define the classes $E\left(\alpha_{n}\right)$
by
\[
E\left(\alpha_{n}\right):=\left\{ \begin{aligned}\left(P_{n+1}H_{n},P_{n+1}H_{n};W'\left(\alpha_{n}\right)\right) & \qquad\textrm{if }n\textrm{ is even},\\
\left(P_{n}H_{n+1},P_{n}H_{n+1};W'\left(\alpha_{n}\right)\right) & \qquad\textrm{if }n\textrm{ is odd}.
\end{aligned}
\right.
\]
For instance,
\[
\begin{aligned}E\left(\alpha_{0}\right) & =\left(1,1;1^{\times3}\right),\\
E\left(\alpha_{1}\right) & =\left(3,3;2^{\times4},1^{\times3}\right),\\
E\left(\alpha_{2}\right) & =\left(15,15;9^{\times5},5,4,1^{\times5}\right),\\
E\left(\alpha_{3}\right) & =\left(85,85;50^{\times5},39,11^{\times3},6,5,1^{\times6}\right).
\end{aligned}
\]
Moreover, we define the sequence $\left(\beta_{n}\right)_{n\geqslant0}$
by
\[
\beta_{n}:=\left\{ \begin{aligned}\frac{H_{n+2}}{H_{n}}=:\frac{p_{n}}{q_{n}} & \qquad\textrm{if }n\textrm{ is even},\\
\frac{P_{n+2}}{P_{n}}=:\frac{p_{n}}{q_{n}} & \qquad\textrm{if }n\textrm{ is odd}.
\end{aligned}
\right.
\]
Set $W\left(\beta_{n}\right)=q_{n}w\left(\beta_{n}\right)$. Then
the classes $E\left(\beta_{n}\right)$ are defined by
\[
E\left(\beta_{n}\right):=\left\{ \begin{aligned}\left(\frac{1}{4}\left(H_{n}+H_{n+2}\right),\frac{1}{4}\left(H_{n}+H_{n+2}\right);W\left(\beta_{n}\right)\right) & \qquad\textrm{if }n\textrm{ is even},\\
\left(\frac{1}{4}\left(P_{n}+P_{n+2}\right)+\frac{1}{2},\frac{1}{4}\left(P_{n}+P_{n+2}\right)-\frac{1}{2};W\left(\beta_{n}\right)\right) & \qquad\textrm{if }n\textrm{ is odd}.
\end{aligned}
\right.
\]
For instance,
\[
\begin{aligned}E\left(\beta_{0}\right) & =\left(1,1;1^{\times3}\right),\\
E\left(\beta_{1}\right) & =\left(2,1;1^{\times5}\right),\\
E\left(\beta_{2}\right) & =\left(5,5;3^{\times5},2,1^{\times2}\right),\\
E\left(\beta_{3}\right) & =\left(9,8;5^{\times5},4,1^{\times4}\right).
\end{aligned}
\]

\begin{thm}
\label{thm:the classes E(alpha_n), E(beta_n) belong to E}For all
$n\geqslant0$, $E\left(\alpha_{n}\right),E\left(\beta_{n}\right)\in\mathcal{E}$.
\end{thm}
The proof that $E\left(\alpha_{n}\right)\in\mathcal{E}$ is given
in the next subsection, while the proof that $E\left(\beta_{n}\right)\in\mathcal{E}$
is given in Corollary \ref{cor:the classes E(beta_n) belong to E}.
Theorem \ref{thm:the classes E(alpha_n), E(beta_n) belong to E} implies
part (i) of Theorem \ref{thm:statement of the result}:
\begin{cor}
\label{cor:c(a) on [1,sigma^2]}On the interval $\left[1,\sigma^{2}\right]$,
\[
c(a)=\left\{ \begin{aligned}1\quad\; & \qquad\textrm{if }a\in\left[1,2\right],\\
\frac{1}{\sqrt{2\alpha_{n}}}\, a & \qquad\textrm{if }a\in\left[\alpha_{n},\beta_{n}\right],\\
\sqrt{\frac{\alpha_{n+1}}{2}} & \qquad\textrm{if }a\in\left[\beta_{n},\alpha_{n+1}\right],
\end{aligned}
\right.
\]
for all $n\geqslant0$.\end{cor}
\begin{proof}
Since for all $n\geqslant0$, $E\left(\beta_{n}\right)$ is a perfect
class, we know by Lemma \ref{lem:perfect elements give the constraint}
that $c\left(\beta_{n}\right)=\mu\left(E\left(\beta_{n}\right)\right)\left(\beta_{n}\right)$.
Hence
\[
c\left(\beta_{n}\right)=\sqrt{\frac{\alpha_{n+1}}{2}}
\]
for all $n\geqslant0$. Indeed, for $n$ even, we compute
\[
\begin{aligned}c\left(\beta_{n}\right) & =\frac{2H_{n}\left\langle w\left(\beta_{n}\right),w\left(\beta_{n}\right)\right\rangle }{H_{n+2}+H_{n}}=\frac{2H_{n}\beta_{n}}{H_{n+2}+H_{n}}=\frac{2H_{n+2}}{H_{n+2}+H_{n}}\\
 & =\frac{2\left(P_{n+2}+P_{n+1}\right)}{P_{n+2}+P_{n+1}+P_{n}+P_{n-1}}=\frac{2\left(P_{n+2}+P_{n+1}\right)}{4P_{n+1}}=\frac{H_{n+2}}{2P_{n+1}}=\sqrt{\frac{\alpha_{n+1}}{2}},
\end{aligned}
\]
and for $n$ odd,
\[
\begin{aligned}c\left(\beta_{n}\right) & =\frac{2P_{n}\left\langle w\left(\beta_{n}\right),w\left(\beta_{n}\right)\right\rangle }{P_{n+2}+P_{n}}=\frac{2P_{n}\beta_{n}}{P_{n+2}+P_{n}}=\frac{2P_{n+2}}{P_{n+2}+P_{n}}\\
 & =\frac{P_{n+2}}{\frac{1}{2}\left(P_{n+2}-P_{n}\right)+P_{n}}=\frac{P_{n+2}}{P_{n+1}+P_{n}}=\frac{P_{n+2}}{H_{n}}=\sqrt{\frac{\alpha_{n+1}}{2}}.
\end{aligned}
\]
Furthermore, $c\left(\alpha_{n}\right)=\sqrt{\frac{\alpha_{n}}{2}}$
for all $n\geqslant0$. Indeed, for $n$ even, we have
\[
\mu\left(E\left(\alpha_{n}\right)\right)\left(\alpha_{n}\right)=\frac{H_{n}^{2}\alpha_{n}}{2P_{n+1}H_{n}}=\frac{H_{n}\alpha_{n}}{2P_{n+1}}=\frac{P_{n+1}}{H_{n}}=\sqrt{\frac{\alpha_{n}}{2}}.
\]
Thus for all $(d,e;m)\in\mathcal{E}$ distinct from $E\left(\alpha_{n}\right)$,
we get by Proposition \ref{prop:characterization of E_M}\,(ii) that
\[
P_{n+1}H_{n}(d+e)\geqslant\left\langle m,W'\left(\alpha_{n}\right)\right\rangle \geqslant H_{n}^{2}\left\langle m,w\left(\alpha_{n}\right)\right\rangle ,
\]
and hence
\[
\mu(d,e;m)\left(\alpha_{n}\right)=\frac{\left\langle m,w\left(\alpha_{n}\right)\right\rangle }{d+e}\leqslant\frac{P_{n+1}H_{n}}{H_{n}^{2}}=\frac{P_{n+1}}{H_{n}}=\sqrt{\frac{\alpha_{n}}{2}}.
\]
Next, for $n$ odd, we have
\[
\mu\left(E\left(\alpha_{n}\right)\right)\left(\alpha_{n}\right)=\frac{2P_{n}^{2}\alpha_{n}}{2P_{n}H_{n+1}}=\frac{P_{n}\alpha_{n}}{H_{n+1}}=\frac{H_{n+1}}{2P_{n}}=\sqrt{\frac{\alpha_{n}}{2}}.
\]
Thus for all $(d,e;m)\in\mathcal{E}$ distinct from $E\left(\alpha_{n}\right)$,
we get
\[
P_{n}H_{n+1}(d+e)\geqslant\left\langle m,W'\left(\alpha_{n}\right)\right\rangle \geqslant2P_{n}^{2}\left\langle m,w\left(\alpha_{n}\right)\right\rangle ,
\]
and hence
\[
\mu(d,e;m)\left(\alpha_{n}\right)=\frac{\left\langle m,w\left(\alpha_{n}\right)\right\rangle }{d+e}\leqslant\frac{P_{n}H_{n+1}}{2P_{n}^{2}}=\frac{H_{n+1}}{2P_{n}}=\sqrt{\frac{\alpha_{n}}{2}}.
\]
Thus, by Proposition \ref{prop:characterization of c(a)}, we have
$c\left(\alpha_{n}\right)=\sqrt{\frac{\alpha_{n}}{2}}$ for all $n\geqslant0$
as required. Since $c$ is nondecreasing, we get that $c(a)=\sqrt{\frac{\alpha_{n+1}}{2}}$
for $a\in\left[\beta_{n},\alpha_{n+1}\right]$. Moreover, we have
that for all $n\geqslant0$
\[
\frac{c\left(\beta_{n}\right)}{\beta_{n}}=\frac{c\left(\alpha_{n}\right)}{\alpha_{n}}.
\]
Indeed, for $n$ even,
\[
\frac{c\left(\beta_{n}\right)}{\beta_{n}}=\frac{\sqrt{\frac{\alpha_{n+1}}{2}}}{\beta_{n}}=\frac{H_{n}}{2P_{n+1}}=\frac{\sqrt{\frac{\alpha_{n}}{2}}}{\alpha_{n}}=\frac{c\left(\alpha_{n}\right)}{\alpha_{n}},
\]
and for $n$ odd,
\[
\frac{c\left(\beta_{n}\right)}{\beta_{n}}=\frac{\sqrt{\frac{\alpha_{n+1}}{2}}}{\beta_{n}}=\frac{P_{n}}{H_{n+1}}=\frac{\sqrt{\frac{\alpha_{n}}{2}}}{\alpha_{n}}=\frac{c\left(\alpha_{n}\right)}{\alpha_{n}}.
\]
Hence, by the scaling property of Lemma \ref{lem:scaling property of c},
the function $c$ has to be linear on $\left[\alpha_{n},\beta_{n}\right]$
and thus $c(a)=\frac{1}{\sqrt{2\alpha_{n}}}a$ for $a\in\left[\alpha_{n},\beta_{n}\right]$.
\end{proof}

\subsection{The classes $E\left(\alpha_{n}\right)$ belong to $\mathcal{E}$}
\begin{lem}
The classes $E\left(\alpha_{n}\right)$ satisfy the Diophantine conditions
of Proposition \ref{prop:characterization of E_M}\,(i).\end{lem}
\begin{proof}
We will prove this separately for $n$ even and odd. In both cases,
we will use Lemma \ref{lem:weight expansion} and the relation
\[
-P_{2m}^{2}+2P_{2m}P_{2m-1}+P_{2m-1}^{2}=1
\]
which can be easily deduced from the following identity 
\[
P_{2m-k}=(-1)^{k+1}\left(P_{k}H_{2m}-H_{k}P_{2m}\right),
\]
given in Corollary \ref{cor:properties of u_k(j), v_k(j)}\,(v).
For $n=2m$, we obtain
\[
\begin{aligned}\sum m_{i} & =H_{2m}^{2}\sum w_{i}+1=H_{2m}^{2}\left(\frac{2P_{2m+1}^{2}}{H_{2m}^{2}}+1-\frac{1}{H_{2m}^{2}}\right)+1\\
 & =2\left(2P_{2m}+P_{2m-1}\right)^{2}+\left(P_{2m}+P_{2m-1}\right)^{2}\\
 & =9P_{2m}^{2}+10P_{2m}P_{2m-1}+3P_{2m-1}^{2}\\
 & =8P_{2m}^{2}+12P_{2m}P_{2m-1}+4P_{2m-1}^{2}-\left(-P_{2m}^{2}+2P_{2m}P_{2m-1}+P_{2m-1}^{2}\right)\\
 & =4P_{2m+1}H_{2m}-1=2(d+e)-1;\\
\\
\sum m_{i}^{2} & =H_{2m}^{4}\sum w_{i}^{2}+1=H_{2m}^{4}\frac{2P_{2m+1}^{2}}{H_{2m}^{2}}+1\\
 & =2P_{2m+1}^{2}H_{2m}^{2}+1=2de+1.
\end{aligned}
\]
Moreover, for $n=2m-1$,
\[
\begin{aligned}\sum m_{i} & =2P_{2m-1}^{2}\sum w_{i}+1=2P_{2m-1}^{2}\left(\frac{H_{2m}^{2}}{2P_{2m-1}^{2}}+1-\frac{1}{2P_{2m-1}^{2}}\right)+1\\
 & =\left(P_{2m}+P_{2m-1}\right)^{2}+2P_{2m-1}^{2}\\
 & =P_{2m}^{2}+2P_{2m}P_{2m-1}+3P_{2m-1}^{2}\\
 & =4P_{2m}P_{2m-1}+4P_{2m-1}^{2}-\left(-P_{2m}^{2}+2P_{2m}P_{2m-1}+P_{2m-1}^{2}\right)\\
 & =4P_{2m-1}H_{2m}-1=2(d+e)-1;\\
\\
\sum m_{i}^{2} & =4P_{2m-1}^{4}\sum w_{i}^{2}+1=4P_{2m-1}^{4}\frac{H_{2m}^{2}}{2P_{2m-1}^{2}}+1\\
 & =2P_{2m-1}^{2}H_{2m}^{2}+1=2de+1.
\end{aligned}
\]
This proves the lemma.
\end{proof}
We will now prove separately for $n$ even and $n$ odd that the classes
$E\left(\alpha_{n}\right)$ reduce to $(0;-1)$ by standard Cremona
moves.

\subsubsection{The classes $E\left(\alpha_{2m}\right)$ reduce to $(0;-1)$}

One readily checks that the classes $E\left(\alpha_{2m}\right)$ reduce
to $(0;-1)$ for $m=0,1,2$. In the following, we reduce the classes
$E\left(\alpha_{2m}\right)$ for $m\geqslant3$.
\begin{lem}
The continued fraction expansion of $\alpha_{2m}$ is
\[
\left[5;\left\{ 1,4\right\} ^{\times(m-1)},1,1,3,1,\left\{ 4,1\right\} ^{\times(m-1)}\right].
\]
Moreover, with $u_{j}:=\left(2H_{j}-P_{j}\right)H_{2m}+H_{j}P_{2m}$,
\[
\begin{aligned}E\left(\alpha_{2m}\right)= & \left(P_{2m+1}H_{2m},P_{2m+1}H_{2m};\right.\\
 & \left(\frac{1}{2}u_{2m}\right)^{\times5},u_{2m-1},\left(\frac{1}{2}u_{2m-2}\right)^{\times4},\ldots,u_{3},\left(\frac{1}{2}u_{2}\right)^{\times4},\\
 & u_{1},H_{2m}+\frac{1}{2}P_{2m},\left(\frac{1}{2}P_{2m}\right)^{\times3},P_{2m-1},\\
 & \left.\left(\frac{1}{2}P_{2m-2}\right)^{\times4},P_{2m-3},\ldots,\left(\frac{1}{2}P_{2}\right)^{\times4},P_{1},1\right).
\end{aligned}
\]
 \end{lem}
\begin{proof}
Since $(\alpha_{n})$ is an increasing sequence converging to $\sigma^{2}<6$
and $\alpha_{2}=\frac{50}{9}>5$, the first term of $W'\left(\alpha_{2m}\right)$
is 
\[
\left(q_{2m}\right)^{\times5}=\left(H_{2m}^{2}\right)^{\times5}=\left(\frac{1}{2}u_{2m}\right)^{\times5}.
\]
To determine the next terms, we will first prove that for all $k\geqslant1$,
$u_{2k+1}>\frac{1}{2}u_{2k}>u_{2k-1}$. Indeed
\[
\begin{aligned}u_{2k+1} & =\left(2H_{2k+1}-P_{2k+1}\right)H_{2m}+H_{2k+1}P_{2m}\\
 & =\left(4P_{2k}+P_{2k-1}\right)H_{2m}+\left(2H_{2k}+H_{2k-1}\right)P_{2m}\\
 & >\left(\frac{1}{2}P_{2k}+P_{2k-1}\right)H_{2m}+\frac{1}{2}H_{2k}P_{2m}\\
 & =\frac{1}{2}\left(2H_{2k}-P_{2k}\right)H_{2m}+\frac{1}{2}H_{2k}P_{2m}\\
 & =\frac{1}{2}u_{2k}\\
 & =\left(P_{2k-1}+\frac{5}{2}P_{2k-2}+P_{2k-3}\right)H_{2m}+\left(H_{2k-1}+\frac{1}{2}H_{2k-2}\right)P_{2m}\\
 & >\left(P_{2k-1}+2P_{2k-2}\right)H_{2m}+H_{2k-1}P_{2m}\\
 & =\left(2H_{2k-1}-P_{2k-1}\right)H_{2m}+H_{2k-1}P_{2m}\\
 & =u_{2k-1}.
\end{aligned}
\]
The second term of $W'\left(\alpha_{2m}\right)$ is
\[
\begin{aligned}2P_{2m+1}^{2}-5H_{2m}^{2} & =3P_{2m}^{2}-2P_{2m}P_{2m-1}-3P_{2m-1}^{2}\\
 & =\left(2H_{2m-1}-P_{2m-1}\right)H_{2m}+H_{2m-1}P_{2m}\\
 & =u_{2m-1}<\frac{1}{2}u_{2m}.
\end{aligned}
\]
Now for all $k\geqslant1$, we have
\[
\begin{aligned}u_{2k+1}-4\left(\frac{1}{2}u_{2k}\right) & =\left(2H_{2k+1}-P_{2k+1}-4H_{2k}+2P_{2k}\right)H_{2m}\\
 & \qquad+\left(H_{2k+1}-2H_{2k}\right)P_{2m}\\
 & =\left(2H_{2k-1}-P_{2k-1}\right)H_{2m}+H_{2k-1}P_{2m}=u_{2k-1}<\frac{1}{2}u_{2k};\\
\\
\frac{1}{2}u_{2k}-u_{2k-1} & =\left(H_{2k}-\frac{1}{2}P_{2k}-2H_{2k-1}+P_{2k-1}\right)H_{2m}\\
 & \qquad+\left(\frac{1}{2}H_{2k}-H_{2k-1}\right)P_{2m}\\
 & =\frac{1}{2}\left(2H_{2k-2}-P_{2k-2}\right)H_{2m}+\frac{1}{2}H_{2k-2}P_{2m}=\frac{1}{2}u_{2k-2}\\
 & <u_{2k-1}.
\end{aligned}
\]
Thus the first terms of $W'\left(\alpha_{2m}\right)$ are
\[
\left(\frac{1}{2}u_{2m}\right)^{\times5},u_{2m-1},\left(\frac{1}{2}u_{2m-2}\right)^{\times4},\ldots,u_{3},\left(\frac{1}{2}u_{2}\right)^{\times4},u_{1}.
\]
The next terms are $H_{2m}+\frac{1}{2}P_{2m},\left(\frac{1}{2}P_{2m}\right)^{\times3},P_{2m-1}$.
Indeed
\[
\begin{aligned}\frac{1}{2}u_{2}-u_{1} & =H_{2m}+\frac{1}{2}P_{2m}<H_{2m}+P_{2m}=u_{1};\\
\\
u_{1}-\left(H_{2m}+\frac{1}{2}P_{2m}\right) & =\frac{1}{2}P_{2m}<H_{2m}+\frac{1}{2}P_{2m};\\
\\
H_{2m}+\frac{1}{2}P_{2m}-3\left(\frac{1}{2}P_{2m}\right) & =P_{2m-1}<\frac{1}{2}P_{2m}.
\end{aligned}
\]
Notice that for all $k\geqslant1$,
\[
P_{k+1}>\frac{1}{2}P_{k}>P_{k-1},
\]
and thus
\[
\begin{aligned}P_{2k+1}-4\left(\frac{1}{2}P_{2k}\right) & =P_{2k-1}<\frac{1}{2}P_{2k};\\
\\
\frac{1}{2}P_{2k}-P_{2k-1} & =\frac{1}{2}P_{2k-2}<P_{2k-1}.
\end{aligned}
\]
This proves that the last terms of $W'\left(\alpha_{2m}\right)$ are
$\left(\frac{1}{2}P_{2m-2}\right)^{\times4},P_{2m-3},\ldots,\left(\frac{1}{2}P_{2}\right)^{\times4},P_{1},1$,
with the last $1$ added by definition of $W'\left(\alpha_{2m}\right)$.
\end{proof}
Let us introduce now some notations in order to simplify the expressions
of the classes.
\begin{defn}
Set
\[
\begin{aligned}A_{k}^{m}:= & \left(\left(\frac{1}{2}u_{2k}\right)^{\times4},u_{2k-1},\left(P_{2k}H_{2m}\right)^{\times2},\left(\frac{1}{2}u_{2k-2}\right)^{\times4},u_{2k-3},\ldots,\left(\frac{1}{2}u_{2}\right)^{\times4},u_{1},\right.\\
 & \qquad\left.H_{2m}+\frac{1}{2}P_{2m},\left(\frac{1}{2}P_{2m}\right)^{\times3},P_{2m-1}\right),\\
B_{k}^{m}:= & \left(\left(\frac{1}{2}P_{2m-2}\right)^{\times4},P_{2m-3},\ldots,\left(\frac{1}{2}P_{2m-2k+2}\right)^{\times4},P_{2m-2k+1},\right.\\
 & \qquad\left.\left(\frac{1}{2}P_{2m-2k}\right)^{\times8},\left(P_{2m-2k-1}\right)^{\times2},\ldots,\left(\frac{1}{2}P_{2}\right)^{\times8},\left(P_{1}\right)^{\times2},1\right),\\
V_{k}^{m}:= & \left(P_{2k+1}H_{2m}+H_{2k}P_{2m};A_{k}^{m},B_{k}^{m}\right).
\end{aligned}
\]
Thus $A_{k}^{m}$ has the structure $\left[4,1,2,\left\{ 4,1\right\} ^{\times(k-1)},1,3,1\right]$
and $B_{k}^{m}$ has the structure $\left[\left\{ 4,1\right\} ^{\times(k-1)},\left\{ 8,2\right\} ^{\times(m-k)},1\right]$.
We use here the convention that if $k=m$, $B_{m}^{m}$ has the structure
$\left[\left\{ 4,1\right\} ^{\times(m-1)},1\right]$ and that if $k=1$,
$B_{1}^{m}$ has the structure $\left[\left\{ 8,2\right\} ^{\times(m-1)},1\right]$.
\end{defn}
The structure of the reduction process of a class $E\left(\alpha_{2m}\right)$
will be the following. First, we compute in Lemma \ref{lem:1st_red_of_E(alpha_2m)}
that the image of $E\left(\alpha_{2m}\right)$ under $\varphi_{*}$
is $V_{m}^{m}=\left(P_{2m+1}H_{2m}+H_{2m}P_{2m};A_{m}^{m},B_{m}^{m}\right)$.
Then, we reduce $V_{m}^{m}$ in Lemma \ref{lem:2nd_red_of_E(alpha_2m)}
and Lemma \ref{lem:3rd_red_of_E(alpha_2m)} to $V^{m}:=\left(H_{2m};H_{2m}-\frac{1}{2}P_{2m},\left(\frac{1}{2}P_{2m}\right)^{\times3},\left(P_{2m-1}\right)^{\times2},B_{1}^{m}\right)$
in $4(m-2)+8$ Cremona moves. Finally, we show in Lemma \ref{lem:4th_red_of_E(alpha_2m)}
and Lemma \ref{lem:5th_red_of_E(alpha_2m)} that $V^{m}$ reduces
to $(0;-1)$ in $5(m-2)+8$ moves.
\begin{rem}
Since the Cremona transform of a class $(d;m)$ only modifies the
first $3$ entries of the vector $m$, we will write some of the first
entries of the classes and will abbreviate the other terms by $(*)$.
It is important to notice that the terms denoted by $(*)$ will always
be left invariant during the reduction process. Then, each time after
applying the Cremona transform, we will reorder the entries of $m$.
We will not always reorder the entries in decreasing order, but this
will have no consequence on the reduction because we will reorder
them in a way such that the first $3$ entries of the vector $m$
will always be the $3$ biggest entries in decreasing order. We will
write down each step of the reduction, that is the class obtained
after applying the Cremona transform and reordering. But sometimes
when the reordering will not be obvious, we will write the intermediate
step before reordering and denote by $\rightarrow$ the Cremona transform
of a class and by $\rightsquigarrow$ the reordering of a class. We
will freely use the three relations
\[
P_{n}=2P_{n-1}+P_{n-2},\quad H_{n}=P_{n}+P_{n-1},
\]
\[
P_{2m-k}=(-1)^{k+1}\left(P_{k}H_{2m}-H_{k}P_{2m}\right).
\]
\end{rem}
\begin{lem}
\label{lem:1st_red_of_E(alpha_2m)}The image of $E\left(\alpha_{2m}\right)$
by $\varphi_{*}$ is the class 
\[
\varphi_{*}\left(E\left(\alpha_{2m}\right)\right)=\left(P_{2m+1}H_{2m}+H_{2m}P_{2m};A_{m}^{m},B_{m}^{m}\right)=V_{m}^{m}.
\]
\end{lem}
\begin{proof}
The first terms of $E\left(\alpha_{2m}\right)$ are $\left(P_{2m+1}H_{2m},P_{2m+1}H_{2m};\frac{1}{2}u_{2m},(*)\right)$.
Since $\frac{1}{2}u_{2m}=H_{2m}^{2}$, we get
\[
\begin{aligned}\varphi_{*}\left(E\left(\alpha_{2m}\right)\right) & =\left(2P_{2m+1}H_{2m}-H_{2m}^{2};\left(P_{2m+1}H_{2m}-H_{2m}^{2}\right)^{\times2},(*)\right)\\
 & =\left(P_{2m+1}H_{2m}+H_{2m}P_{2m};\left(P_{2m}H_{2m}\right)^{\times2},(*)\right).
\end{aligned}
\]
After reordering this class, we see that we indeed obtain $V_{m}^{m}$
as required.\end{proof}
\begin{lem}
\label{lem:2nd_red_of_E(alpha_2m)}For all $3\leqslant k\leqslant m$,
$V_{k}^{m}$ reduces to $V_{k-1}^{m}$ in 4 Cremona moves.\end{lem}
\begin{proof}
We have\medskip{}

$\begin{aligned}V_{k}^{m}= & \left(P_{2k+1}H_{2m}+H_{2k}P_{2m};\frac{1}{2}\left(\left(2H_{2k}-P_{2k}\right)H_{2m}+H_{2k}P_{2m}\right)^{\times4},\right.\\
 & \left.\left(2H_{2k-1}-P_{2k-1}\right)H_{2m}+H_{2k-1}P_{2m},\left(P_{2k}H_{2m}\right)^{\times2},(*)\right)\quad\rightarrow
\end{aligned}
$

\medskip{}

$\left(\left(-H_{2k}+\frac{7}{2}P_{2k}\right)H_{2m}+\frac{1}{2}H_{2k}P_{2m};\left(H_{2k-1}H_{2m}\right)^{\times3},\right.$

$\hphantom{\qquad}\frac{1}{2}\left(\left(2H_{2k}-P_{2k}\right)H_{2m}+H_{2k}P_{2m}\right),\left(2H_{2k-1}-P_{2k-1}\right)H_{2m}+H_{2k-1}P_{2m},$

$\hphantom{\qquad}\left.\left(P_{2k}H_{2m}\right)^{\times2},(*)\right)\quad\rightsquigarrow$

\medskip{}

$\left(\left(-H_{2k}+\frac{7}{2}P_{2k}\right)H_{2m}+\frac{1}{2}H_{2k}P_{2m};\frac{1}{2}\left(\left(2H_{2k}-P_{2k}\right)H_{2m}+H_{2k}P_{2m}\right),\right.$

$\hphantom{\qquad}\left(2H_{2k-1}-P_{2k-1}\right)H_{2m}+H_{2k-1}P_{2m},\left(P_{2k}H_{2m}\right)^{\times2},$

$\hphantom{\qquad}\left.\left(H_{2k-1}H_{2m}\right)^{\times3},(*)\right)\quad\rightarrow$

\medskip{}

$\left(\frac{3}{2}P_{2k}H_{2m}+\frac{1}{2}H_{2k-2}P_{2m};\left(2H_{2k}-\frac{5}{2}P_{2k}\right)H_{2m}+\frac{1}{2}H_{2k-2}P_{2m},P_{2k-2}H_{2m},\right.$

$\hphantom{\qquad}\left.P_{2k-1}H_{2m}-H_{2k-1}P_{2m},P_{2k}H_{2m},\left(H_{2k-1}H_{2m}\right)^{\times3},(*)\right)\quad\rightsquigarrow$

\medskip{}

$\left(\frac{3}{2}P_{2k}H_{2m}+\frac{1}{2}H_{2k-2}P_{2m};P_{2k}H_{2m},\left(H_{2k-1}H_{2m}\right)^{\times3},\right.$

$\hphantom{\qquad}\left.\left(2H_{2k}-\frac{5}{2}P_{2k}\right)H_{2m}+\frac{1}{2}H_{2k-2}P_{2m},P_{2k-2}H_{2m},P_{2m-(2k-1)},(*)\right)\quad\rightarrow$

\medskip{}

$\left(2P_{2k-1}H_{2m}+H_{2k-2}P_{2m};\left(2H_{2k}-\frac{5}{2}P_{2k}\right)H_{2m}+\frac{1}{2}H_{2k-2}P_{2m},\right.$

$\hphantom{\qquad}\left(-\frac{1}{2}P_{2k-2}H_{2m}+\frac{1}{2}H_{2k-2}P_{2m}\right)^{\times2},H_{2k-1}H_{2m},$

$\hphantom{\qquad}\left.\left(2H_{2k}-\frac{5}{2}P_{2k}\right)H_{2m}+\frac{1}{2}H_{2k-2}P_{2m},P_{2k-2}H_{2m},P_{2m-(2k-1)},(*)\right)\quad\rightsquigarrow$

\medskip{}

$\left(2P_{2k-1}H_{2m}+H_{2k-2}P_{2m};H_{2k-1}H_{2m},\left(\left(2H_{2k}-\frac{5}{2}P_{2k}\right)H_{2m}+\frac{1}{2}H_{2k-2}P_{2m}\right)^{\times2},\right.$

$\hphantom{\qquad}\left.P_{2k-2}H_{2m},\left(\frac{1}{2}P_{2m-(2k-2)}\right)^{\times2},P_{2m-(2k-1)},(*)\right)\quad\rightarrow$

\medskip{}

$\left(P_{2k-1}H_{2m}+H_{2k-2}P_{2m};P_{2k-2}H_{2m},\left(-\frac{1}{2}P_{2k-2}H_{2m}+\frac{1}{2}H_{2k-2}P_{2m}\right)^{\times2},\right.$

$\hphantom{\qquad}\left.P_{2k-2}H_{2m},\left(\frac{1}{2}P_{2m-(2k-2)}\right)^{\times2},P_{2m-(2k-1)},(*)\right)\quad\rightsquigarrow$

\medskip{}

$\left(P_{2k-1}H_{2m}+H_{2k-2}P_{2m};\left(P_{2k-2}H_{2m}\right)^{\times2},\left(\frac{1}{2}P_{2m-(2k-2)}\right)^{\times4},P_{2m-(2k-1)},(*)\right).$

\medskip{}

\noindent Now, after reordering this last class, we obtain $V_{k-1}^{m}$
as required.\end{proof}
\begin{lem}
\label{lem:3rd_red_of_E(alpha_2m)}$V_{2}^{m}$ reduces in 8 Cremona
moves to the class
\[
V^{m}:=\left(H_{2m};H_{2m}-\frac{1}{2}P_{2m},\left(\frac{1}{2}P_{2m}\right)^{\times3},\left(P_{2m-1}\right)^{\times2},B_{1}^{m}\right).
\]
\end{lem}
\begin{proof}
We have

\medskip{}

$\begin{aligned}V_{2}^{m}= & \left(29H_{2m}+17P_{2m};\left(11H_{2m}+\frac{17}{2}P_{2m}\right)^{\times4},9H_{2m}+7P_{2m},\left(12H_{2m}\right)^{\times2},\right.\\
 & \left.\left(2H_{2m}+\frac{3}{2}P_{2m}\right)^{\times4},H_{2m}+P_{2m},H_{2m}+\frac{1}{2}P_{2m},(*)\right);
\end{aligned}
$

\medskip{}

$\left(25H_{2m}+\frac{17}{2}P_{2m};11H_{2m}+\frac{17}{2}P_{2m},9H_{2m}+7P_{2m},\left(12H_{2m}\right)^{\times2},\left(7H_{2m}\right)^{\times3},\right.$

$\hphantom{\qquad}\left.\left(2H_{2m}+\frac{3}{2}P_{2m}\right)^{\times4},H_{2m}+P_{2m},H_{2m}+\frac{1}{2}P_{2m},(*)\right);$

\medskip{}

$\left(18H_{2m}+\frac{3}{2}P_{2m};12H_{2m},\left(7H_{2m}\right)^{\times3},4H_{2m}+\frac{3}{2}P_{2m},\left(2H_{2m}+\frac{3}{2}P_{2m}\right)^{\times4},\right.$

$\hphantom{\qquad}\left.2H_{2m},H_{2m}+P_{2m},H_{2m}+\frac{1}{2}P_{2m},P_{2m-3},(*)\right);$

\medskip{}

$\left(10H_{2m}+3P_{2m};7H_{2m},\left(4H_{2m}+\frac{3}{2}P_{2m}\right)^{\times2},\left(2H_{2m}+\frac{3}{2}P_{2m}\right)^{\times4},\right.$

$\hphantom{\qquad}\left.2H_{2m},H_{2m}+P_{2m},H_{2m}+\frac{1}{2}P_{2m},\left(\frac{1}{2}P_{2m-2}\right)^{\times2},P_{2m-3},(*)\right);$

\medskip{}

$\left(5H_{2m}+3P_{2m};\left(2H_{2m}+\frac{3}{2}P_{2m}\right)^{\times4},\left(2H_{2m}\right)^{\times2},H_{2m}+P_{2m},H_{2m}+\frac{1}{2}P_{2m},\right.$

$\hphantom{\qquad}\left.\left(\frac{1}{2}P_{2m-2}\right)^{\times2},P_{2m-3},(*)\right);$

\medskip{}

$\left(4H_{2m}+\frac{3}{2}P_{2m};2H_{2m}+\frac{3}{2}P_{2m},\left(2H_{2m}\right)^{\times2},H_{2m}+P_{2m},H_{2m}+\frac{1}{2}P_{2m},\right.$

$\hphantom{\qquad}\left.\left(H_{2m}\right)^{\times3},\left(\frac{1}{2}P_{2m-2}\right)^{\times4},P_{2m-3},(*)\right);$

\medskip{}

$\left(2H_{2m}+\frac{3}{2}P_{2m};H_{2m}+P_{2m},H_{2m}+\frac{1}{2}P_{2m},\frac{3}{2}P_{2m},\left(H_{2m}\right)^{\times3},\left(\frac{1}{2}P_{2m-2}\right)^{\times4},\right.$

$\hphantom{\qquad}\left.P_{2m-3},(*)\right);$

\medskip{}

$\left(2H_{2m};\left(H_{2m}\right)^{\times3},H_{2m}-\frac{1}{2}P_{2m},P_{2m-1},\left(\frac{1}{2}P_{2m-2}\right)^{\times4},P_{2m-3},(*)\right);$

\medskip{}

$\left(H_{2m};H_{2m}-\frac{1}{2}P_{2m},P_{2m-1},\left(\frac{1}{2}P_{2m-2}\right)^{\times4},P_{2m-3},(*)\right).$

\medskip{}

\noindent After reordering this last class, we obtain $V^{m}$ as
required.\end{proof}
\begin{lem}
\label{lem:4th_red_of_E(alpha_2m)}For all $m\geqslant3$, $V^{m}$
reduces in 5 Cremona moves to $V^{m-1}$.\end{lem}
\begin{proof}
We have

\medskip{}

$V^{m}=\left(H_{2m};H_{2m}-\frac{1}{2}P_{2m},\left(\frac{1}{2}P_{2m}\right)^{\times3},\left(P_{2m-1}\right)^{\times2},\left(\frac{1}{2}P_{2m-2}\right)^{\times8},(*)\right);$

\medskip{}

$\left(H_{2m}-\frac{1}{2}P_{2m};\frac{1}{2}P_{2m},\left(P_{2m-1}\right)^{\times3},\left(\frac{1}{2}P_{2m-2}\right)^{\times8},(*)\right);$

\medskip{}

$\left(\frac{1}{2}P_{2m};P_{2m-1},\left(\frac{1}{2}P_{2m-2}\right)^{\times9},(*)\right);$\medskip{}

$\left(P_{2m-1};P_{2m-1}-\frac{1}{2}P_{2m-2},\left(\frac{1}{2}P_{2m-2}\right)^{\times7},(*)\right);$

\medskip{}

$\left(P_{2m-1}-\frac{1}{2}P_{2m-2};H_{2m-2},\left(\frac{1}{2}P_{2m-2}\right)^{\times5},(*)\right);$

\medskip{}

$\left(H_{2m-2};H_{2m-2}-\frac{1}{2}P_{2m-2},\left(\frac{1}{2}P_{2m-2}\right)^{\times3},(*)\right).$

\medskip{}

\noindent After reordering this class, we obtain $V^{m-1}$ as required.\end{proof}
\begin{lem}
\label{lem:5th_red_of_E(alpha_2m)}$V^{2}$ reduces in 8 Cremona moves
to $(0;-1)$.\end{lem}
\begin{proof}
We have

\medskip{}

$V^{2}=\left(17;11,6^{\times3},5^{\times2},1^{\times11}\right);$

\medskip{}

$\left(11;6,5^{\times3},1^{\times11}\right);$\medskip{}

$\left(6;5,1^{\times12}\right);$\medskip{}

$\left(5;4,1^{\times10}\right);$\medskip{}

$\left(4;3,1^{\times8}\right);$\medskip{}

$\left(3;2,1^{\times6}\right);$\medskip{}

$\left(2;1^{\times5}\right);$\medskip{}

$\left(1;1^{\times2}\right);$\medskip{}

$\left(0;-1\right).$
\end{proof}

\subsubsection{The classes $E\left(\alpha_{2m-1}\right)$ reduce to $(0;-1)$}

One readily checks that the classes $E\left(\alpha_{2m-1}\right)$
reduce to $(0;-1)$ for $m=1,2$. In the following, we reduce the
classes $E\left(\alpha_{2m-1}\right)$ for $m\geqslant3$.
\begin{lem}
The continued fraction expansion of $\alpha_{2m-1}$ is
\[
\left[5;\left\{ 1,4\right\} ^{\times(m-2)},1,3,1,1,\left\{ 4,1\right\} ^{\times(m-1)}\right].
\]
Moreover, if $v_{j}:=\left(2H_{j}+P_{j}\right)H_{2m}-H_{j}P_{2m}$,
then
\[
\begin{aligned}E\left(\alpha_{2m-1}\right)= & \left(P_{2m-1}H_{2m},P_{2m-1}H_{2m};\right.\\
 & \left(\frac{1}{2}v_{2m-2}\right)^{\times5},v_{2m-3},\left(\frac{1}{2}v_{2m-4}\right)^{\times4},\ldots,v_{3},\left(\frac{1}{2}v_{2}\right)^{\times4},\\
 & \quad v_{1},\left(H_{2m}-\frac{1}{2}P_{2m}\right)^{\times3},\frac{1}{2}P_{2m},P_{2m-1},\\
 & \left.\left(\frac{1}{2}P_{2m-2}\right)^{\times4},P_{2m-3},\ldots,\left(\frac{1}{2}P_{2}\right)^{\times4},P_{1},1\right).
\end{aligned}
\]
 \end{lem}
\begin{proof}
The first terms of $W'\left(\alpha_{2m-1}\right)$ are $\left(\frac{1}{2}v_{2m-2}\right)^{\times5}$
since $\frac{1}{2}v_{2m-2}=2P_{2m-1}^{2}$ and $5<\alpha_{2m-1}<6$
for $m\geqslant2$. Before determining the next terms, we prove that
for all $k\geqslant1$, $v_{2k+1}>\frac{1}{2}v_{2k}>v_{2k-1}$. Indeed
\[
\begin{aligned}v_{2k+1} & =\left(H_{2k+1}+P_{2k+1}\right)P_{2m}+\left(2H_{2k+1}+P_{2k+1}\right)P_{2m-1}\\
 & =\left(5P_{2k}+2P_{2k-1}\right)P_{2m}+\left(8P_{2k}+3P_{2k-1}\right)P_{2m-1}\\
 & >\left(P_{2k}+\frac{1}{2}P_{2k-1}\right)P_{2m}+\left(\frac{3}{2}P_{2k}+P_{2k-1}\right)P_{2m-1}\\
 & =\frac{1}{2}\left(H_{2k}+P_{2k}\right)P_{2m}+\frac{1}{2}\left(2H_{2k}+P_{2k}\right)P_{2m-1}\\
 & =\frac{1}{2}v_{2k}\\
 & =\left(\frac{5}{2}P_{2k-1}+P_{2k-2}\right)P_{2m}+\left(3P_{2k-1}+\frac{7}{2}P_{2k-2}+P_{2k-3}\right)P_{2m-1}\\
 & >\left(2P_{2k-1}+P_{2k-2}\right)P_{2m}+\left(3P_{2k-1}+2P_{2k-2}\right)P_{2m-1}\\
 & =\left(H_{2k-1}+P_{2k-1}\right)P_{2m}+\left(2H_{2k-1}+P_{2k-1}\right)P_{2m-1}\\
 & =v_{2k-1}.
\end{aligned}
\]
So the next term of $W'\left(\alpha_{2m-1}\right)$ is $H_{2m}^{2}-5\left(2P_{2m-1}^{2}\right)=v_{2m-3}<\frac{1}{2}v_{2m-2}$.
Moreover, for all $k\geqslant1$,
\[
\begin{aligned}v_{2k+1}-4\left(\frac{1}{2}v_{2k}\right) & =\left(2H_{2k+1}+P_{2k+1}-4H_{2k}-2P_{2k}\right)H_{2m}\\
 & -\left(H_{2k+1}-2H_{2k}\right)P_{2m}\\
 & =\left(2H_{2k-1}+P_{2k-1}\right)H_{2m}-H_{2k-1}P_{2m}=v_{2k-1}<\frac{1}{2}v_{2k};\\
\\
\frac{1}{2}v_{2k}-v_{2k-1} & =\left(H_{2k}+\frac{1}{2}P_{2k}-2H_{2k-1}-P_{2k-1}\right)H_{2m}\\
 & -\left(\frac{1}{2}H_{2k}-H_{2k-1}\right)P_{2m}\\
 & =\frac{1}{2}\left(2H_{2k-2}+P_{2k-2}\right)H_{2m}-\frac{1}{2}H_{2k-2}P_{2m}=\frac{1}{2}v_{2k-2}<v_{2k-1}.
\end{aligned}
\]
This proves that the first terms of $W'\left(\alpha_{2m-1}\right)$
are 
\[
\left(\frac{1}{2}v_{2m-2}\right)^{\times5},v_{2m-3},\left(\frac{1}{2}v_{2m-4}\right)^{\times4},\ldots,v_{3},\left(\frac{1}{2}v_{2}\right)^{\times4},v_{1}.
\]
The next three terms are $\left(H_{2m}-\frac{1}{2}P_{2m}\right)^{\times3},\frac{1}{2}P_{2m},P_{2m-1}$
since
\[
\begin{aligned}\frac{1}{2}v_{2}-v_{1} & =H_{2m}-\frac{1}{2}P_{2m}<3H_{2m}-P_{2m}=v_{1};\\
\\
v_{1}-3\left(H_{2m}-\frac{1}{2}P_{2m}\right) & =\frac{1}{2}P_{2m}<H_{2m}-\frac{1}{2}P_{2m};\\
\\
H_{2m}-\frac{1}{2}P_{2m}-\frac{1}{2}P_{2m} & =P_{2m-1}<\frac{1}{2}P_{2m}.
\end{aligned}
\]
Since the last terms are the same as those of $W'\left(\alpha_{2m}\right)$,
the lemma is proved.
\end{proof}
Let us introduce again some notations.
\begin{defn}
Set
\[
\begin{aligned}\hat{A}_{k}^{m}:= & \left(\left(\frac{1}{2}v_{2k-2}\right)^{\times4},v_{2k-3},\left(H_{2k-1}H_{2m}-H_{2k-1}P_{2m}\right)^{\times2},\left(\frac{1}{2}v_{2k-4}\right)^{\times4},v_{2k-5},\right.\\
 & \left.\ldots,\left(\frac{1}{2}v_{2}\right)^{\times4},v_{1},\left(H_{2m}-\frac{1}{2}P_{2m}\right)^{\times3},\frac{1}{2}P_{2m},P_{2m-1}\right),\\
\hat{B}_{k}^{m}:= & \left(\left(\frac{1}{2}P_{2m-2}\right)^{\times4},P_{2m-3},\ldots,\left(\frac{1}{2}P_{2m-2k+2}\right)^{\times4},P_{2m-2k+1},\right.\\
 & \left.\left(\frac{1}{2}P_{2m-2k}\right)^{\times8},\left(P_{2m-2k-1}\right)^{\times2},\ldots,\left(\frac{1}{2}P_{2}\right)^{\times8},\left(P_{1}\right)^{\times2},1\right),\\
\hat{V}_{k}^{m}:= & \left(P_{2k}H_{2m}+H_{2k-1}P_{2m};\hat{A}_{k}^{m},\hat{B}_{k}^{m}\right).
\end{aligned}
\]
Note that $\hat{B}_{k}^{m}$ is actually equal to the vector $B_{k}^{m}$
that we used in the reduction of the classes~$E\left(\alpha_{2m}\right)$.
Here, $\hat{A}_{k}^{m}$ has the structure $\left[4,1,2,\left\{ 4,1\right\} ^{\times(k-2)},3,1,1\right]$
and $\hat{B}_{k}^{m}$ has the structure $\left[\left\{ 4,1\right\} ^{\times(k-1)},\left\{ 8,2\right\} ^{\times(m-k)},1\right]$.
We use again the convention that if $k=m$, $\hat{B}_{m}^{m}$ has
the structure $\left[\left\{ 4,1\right\} ^{\times(m-1)},1\right]$
and that if $k=1$, $\hat{B}_{1}^{m}$ has the structure $\left[\left\{ 8,2\right\} ^{\times(m-1)},1\right]$.\end{defn}
\begin{lem}
The image of $E\left(\alpha_{2m-1}\right)$ by $\varphi_{*}$ is the
class 
\[
\varphi_{*}\left(E\left(\alpha_{2m-1}\right)\right)=\left(P_{2m}H_{2m}-H_{2m-1}P_{2m};\hat{A}_{m}^{m},\hat{B}_{m}^{m}\right)=\hat{V}_{m}^{m}.
\]
\end{lem}
\begin{proof}
The first terms of $E\left(\alpha_{2m-1}\right)$ are $\left(P_{2m-1}H_{2m},P_{2m-1}H_{2m};\frac{1}{2}v_{2m-2},(*)\right)$.
Since $\frac{1}{2}v_{2m-2}=2P_{2m-1}^{2}$, we get
\[
\begin{aligned}\varphi_{*}\left(E\left(\alpha_{2m-1}\right)\right) & =\left(2P_{2m-1}H_{2m}-2P_{2m-1}^{2};\left(P_{2m-1}H_{2m}-2P_{2m-1}^{2}\right)^{\times2},(*)\right)\\
 & =\left(P_{2m}H_{2m}-H_{2m-1}P_{2m};\left(H_{2m-1}H_{2m}-H_{2m-1}P_{2m}\right)^{\times2},(*)\right).
\end{aligned}
\]
After reordering, this last class is $\left(P_{2m}H_{2m}-H_{2m-1}P_{2m};\hat{A}_{m}^{m},\hat{B}_{m}^{m}\right)=\hat{V}_{m}^{m}$
as required.\end{proof}
\begin{lem}
For all $3\leqslant k\leqslant m$, $\hat{V}_{k}^{m}$ reduces to
$\hat{V}_{k-1}^{m}$ in 4 Cremona moves.\end{lem}
\begin{proof}
We have

\medskip{}

$\begin{aligned}\hat{V}_{k}^{m}= & \left(P_{2k}H_{2m}-H_{2k-1}P_{2m};\frac{1}{2}\left(\left(2H_{2k-2}+P_{2k-2}\right)H_{2m}-H_{2k-2}P_{2m}\right)^{\times4},\right.\\
 & \left.\left(2H_{2k-3}+P_{2k-3}\right)H_{2m}-H_{2k-3}P_{2m},\left(H_{2k-1}H_{2m}-H_{2k-1}P_{2m}\right)^{\times2},(*)\right)\quad\rightarrow
\end{aligned}
$

\medskip{}

$\left(\left(H_{2k-2}+\frac{9}{2}P_{2k-2}\right)H_{2m}-\left(\frac{1}{2}H_{2k-2}+4P_{2k-2}\right)P_{2m};\right.$

$\hphantom{\qquad}\left(2P_{2k-2}H_{2m}-2P_{2k-2}P_{2m}\right)^{\times3},\frac{1}{2}\left(\left(2H_{2k-2}+P_{2k-2}\right)H_{2m}-H_{2k-2}P_{2m}\right),$

$\hphantom{\qquad}\left.\left(2H_{2k-3}+P_{2k-3}\right)H_{2m}-H_{2k-3}P_{2m},\left(H_{2k-1}H_{2m}-H_{2k-1}P_{2m}\right)^{\times2},(*)\right)\quad\rightsquigarrow$

\medskip{}

$\left(\left(H_{2k-2}+\frac{9}{2}P_{2k-2}\right)H_{2m}-\left(\frac{1}{2}H_{2k-2}+4P_{2k-2}\right)P_{2m};\right.$

$\hphantom{\qquad}\frac{1}{2}\left(\left(2H_{2k-2}+P_{2k-2}\right)H_{2m}-H_{2k-2}P_{2m}\right),\left(2H_{2k-3}+P_{2k-3}\right)H_{2m}-H_{2k-3}P_{2m},$

$\hphantom{\qquad}\left.\left(H_{2k-1}H_{2m}-H_{2k-1}P_{2m}\right)^{\times2},\left(2P_{2k-2}H_{2m}-2P_{2k-2}P_{2m}\right)^{\times3},(*)\right)\quad\rightarrow$

\medskip{}

$\left(\left(H_{2k-2}+\frac{7}{2}P_{2k-2}\right)H_{2m}-\left(\frac{1}{2}H_{2k-2}+4P_{2k-2}\right)P_{2m};\right.$

$\hphantom{\qquad}\left(H_{2k-2}-\frac{1}{2}P_{2k-2}\right)H_{2m}-\frac{1}{2}H_{2k-2}P_{2m},H_{2k-3}H_{2m}-H_{2k-3}P_{2m},$

$\hphantom{\qquad}P_{2k-1}H_{2m}-H_{2k-1}P_{2m},H_{2k-1}H_{2m}-H_{2k-1}P_{2m},$

$\hphantom{\qquad}\left.\left(2P_{2k-2}H_{2m}-2P_{2k-2}P_{2m}\right)^{\times3},(*)\right)\quad\rightsquigarrow$

\medskip{}

$\left(\left(H_{2k-2}+\frac{7}{2}P_{2k-2}\right)H_{2m}-\left(\frac{1}{2}H_{2k-2}+4P_{2k-2}\right)P_{2m};\right.$

$\hphantom{\qquad}H_{2k-1}H_{2m}-H_{2k-1}P_{2m},\left(2P_{2k-2}H_{2m}-2P_{2k-2}P_{2m}\right)^{\times3},$

$\hphantom{\qquad}\left(H_{2k-2}-\frac{1}{2}P_{2k-2}\right)H_{2m}-\frac{1}{2}H_{2k-2}P_{2m},H_{2k-3}H_{2m}-H_{2k-3}P_{2m},$

$\hphantom{\qquad}\left.P_{2m-(2k-1)},(*)\right)\quad\rightarrow$

\medskip{}

$\left(P_{2k-1}H_{2m}-2P_{2k-2}P_{2m};\left(H_{2k-2}-\frac{1}{2}P_{2k-2}\right)H_{2m}-\frac{1}{2}H_{2k-2}P_{2m},\right.$

$\hphantom{\qquad}\left(-\frac{1}{2}P_{2k-2}H_{2m}+\frac{1}{2}H_{2k-2}P_{2m}\right)^{\times2},2P_{2k-2}H_{2m}-2P_{2k-2}P_{2m},$

$\hphantom{\qquad}\left(H_{2k-2}-\frac{1}{2}P_{2k-2}\right)H_{2m}-\frac{1}{2}H_{2k-2}P_{2m},H_{2k-3}H_{2m}-H_{2k-3}P_{2m},$

$\hphantom{\qquad}\left.P_{2m-(2k-1)},(*)\right)\quad\rightsquigarrow$

\medskip{}

$\left(P_{2k-1}H_{2m}-2P_{2k-2}P_{2m};2P_{2k-2}H_{2m}-2P_{2k-2}P_{2m},\right.$

$\hphantom{\qquad}\left(\left(H_{2k-2}-\frac{1}{2}P_{2k-2}\right)H_{2m}-\frac{1}{2}H_{2k-2}P_{2m}\right)^{\times2},H_{2k-3}H_{2m}-H_{2k-3}P_{2m},$

$\hphantom{\qquad}\left.\left(\frac{1}{2}P_{2m-(2k-2)}\right)^{\times2},P_{2m-(2k-1)},(*)\right)\quad\rightarrow$

\medskip{}

$\left(P_{2k-2}H_{2m}-H_{2k-3}P_{2m};H_{2k-3}H_{2m}-H_{2k-3}P_{2m},\right.$

$\hphantom{\qquad}\left(-\frac{1}{2}P_{2k-2}H_{2m}+\frac{1}{2}H_{2k-2}P_{2m}\right)^{\times2},H_{2k-3}H_{2m}-H_{2k-3}P_{2m},$

$\hphantom{\qquad}\left.\left(\frac{1}{2}P_{2m-(2k-2)}\right)^{\times2},P_{2m-(2k-1)},(*)\right)\quad\rightsquigarrow$

\medskip{}

$\left(P_{2k-2}H_{2m}-H_{2k-3}P_{2m};\left(H_{2k-3}H_{2m}-H_{2k-3}P_{2m}\right)^{\times2},\left(\frac{1}{2}P_{2m-(2k-2)}\right)^{\times4},\right.$

$\hphantom{\qquad}\left.P_{2m-(2k-1)},(*)\right).$

\medskip{}
Now, after reordering this last class, we obtain $\hat{V}_{k-1}^{m}$
as required.\end{proof}
\begin{lem}
$\hat{V}_{2}^{m}$ reduces in 5 Cremona moves to the class
\[
\hat{V}^{m}:=\left(H_{2m}-\frac{1}{2}P_{2m};\frac{1}{2}P_{2m},\left(P_{2m-1}\right)^{\times3},\hat{B}_{1}^{m}\right).
\]
\end{lem}
\begin{proof}
We have

\medskip{}

$\begin{aligned}\hat{V}_{2}^{m}= & \left(12H_{2m}-7P_{2m};\left(4H_{2m}-\frac{3}{2}P_{2m}\right)^{\times4},3H_{2m}-P_{2m},\left(7H_{2m}-7P_{2m}\right)^{\times2},\right.\\
 & \left.\left(H_{2m}-\frac{1}{2}P_{2m}\right)^{\times3},(*)\right);
\end{aligned}
$

\medskip{}

$\left(12H_{2m}-\frac{19}{2}P_{2m};4H_{2m}-\frac{3}{2}P_{2m},3H_{2m}-P_{2m},\left(7H_{2m}-7P_{2m}\right)^{\times2},\right.$

$\hphantom{\qquad}\left.\left(4H_{2m}-4P_{2m}\right)^{\times3},\left(H_{2m}-\frac{1}{2}P_{2m}\right)^{\times3},(*)\right);$

\medskip{}

$\left(10H_{2m}-\frac{19}{2}P_{2m};7H_{2m}-7P_{2m},\left(4H_{2m}-4P_{2m}\right)^{\times3},2H_{2m}-\frac{3}{2}P_{2m},\right.$

$\hphantom{\qquad}\left.\left(H_{2m}-\frac{1}{2}P_{2m}\right)^{\times3},P_{2m-1},P_{2m-3},(*)\right);$

\medskip{}

$\left(5H_{2m}-4P_{2m};4H_{2m}-4P_{2m},\left(2H_{2m}-\frac{3}{2}P_{2m}\right)^{\times2},\left(H_{2m}-\frac{1}{2}P_{2m}\right)^{\times3},\right.$

$\hphantom{\qquad}\left.P_{2m-1},\left(\frac{1}{2}P_{2m-2}\right)^{\times2},P_{2m-3},(*)\right);$

\medskip{}

$\left(2H_{2m}-P_{2m};\left(H_{2m}-\frac{1}{2}P_{2m}\right)^{\times3},\left(P_{2m-1}\right)^{\times2},\left(\frac{1}{2}P_{2m-2}\right)^{\times4},P_{2m-3},(*)\right);$

\medskip{}

$\left(H_{2m}-\frac{1}{2}P_{2m};\left(H_{2m}-\frac{1}{2}P_{2m}\right)^{\times3},\left(P_{2m-1}\right)^{\times2},\left(\frac{1}{2}P_{2m-2}\right)^{\times4},P_{2m-3},(*)\right).$

\medskip{}
After reordering this last class, we obtain $\hat{V}^{m}$ as required.\end{proof}
\begin{lem}
For all $m\geqslant3$, $\hat{V}^{m}$ reduces in 5 Cremona moves
to $\hat{V}^{m-1}$.\end{lem}
\begin{proof}
We have

\medskip{}

$\hat{V}^{m}=\left(H_{2m}-\frac{1}{2}P_{2m};\frac{1}{2}P_{2m},\left(P_{2m-1}\right)^{\times3},\left(\frac{1}{2}P_{2m-2}\right)^{\times8},(*)\right);$

\medskip{}

$\left(\frac{1}{2}P_{2m};P_{2m-1},\left(\frac{1}{2}P_{2m-2}\right)^{\times9},(*)\right);$

\medskip{}

$\left(P_{2m-1};P_{2m-1}-\frac{1}{2}P_{2m-2},\left(\frac{1}{2}P_{2m-2}\right)^{\times7},(*)\right);$

\medskip{}

$\left(P_{2m-1}-\frac{1}{2}P_{2m-2};H_{2m-2},\left(\frac{1}{2}P_{2m-2}\right)^{\times5},(*)\right);$

\medskip{}

$\left(H_{2m-2};H_{2m-2}-\frac{1}{2}P_{2m-2},\left(\frac{1}{2}P_{2m-2}\right)^{\times3},(*)\right);$

\medskip{}

$\left(H_{2m-2}-\frac{1}{2}P_{2m-2};\frac{1}{2}P_{2m-2},P_{2m-3},(*)\right).$

\medskip{}
After reordering this class, we obtain $\hat{V}^{m-1}$ as required.\end{proof}
\begin{lem}
$\hat{V}^{2}$ reduces in 7 Cremona moves to $(0;-1)$.\end{lem}
\begin{proof}
We have

\medskip{}

$\hat{V}^{2}=\left(11;6,5^{\times3},1^{\times11}\right);$

\medskip{}

$\left(6;5,1^{\times12}\right);$\medskip{}

$\left(5;4,1^{\times10}\right);$\medskip{}

$\left(4;3,1^{\times8}\right);$\medskip{}

$\left(3;2,1^{\times6}\right);$\medskip{}

$\left(2;1^{\times5}\right);$\medskip{}

$\left(1;1^{\times2}\right);$\medskip{}

$\left(0;-1\right).$
\end{proof}

\section{The interval $\left[\sigma^{2},6\right]$}

\noindent In this section we prove that $c(a)=\frac{a+1}{4}$ on the
interval $\left[\sigma^{2},6\right]$, which is a part of Theorem
\ref{thm:c(a) on [sigma^2,7+1/32]}. Notice that the class $\left(2,2;2,1^{\times5}\right)$
gives the constraint $\frac{a+1}{4}$ on $\left[\sigma^{2},6\right]$.
If suffices thus to show that no class gives a stronger constraint.
We begin by showing this on the interval $\left[5\frac{12}{13},6\right]$,
and will then spend some efforts to extend it to $\left[\sigma^{2},6\right]$.
\begin{prop}
\label{prop:c(a)=00003D(a+1)/4 on [5+12/13,6]}For $a\in\left[5\frac{12}{13},6\right]$,
we have $c(a)=\frac{a+1}{4}$.\end{prop}
\begin{proof}
Write $a=5+x\in\left[5\frac{12}{13},6\right[$. Then $w(a)=\left(1^{\times5},x,w_{7},\ldots,w_{M}\right)$.
Since $x\geqslant\frac{12}{13}$, $1-x\leqslant\frac{x}{9}$, thus
at least the nine first of the weights $w_{7},\ldots,w_{M}$ are equal.
Then, by Corollary 1.2.4 in \cite{MS}, the $M-6$ balls of weights
$w_{7},\ldots,w_{M}$ fully fill the ball of weight $\lambda$, where
$a=5+x^{2}+\lambda^{2}$. Thus to prove that $c(a)=\frac{a+1}{4}$
on $\left[\sigma^{2},6\right]$, it suffices to check that the seven
balls of weights $1^{\times5},x,\lambda$ embed into the cube $C\left(\frac{(5+x)+1}{4}\right)$,
or in other words, that the finite number of classes of $\mathcal{E}_{M}$
with $M\leqslant7$ don't give any embedding constraint stronger than
$\frac{a+1}{4}$ for these seven balls.

This is clear for the classes belonging to $\mathcal{E}_{M}$ with
$M\leqslant6$. The strongest constraint of $\mathcal{E}_{7}$ comes
from the class $\left(4,3;2^{\times6},1\right)$ for which
\[
\mu\left(4,3;2^{\times6},1\right)(5+x)=\frac{10+2x+\lambda}{7}.
\]
Notice that $\frac{10+2x+\lambda}{7}\leqslant\frac{(5+x)+1}{4}$ if
and only if $\lambda\leqslant\frac{2-x}{4}$. Recall that $a=5+x=5+x^{2}+\lambda^{2}$,
whence $\lambda^{2}=x-x^{2}$. Since $x-x^{2}\leqslant\left(\frac{2-x}{4}\right)^{2}$
for $x\in\left[\frac{12}{13},1\right]$. Thus the class $\left(4,3;2^{\times6},1\right)$
indeed gives no stronger constraint than $\frac{a+1}{4}$. Finally,
by continuity of $c$, $c(6)=\frac{6+1}{4}=\frac{7}{4}$.\end{proof}
\begin{prop}
\label{prop:first estimates of d and lambda^2}\renewcommand{\labelenumi}{(\roman{enumi})}
Let $(d,e;m)\in\mathcal{E}$ be a class such that $\mu(d,e;m)(a)>\frac{a+1}{4}\geqslant\sqrt{\frac{a}{2}}$ for some $a\in\left[\sigma^{2},6\right]$. Then
\begin{enumerate}
\item $d<\frac{2\sqrt{a}}{\sqrt{a^{2}-6a+1}}$ for a class of the form $(d,d;m)$ and $d<\frac{\sqrt{2a}}{\sqrt{a^{2}-6a+1}}$ for a class of the form $(d+\frac{1}{2},d-\frac{1}{2};m)$.
\item Moreover, if we denote $\lambda^{2}:=1-\sum_{i=1}^{M}\varepsilon_{i}^{2}$ (respectively $\lambda^{2}:=\frac{1}{2}-\sum_{i=1}^{M}\varepsilon_{i}^{2}$) for a class of the form $(d,d;m)$ (respectively for a class of the form $(d+\frac{1}{2},d-\frac{1}{2};m)$), then $\lambda^{2}>d^{2}\sqrt{\frac{2}{a}}\, y(a)$.
\end{enumerate}%
\end{prop}
\begin{proof}
(i) Let us first prove this for a class of the form $(d,d;m)$. By
Lemma \ref{lem:first properties of mu}\,(i), we have
\[
\frac{a+1}{4}<\mu(d,d;m)(a)\leqslant\sqrt{1+\frac{1}{2d^{2}}}\,\sqrt{\frac{a}{2}}.
\]
Thus
\[
\frac{(a+1)^{2}}{8a}-1<\frac{1}{2d^{2}}
\]
and so
\[
d<\frac{2\sqrt{a}}{\sqrt{a^{2}-6a+1}}.
\]
Similarly, for a class of the form $(d+\frac{1}{2},d-\frac{1}{2};m)$,
we have
\[
\frac{a+1}{4}<\mu(d+\frac{1}{2},d-\frac{1}{2};m)(a)\leqslant\sqrt{1+\frac{1}{4d^{2}}}\sqrt{\frac{a}{2}},
\]
and thus
\[
d<\frac{\sqrt{2a}}{\sqrt{a^{2}-6a+1}}.
\]

(ii) For a class of the form $(d,d;m)$, we have by Proposition \ref{prop:characterization of E_M}\,(i),
\[
\begin{aligned}2d^{2}+1 & =\left\langle m,m\right\rangle =\left\langle \frac{\sqrt{2}d}{\sqrt{a}}w(a)+\varepsilon,\frac{\sqrt{2}d}{\sqrt{a}}w(a)+\varepsilon\right\rangle \\
 & =2d^{2}+\frac{2\sqrt{2}d}{\sqrt{a}}\left\langle w(a),\varepsilon\right\rangle +\left\langle \varepsilon,\varepsilon\right\rangle .
\end{aligned}
\]
Thus
\begin{equation}
\left\langle w(a),\varepsilon\right\rangle =\underset{{\scriptstyle =\lambda^{2}}}{\underbrace{\left(1-\left\langle \varepsilon,\varepsilon\right\rangle \right)}}\,\frac{1}{2d}\sqrt{\frac{a}{2}}.\label{eq:<w(a),epsilon>}
\end{equation}
On the other hand
\[
\begin{aligned}\frac{a+1}{4} & <\mu(d,d;m)(a)=\frac{\left\langle m,w(a)\right\rangle }{2d}=\frac{1}{2d}\left\langle \frac{\sqrt{2}d}{\sqrt{a}}w(a)+\varepsilon,w(a)\right\rangle \\
 & =\sqrt{\frac{a}{2}}+\frac{1}{2d}\left\langle \varepsilon,w(a)\right\rangle .
\end{aligned}
\]
Thus
\[
\left\langle \varepsilon,w(a)\right\rangle >2d\left(\frac{a+1}{4}-\sqrt{\frac{a}{2}}\right)=\frac{1}{2}d\:\underset{{\scriptstyle =y(a)}}{\underbrace{\left(a+1-2\sqrt{2a}\right)}}.
\]
Inserting (\ref{eq:<w(a),epsilon>}) in this inequality, we get
\[
\frac{\lambda^{2}}{2d}\sqrt{\frac{a}{2}}>\frac{d}{2}y(a),
\]
and finally
\[
\lambda^{2}>d^{2}\sqrt{\frac{2}{a}}\, y(a).
\]

Similarly, for a class $(d+\frac{1}{2},d-\frac{1}{2};m)$, we have
\[
2d^{2}+\frac{1}{2}=2d^{2}+\frac{2\sqrt{2}d}{\sqrt{a}}\left\langle w(a),\varepsilon\right\rangle +\left\langle \varepsilon,\varepsilon\right\rangle ,
\]
so
\[
\left\langle w(a),\varepsilon\right\rangle =\underset{{\scriptstyle =\lambda^{2}}}{\underbrace{\left(\frac{1}{2}-\left\langle \varepsilon,\varepsilon\right\rangle \right)}}\frac{1}{2d}\sqrt{\frac{a}{2}}.
\]
The rest of the proof is then identical to the case of a class $(d,d;m)$.
\end{proof}
Notice that the continued fraction of $\sigma^{2}=3+2\sqrt{2}$ is
$\left[5;1,4,1,4,\ldots\right]$. We will now define the so called
\emph{convergents} $c_{k}$ of $\sigma^{2}$ and some other numbers
$u_{k}(j)$ and $v_{k}(j)$ which will play a crucial role in the
proof of Theorem \ref{thm:c(a) on [sigma^2,6]}.
\begin{defn}
\label{def:pell numbers}For all $k,j\geqslant1$, set
\[
\begin{aligned}c_{2k-1} & :=\left[5;\left\{ 1,4\right\} ^{\times(k-1)},1\right]=\left[5;\left\{ 1,4\right\} ^{\times(k-2)},1,5\right],\\
c_{2k} & :=\left[5;\left\{ 1,4\right\} ^{\times k}\right],\\
u_{k}(j) & :=\left[5;\left\{ 1,4\right\} ^{\times(k-1)},1,5,j\right],\\
v_{k}(j) & :=\left[5;\left\{ 1,4\right\} ^{\times(k-1)},1,j\right].
\end{aligned}
\]
\end{defn}
\begin{lem}
\label{lem:formulae for c_k, u_k(j), v_k(j)}\renewcommand{\labelenumi}{(\roman{enumi})}
For all $k,j\geqslant1$, we have the following relations written in lowest terms
\begin{enumerate}
\item ${\displaystyle c_{2k-1}=\frac{\frac{1}{2}P_{2k+2}}{\frac{1}{2}P_{2k}},\quad c_{2k}=\frac{P_{2k+3}}{P_{2k+1}}}$, \item ${\displaystyle u_{k}(j)=\frac{\frac{1}{2}\left(jP_{2k+4}+P_{2k+2}\right)}{\frac{1}{2}\left(jP_{2k+2}+P_{2k}\right)}}$, \item ${\displaystyle v_{k}(j)=\frac{\frac{1}{2}jP_{2k+2}+P_{2k+1}}{\frac{1}{2}jP_{2k}+P_{2k-1}}}$.
\end{enumerate}%
\end{lem}
\begin{proof}
We use the fact that if $\left[a_{0};a_{1},\ldots,a_{M}\right]$ is
a continued fraction and $\frac{p_{k}}{q_{k}}:=\left[a_{0};a_{1},\ldots,a_{k}\right]$
is its $k$-th convergent written in lowest terms, then for any real
number $x$,
\[
\left[a_{0};a_{1},\ldots,a_{k},x\right]=\frac{xp_{k}+p_{k-1}}{xq_{k}+q_{k-1}},
\]
written in lowest terms.

\medskip{}

(i) We argue by induction on $k$. Assertion (i) is clear for $k=1$.
Assume it holds for $k-1$. Then
\[
c_{2k-1}=\left[5;\left\{ 1,4\right\} ^{\times(k-1)},1\right]=\frac{\frac{1}{2}P_{2k}+P_{2k+1}}{\frac{1}{2}P_{2k-2}+P_{2k-1}}=\frac{\frac{1}{2}P_{2k+2}}{\frac{1}{2}P_{2k}},
\]
and
\[
c_{2k}=\left[5;\left\{ 1,4\right\} ^{\times k}\right]=\frac{P_{2k+1}+2P_{2k+2}}{P_{2k-1}+2P_{2k}}=\frac{P_{2k+3}}{P_{2k+1}}.
\]

The proofs of (ii) and (iii) are then straightforward.\end{proof}
\begin{cor}
\label{cor:c_2k+1 < ... < u_k(2) ...}\renewcommand{\labelenumi}{(\roman{enumi})}
We have
\begin{enumerate}
\item $c_{2}<c_{4}<\ldots<c_{2k}<\ldots<\sigma^{2}<\ldots<c_{2k+1}<\ldots<c_{3}<c_{1}$,
\item $c_{2k+1}<\ldots<u_{k}(2)<u_{k}(1)=v_{k}(6)<v_{k}(7)<\ldots<c_{2k-1}$.
\end{enumerate} %
\end{cor}
\begin{proof}
(i) It is a property of convergents that the even convergents (respectively
odd convergents) of a real number $a$ form an increasing sequence
(respectively decreasing sequence) converging to $a$.

(ii) This follows from Lemma \ref{lem:formulae for c_k, u_k(j), v_k(j)}
and the identity $P_{2k+1}P_{2k-1}-P_{2k}^{2}=1$ for all $k\geqslant1$.\end{proof}
\begin{lem}
\label{lem:check three values}To prove that a quadratic identity
of the form
\[
Q(s):=\sum_{i,j\geqslant0}a_{ij}P_{s+i}P_{s+j}+\sum_{j\geqslant0}b_{j}P_{2s+j}+(-1)^{s}c=0
\]
holds for all $s\geqslant0$, it suffices to check it for three distinct
values of $s$. Moreover, if $Q$ is homogeneous and linear (that
is, $c=0$ and $a_{ij}=0$ for all $i,j$), it suffices to check it
for two distinct values of $s$.\end{lem}
\begin{proof}
The proof is similar to the proof of Proposition 3.2.3 in \cite{MS}.
The only part of the proof which is slightly different is the proof
of their Lemma 3.2.2, which we have adapted in the following Lemma
\ref{lem:lemma before check for three values}.\end{proof}
\begin{lem}
\label{lem:lemma before check for three values}For all $i\geqslant0$,
there exists a finite number of rational coefficients $a_{ij},c_{i}$
such that the identity
\[
P_{s+i}P_{s}=\sum_{j\geqslant0}a_{ij}P_{2s+j}+(-1)^{s}c_{i}
\]
holds for all $s\geqslant0$. Moreover, $c_{i}=-\sum_{j\geqslant0}a_{ij}P_{j}$.\end{lem}
\begin{proof}
In view of the relation 
\begin{equation}
P_{k}=2P_{k-1}+P_{k-2}\label{eq:relation Pell numbers}
\end{equation}
it suffices to prove it for two distinct values of $i$. We claim
that for $i=0$ and $i=2$ we have the following relations for all
$s\geqslant0$:\foreignlanguage{french}{
\begin{eqnarray}
4P_{s}^{2} & = & P_{2s+1}-P_{2s}-(-1)^{s},\label{eq:eq for i=00003D0}\\
8P_{s+2}P_{s} & = & P_{2s+1}+P_{2s+3}-6(-1)^{s},\label{eq:eq for i=00003D2}
\end{eqnarray}
}To prove this, we use the two well-known identities for Pell numbers,\foreignlanguage{french}{
\begin{eqnarray}
P_{k}^{2} & = & P_{k+1}P_{k-1}-(-1)^{k},\label{eq:eq P_k^2}\\
P_{2k-1} & = & P_{k}^{2}+P_{k-1}^{2}.\label{eq:eq P_2k-1}
\end{eqnarray}
}

We start with $i=0$. The relation is true for $s=0$. Now, if $s\geqslant1$,
applying (\ref{eq:relation Pell numbers}) to (\ref{eq:eq for i=00003D0})
gives
\[
8P_{s}^{2}=P_{s+1}^{2}+2P_{s}^{2}+P_{s-1}^{2}-2(-1)^{s}.
\]
Using again (\ref{eq:relation Pell numbers}), we get
\begin{equation}
6P_{s}^{2}=P_{s+1}^{2}+P_{s-1}^{2}-2(-1)^{s}=4P_{s}^{2}+4P_{s}P_{s-1}+2P_{s-1}^{2}-2(-1)^{s}.\label{eq:formula_to_prove the_case i=00003D2}
\end{equation}
Finally, applying once more (\ref{eq:relation Pell numbers}), we
obtain
\[
2P_{s}^{2}=2P_{s-1}\left(2P_{s}+P_{s-1}\right)-2(-1)^{s}=2P_{s+1}P_{s-1}-2(-1)^{s},
\]
which is true by (\ref{eq:eq P_k^2}).

For $i=2$, applying (\ref{eq:eq P_k^2}) to the LHS of (\ref{eq:eq for i=00003D2}),
and (\ref{eq:eq P_2k-1}) to the RHS, we obtain
\[
8P_{s+1}^{2}+8(-1)^{s+1}=P_{s+2}^{2}+2P_{s+1}^{2}+P_{s}^{2}-6(-1)^{s},
\]
which is equivalent to
\[
6P_{s+1}^{2}=P_{s+2}^{2}+P_{s}^{2}-2(-1)^{s+1},
\]
which is true by (\ref{eq:formula_to_prove the_case i=00003D2}).

Finally, we easily check that the formula for $c_{0}$ and $c_{2}$
holds.\end{proof}
\begin{cor}
\label{cor:properties of u_k(j), v_k(j)}\renewcommand{\labelenumi}{(\roman{enumi})}
For all $k,j\geqslant1$, we have:
\begin{enumerate}
\item If we abbreviate $u:=u_{k}(j)=:\frac{p}{q}$, then $q^{2}\left(u^{2}-6u+1\right)=j^{2}+6j+1$, \item If $v:=v_{k}(j)=\frac{p}{q}$, then $q^{2}\left(v^{2}-6v+1\right)=j^{2}-4j-4$, \item If we denote $\frac{p}{q}:=u_{k}(2)$, then $p^{2}-6pq+q^{2}-16=1$, \\ and if $\frac{p}{q}:=u_{k}(3)$, then $p^{2}-6pq+q^{2}-12=16$, \item If we denote $\frac{p}{q}:=v_{k}(6)$, then $5p^{2}-30pq+5q^{2}-32=8$,\\ and if $\frac{p}{q}:=v_{k}(7)$, then $3p^{2}-18pq+3q^{2}-28=23$. \item $P_{2m-k}=(-1)^{k+1}\left(P_{k}H_{2m}-H_{k}P_{2m}\right)$ for all $m\geqslant0$ and $k\leqslant2m$.
\end{enumerate}%
\end{cor}
\begin{proof}
By Lemma \ref{lem:check three values} we only have to check the first
four identities for three values of~$k$. It is easy to see that
they are true for $k=1,2,3$. Similarly, it suffices to check the
last identity for three even and three odd values.\end{proof}
\begin{lem}
\label{lem:phi>psi}\renewcommand{\labelenumi}{(\roman{enumi})}
Set $\varphi(a):=\frac{a+1}{4}$ and $\psi(a)=\sqrt{\frac{a}{2}}$. Then
\begin{enumerate} \item $\varphi\left(u_{k}(j+1)\right)>\psi\left(u_{k}(j)\right)$ for all $k,j\geqslant1$,
\item $\varphi\left(v_{k}(j)\right)>\psi\left(v_{k}(j+1)\right)$ for all $k\geqslant1,j\geqslant6$.
\end{enumerate} %
\end{lem}
\begin{proof}
(i) Abbreviate $u:=u_{k}(j+1)$ and $u':=u_{k}(j)$. Due to Corollary~\ref{cor:properties of u_k(j), v_k(j)}\,(i)
we have
\[
u^{2}=\frac{(j+1)^{2}+6(j+1)+1}{q^{2}}+6u-1.
\]
We have to prove that $\frac{u+1}{4}>\sqrt{\frac{u'}{2}}$ which is
equivalent to
\[
u^{2}+2u+1>8u',
\]
which becomes
\[
\frac{(j+1)^{2}+6(j+1)+1}{q^{2}}+8u>8u',
\]
and finally
\begin{equation}
8(u'-u)q^{2}<(j+1)^{2}+6(j+1)+1.\label{eq:8(u'-u)q^2=00003D(j+1)^2+6(j+1)+1}
\end{equation}
By Lemma \ref{lem:formulae for c_k, u_k(j), v_k(j)}\,(ii)
\[
\begin{aligned}u'-u & =\frac{\frac{1}{2}\left(jP_{2k+4}+P_{2k+2}\right)}{\frac{1}{2}\left(jP_{2k+2}+P_{2k}\right)}-\frac{\frac{1}{2}\left((j+1)P_{2k+4}+P_{2k+2}\right)}{\frac{1}{2}\left((j+1)P_{2k+2}+P_{2k}\right)}\\
 & =\frac{(j+1)P_{2k+2}^{2}-(j+1)P_{2k}P_{2k+4}+jP_{2k+4}P_{2k}-jP_{2k+2}^{2}}{\left((j+1)P_{2k+2}+P_{2k}\right)\left(jP_{2k+2}+P_{2k}\right)}\\
 & =\frac{P_{2k+2}^{2}-P_{2k}P_{2k+4}}{\left((j+1)P_{2k+2}+P_{2k}\right)\left(jP_{2k+2}+P_{2k}\right)}\\
 & =\frac{4}{\left((j+1)P_{2k+2}+P_{2k}\right)\left(jP_{2k+2}+P_{2k}\right)}
\end{aligned}
\]
since $P_{2k+2}^{2}-P_{2k}P_{2k+4}=4$ for all $k\geqslant1$ by Lemma
\ref{lem:check three values}. Inserting this in (\ref{eq:8(u'-u)q^2=00003D(j+1)^2+6(j+1)+1})
gives 
\[
\frac{8\left((j+1)P_{2k+2}+P_{2k}\right)}{jP_{2k+2}+P_{2k}}<(j+1)^{2}+6(j+1)+1,
\]
since $q^{2}=\frac{1}{4}\left((j+1)P_{2k+2}+P_{2k}\right)^{2}$. This
inequality is now true for all $k,j\geqslant1$ since the left hand
side is smaller than 16, and the right hand side is bigger than 17.\medskip{}

(ii) Abbreviate $v:=v_{k}(j)$ and $v':=v_{k}(j+1)$. Due to Corollary
\ref{cor:properties of u_k(j), v_k(j)}\,(ii), we have
\[
v^{2}=\frac{j{}^{2}-4j-4}{q^{2}}+6v-1.
\]
We have to prove that $\frac{v+1}{4}>\sqrt{\frac{v'}{2}}$ which is
equivalent to
\[
v^{2}+2v+1>8v',
\]
which becomes
\[
\frac{j^{2}-4j-4}{q^{2}}+8v>8v',
\]
and finally
\begin{equation}
8(v'-v)q^{2}<j^{2}-4j-4.\label{eq:8(v'-v)q^2=00003Dj^2-4j-4}
\end{equation}
By Lemma \ref{lem:formulae for c_k, u_k(j), v_k(j)}\,(iii)
\[
\begin{aligned}v'-v & =\frac{\frac{1}{2}jP_{2k+2}+P_{2k+1}}{\frac{1}{2}jP_{2k}+P_{2k-1}}-\frac{\frac{1}{2}(j+1)P_{2k+2}+P_{2k+1}}{\frac{1}{2}(j+1)P_{2k}+P_{2k-1}}\\
 & =\frac{1}{2}\frac{jP_{2k-1}P_{2k+2}+(j+1)P_{2k}P_{2k+1}-jP_{2k}P_{2k+1}-(j+1)P_{2k-1}P_{2k+2}}{\left(\frac{1}{2}jP_{2k}+P_{2k-1}\right)\left(\frac{1}{2}(j+1)P_{2k}+P_{2k-1}\right)}\\
 & =\frac{1}{2}\frac{P_{2k}P_{2k+1}-P_{2k-1}P_{2k+2}}{\left(\frac{1}{2}jP_{2k}+P_{2k-1}\right)\left(\frac{1}{2}(j+1)P_{2k}+P_{2k-1}\right)}\\
 & =\frac{-1}{\left(\frac{1}{2}jP_{2k}+P_{2k-1}\right)\left(\frac{1}{2}(j+1)P_{2k}+P_{2k-1}\right)}
\end{aligned}
\]
since $P_{2k}P_{2k+1}-P_{2k-1}P_{2k+2}=-2$ for all $k\geqslant1$
by Lemma \ref{lem:check three values}. Inserting this into (\ref{eq:8(v'-v)q^2=00003Dj^2-4j-4})
gives 
\[
-\frac{8\left(\frac{1}{2}jP_{2k}+P_{2k-1}\right)}{\frac{1}{2}(j+1)P_{2k}+P_{2k-1}}<j^{2}-4j-4,
\]
since $q^{2}=\left(\frac{1}{2}jP_{2k}+P_{2k-1}\right)^{2}$. This
inequality is now true for all $k\geqslant1,j\geqslant6$ since the
left hand side is negative, and the right hand side is positive.\end{proof}
\begin{defn}
A point $a\in\left[\sigma^{2},6\right]$ is said to be \emph{regular}
if for all $(d,e;m)\in\mathcal{E}$ such that $l(m)=l(a)$, it holds
that $\mu(d,e;m)(a)\leqslant\frac{a+1}{4}$.\end{defn}
\begin{prop}
\label{prop:if points are regular then c(a)=00003D(a+1)/4}Assume
that the points\medskip{}

$c_{2k-1}$ for all $k\geqslant1$,

$u_{k}(j)$ for all $k\geqslant1,j\geqslant2$,

$v_{k}(j)$ for all $k\geqslant1,j\geqslant6$\medskip{}

\noindent are regular. Then $c(a)=\frac{a+1}{4}$ on $\left[\sigma^{2},6\right]$.\end{prop}
\begin{proof}
Assume by contradiction that $c(a_{0})>\frac{a_{0}+1}{4}$ for some
$a_{0}\in\left[\sigma^{2},6\right]$. Since for all $a\in\left]\sigma^{2},6\right]$,
we have $c(a)\geqslant\frac{a+1}{4}>\sqrt{\frac{a}{2}}$, the function
$c(a)$ is piecewise linear on $\left]\sigma^{2},6\right]$ by Corollary
\ref{cor:set of mu st mu=00003Dc is finite}. Let $S\subset\left]\sigma^{2},6\right[$
be the set of non-smooth points of $c$ on $\left]\sigma^{2},6\right[$.
Decompose this set as $S=S_{+}\cup S_{-}$, where $S_{+}$ (respectively
$S_{-}$) consists of the points $s\in S$ near which $c$ is concave
(respectively convex). Since $c(s)>\frac{s+1}{4}$ for all for all
$s\in S_{+}$, the biggest point of $S$ is in $S_{-}$ because $c(a)=\frac{a+1}{4}$
for $a\in\left[5\frac{12}{13},6\right]$ by Proposition \ref{prop:c(a)=00003D(a+1)/4 on [5+12/13,6]}.
And since $c\left(\sigma^{2}\right)=\frac{\sigma^{2}+1}{4}$, it follows
that the set $S_{+}$ is non-empty. Let $s_{0}=\max S_{+}$. Then
$s_{0}\in\left]\sigma^{2},6\right[$. By Corollary \ref{cor:set of mu st mu=00003Dc is finite}\,(i)
there exists $(d,e;m)\in\mathcal{E}$ and $\varepsilon>0$ such that
\begin{equation}
c(z)=\mu(d,e;m)(z)\label{eq:c(z)=00003Dmu(d,e;m)(z)}
\end{equation}
on $\left[s_{0},s_{0}+\varepsilon\right[$. Abbreviate $\mu(z):=\mu(d,e;m)(z)$.
Then, $\mu\left(s_{0}\right)=c\left(s_{0}\right)>\frac{s_{0}+1}{4}>\sqrt{\frac{s_{0}}{2}}$.
Let $I$ be the maximal open interval containing $s_{0}$ on which
$\mu(z)>\sqrt{\frac{z}{2}}$ for all $z\in I$. By Lemma \ref{lem:there exists a unique a_0 center of the class},
there exists a unique $s'\in I$ with $l(m)=l(s')$, and $l(m)<l(z)$
for all other $z\in I$. Moreover, by Proposition~\ref{prop:structure of the obstructions mu},
the constraint $\mu(z)$ is given by two linear functions on $I$:
\[
\mu(z)=\left\{ \begin{array}{cl}
\alpha+\beta z & \textrm{if }z\leqslant s',z\in I,\\
\alpha'+\beta'z & \textrm{if }z\geqslant s',z\in I.
\end{array}\right.
\]
Thus, $s'$ is the only non-smooth point of $\mu$ on $I$. But since
$s_{0}\in S_{+}$ and $\mu\leqslant c$, $s_{0}$ is also a non-smooth
point of $\mu$, and so $s'=s_{0}$. Now, since $c$ is nondecreasing
and by (\ref{eq:c(z)=00003Dmu(d,e;m)(z)}), we see that $\beta'\geqslant0$. 

Let $k\geqslant1$ be such that $s_{0}\in\left[c_{2k+1},c_{2k-1}\right]$.
Since $c_{2k+1}$ and $c_{2k-1}$ are regular by assumption, we have
$s_{0}\in\left]c_{2k+1},c_{2k-1}\right[$. Notice that $u_{k}(j)\rightarrow c_{2k+1}$
and $v_{k}(j)\rightarrow c_{2k-1}$ as $j\rightarrow\infty$. Let
$u_{-},u_{+}$ be the two points from the sequence
\[
c_{2k+1}<\ldots<u_{k}(2)<u_{k}(1)=v_{k}(6)<v_{k}(7)<\ldots<c_{2k-1}
\]
of Corollary \ref{cor:c_2k+1 < ... < u_k(2) ...}\,(ii) such that
$s_{0}\in\left[u_{-},u_{+}\right]$. Since $u_{-}$ and $u_{+}$ are
regular by assumption, we must have $s_{0}\in\left]u_{-},u_{+}\right[$.
But then Lemma \ref{lem:phi>psi} shows that $\mu(s_{0})>\frac{s_{0}+1}{4}>\frac{u_{-}+1}{4}=\varphi\left(u_{-}\right)>\psi\left(u_{+}\right)=\sqrt{\frac{u_{+}}{2}}$
. And since $\beta'\geqslant0$, we find that $\mu\left(u_{+}\right)\geqslant\mu\left(s_{0}\right)>\sqrt{\frac{u_{+}}{2}}$,
and thus $u_{+}\in I$, and so $l\left(u_{+}\right)>l\left(s_{0}\right)$.
But, for all $z\in\left]u_{-},u_{+}\right[$ we have $l(z)>l\left(u_{-}\right)$
and $l(z)>l\left(u_{+}\right)$. In particular, $l\left(s_{0}\right)>l\left(u_{+}\right)$,
which is a contradiction.\end{proof}
\begin{lem}
\label{lem:u_k(j) regular}The points $u_{k}(j)$ with $k\geqslant1,j\geqslant2$
are regular.\end{lem}
\begin{proof}
Abbreviate $u:=u_{k}(j)=:\frac{p}{q}$. Let us first prove that a
class of the form $(d+\frac{1}{2},d-\frac{1}{2};m)$ with $l(m)=l(u)$
cannot give an obstruction bigger than~$\frac{u+1}{4}$. Suppose
by contradiction that there is such a class $(d+\frac{1}{2},d-\frac{1}{2};m)$.
By Proposition \ref{prop:first estimates of d and lambda^2}\,(i)
and Corollary \ref{cor:properties of u_k(j), v_k(j)}\,(i),
\[
\frac{\sqrt{2}d}{q\sqrt{u}}<\frac{2}{q\sqrt{u^{2}-6u+1}}=\frac{2}{\sqrt{j^{2}+6j+1}},
\]
which is smaller than $1$ for all $j\geqslant2$. Since $l(m)=l(u)$,
$m_{i}\geqslant1$ for all $i$. Thus
\[
\sum\varepsilon_{i}^{2}\geqslant j\left(1-\frac{\sqrt{2}d}{q\sqrt{u}}\right)^{2}>j\left(1-\frac{2}{\sqrt{j^{2}+6j+1}}\right)^{2}=:s(j).
\]
Now, since $s(j)$ is increasing for $j\geqslant2$ and $s(2)>\frac{1}{2}$,
we have $\mu(d+\frac{1}{2},d-\frac{1}{2};m)(u)\leqslant\sqrt{\frac{u}{2}}<\frac{u+1}{4}$
by Lemma \ref{lem:first properties of mu}\,(iii). The lemma is thus
proven for a class of the form $(d+\frac{1}{2},d-\frac{1}{2};m)$.

Let us now prove it for a class of the form $(d,d;m)$. Suppose that
there exists a class $(d,d;m)\in\mathcal{E}$ with $l(m)=l(u)$ such
that $\mu(d,d;m)(u)>\frac{u+1}{4}$. By Proposition \ref{prop:first estimates of d and lambda^2}\,(i)
and Corollary \ref{cor:properties of u_k(j), v_k(j)}\,(i),
\[
\frac{\sqrt{2}d}{q\sqrt{u}}<\frac{2\sqrt{2}}{q\sqrt{u^{2}-6u+1}}=\frac{2\sqrt{2}}{\sqrt{j^{2}+6j+1}},
\]
which is smaller than $1$ for all $j\geqslant2$. Since $l(m)=l(u)$,
$m_{i}\geqslant1$ for all $i$. Thus
\[
\sum\varepsilon_{i}^{2}\geqslant j\left(1-\frac{\sqrt{2}d}{q\sqrt{u}}\right)^{2}>j\left(1-\frac{2\sqrt{2}}{\sqrt{j^{2}+6j+1}}\right)^{2}=:s(j).
\]
Now, since $s(j)$ is increasing for $j\geqslant2$ and $s(4)>1$,
we have $\mu(d,d;m)(u)\leqslant\sqrt{\frac{u}{2}}<\frac{u+1}{4}$
by Lemma \ref{lem:first properties of mu}\,(iii). So the lemma is
proven for $j\geqslant4$.

It remains to show the lemma for $j=2,3$. By Proposition \ref{prop:first estimates of d and lambda^2}\negthinspace{}(ii),
\[
\sum\varepsilon_{i}^{2}=1-\lambda^{2}<1-d^{2}\sqrt{\frac{2}{u}}y(u).
\]
Thus
\[
\begin{aligned}0 & =\sum\varepsilon_{i}^{2}+\lambda^{2}-1>j\left(1-\frac{\sqrt{2}d}{q\sqrt{u}}\right)^{2}+d^{2}\sqrt{\frac{2}{u}}y(u)-1\\
 & =\left(\frac{2j}{q^{2}u}+\sqrt{\frac{2}{u}}y(u)\right)d^{2}-\frac{2\sqrt{2}j}{q\sqrt{u}}d+j-1=:f(d).
\end{aligned}
\]
To obtain a contradiction, we need to show that $f(d)\geqslant0$
for all $d\geqslant1$. Since $y(u)>0$ for $u>\sigma^{2}$, it is
sufficient to show that the discriminant of $f(d)$ is negative, that
is:
\[
\frac{8j^{2}}{q^{2}u}-4(j-1)\left(\frac{2j}{q^{2}u}+\sqrt{\frac{2}{u}}y(u)\right)\leqslant0,
\]
which is equivalent to
\[
\frac{j}{q^{2}}+2(j-1)u\leqslant(j-1)\sqrt{\frac{u}{2}}(u+1).
\]
Taking squares and using $u=\frac{p}{q}$, we get
\begin{equation}
0\leqslant-2j^{2}+pq\left((j-1)^{2}p^{2}-6(j-1)^{2}pq+(j-1)^{2}q^{2}-8j(j-1)\right).\label{eq:0<=00003D-2j^2+pq(...)}
\end{equation}
For $j=2$, this gives
\[
0\leqslant-8+pq\left(p^{2}-6pq+q^{2}-16\right),
\]
and this inequality is true since $pq\geqslant26$ and $p^{2}-6pq+q^{2}-16=1$
by Corollary~\ref{cor:properties of u_k(j), v_k(j)}\,(iii). For
$j=3$, we get
\[
0\leqslant-9+pq\left(2p^{2}-12pq+2q^{2}-24\right),
\]
and this inequality is also true since $pq\geqslant57$ and $2p^{2}-12pq+2q^{2}-24=32$
by Corollary \ref{cor:properties of u_k(j), v_k(j)}\,(iii). This
concludes the proof.\end{proof}
\begin{lem}
\label{lem:v_k(j) regular}The points $v_{k}(j)$ with $k\geqslant1,j\geqslant6$
are regular.\end{lem}
\begin{proof}
Abbreviate $v:=v_{k}(j)=:\frac{p}{q}$. Let us first prove that a
class of the form $(d+\frac{1}{2},d-\frac{1}{2};m)$ with $l(m)=l(v)$
cannot give an obstruction bigger than~$\frac{v+1}{4}$. Suppose
by contradiction that there is such a class $(d+\frac{1}{2},d-\frac{1}{2};m)$.
By Proposition \ref{prop:first estimates of d and lambda^2}\,(i)
and Corollary \ref{cor:properties of u_k(j), v_k(j)}\,(ii),
\[
\frac{\sqrt{2}d}{q\sqrt{u}}<\frac{2}{q\sqrt{v^{2}-6v+1}}=\frac{2}{\sqrt{j^{2}-4j-4}},
\]
which is smaller than $1$ for all $j\geqslant6$. Since $l(m)=l(v)$,
$m_{i}\geqslant1$ for all $i$. Thus
\[
\sum\varepsilon_{i}^{2}\geqslant j\left(1-\frac{\sqrt{2}d}{q\sqrt{u}}\right)^{2}>j\left(1-\frac{2}{\sqrt{j^{2}-4j-4}}\right)^{2}=:s(j).
\]
Now, since $s(j)$ is increasing for $j\geqslant6$ and $s(6)>\frac{1}{2}$,
we have $\mu(d+\frac{1}{2},d-\frac{1}{2};m)(v)\leqslant\sqrt{\frac{v}{2}}<\frac{v+1}{4}$
by Lemma \ref{lem:first properties of mu}\,(iii). The lemma is thus
proven for a class of the form $(d+\frac{1}{2},d-\frac{1}{2};m)$.

Let us now prove the lemma for a class of the form $(d,d;m)$. Suppose
that there exists a class $(d,d;m)\in\mathcal{E}$ with $l(m)=l(v)$
such that $\mu(d,d;m)(v)>\frac{v+1}{4}$. By Proposition \ref{prop:first estimates of d and lambda^2}\,(i)
and Corollary \ref{cor:properties of u_k(j), v_k(j)}\,(ii),
\[
\frac{\sqrt{2}d}{q\sqrt{u}}<\frac{2\sqrt{2}}{q\sqrt{v^{2}-6v+1}}=\frac{2\sqrt{2}}{\sqrt{j^{2}-4j-4}},
\]
which is smaller than $1$ for all $j\geqslant6$. Since $l(m)=l(u)$,
$m_{i}\geqslant1$ for all $i$. Thus
\[
\sum\varepsilon_{i}^{2}\geqslant j\left(1-\frac{\sqrt{2}d}{q\sqrt{u}}\right)^{2}>j\left(1-\frac{2\sqrt{2}}{\sqrt{j^{2}-4j-4}}\right)^{2}=:s(j).
\]
Now, since $s(j)$ is increasing for $j\geqslant6$ and $s(8)>1$,
we have $\mu(d,d;m)(v)\leqslant\sqrt{\frac{v}{2}}<\frac{v+1}{4}$
by Lemma \ref{lem:first properties of mu}\,(iii). So the lemma is
proven for $j\geqslant8$.

It remains to show it for $j=6,7$. If $j=6$, the same arguments
as in the proof of Lemma \ref{lem:u_k(j) regular} in the cases $j=2,3$
show that a point $v_{k}(6)=:\frac{p}{q}$ is regular if (\ref{eq:0<=00003D-2j^2+pq(...)})
is satisfied for $j=6$, that is if and only if
\[
0\leqslant-72+pq\left(25p^{2}-150pq+25q^{2}-160\right).
\]
This inequality is true since $pq\geqslant287$ and $25p^{2}-150pq+25q^{2}-160=40$
by Corollary \ref{cor:properties of u_k(j), v_k(j)}\,(iv). Similarly,
a point $v_{k}(7)=:\frac{p}{q}$ is regular if (\ref{eq:0<=00003D-2j^2+pq(...)})
is satisfied for $j=7$, that is if and only if

\[
0\leqslant-98+pq\left(36p^{2}-216pq+36q^{2}-336\right),
\]
which is true since $pq\geqslant376$ and $36p^{2}-216pq+36q^{2}-336=276$
by Corollary~\ref{cor:properties of u_k(j), v_k(j)}\,(iv). This
completes the proof.\end{proof}
\begin{defn}
\label{def:definition of b_k(i) and E(b_k(i))}Define for $k\geqslant1$
and $i\geqslant0$ the points
\[
b_{k}(i):=v_{k}(2+2i)=\left[5;\left\{ 1,4\right\} ^{\times(k-1)},1,2+2i\right].
\]
In particular, for all $k\geqslant1$, $u_{k}(1)=v_{k}(6)=b_{k}(2)$.
Let $\frac{p}{q}:=b_{k}(i)$ written in lowest terms. We will now
associate to every $b_{k}(i)$ a class $E\left(b_{k}(i)\right)\in\mathcal{E}$
for which we will prove that it gives the constraint at $b_{k}(i)$.
We distinguish the cases $i$ even and $i$ odd.

If $i=2j$, set 
\[
m_{k}(2j):=qw\left(b_{k}(2j)\right)
\]
 but with the last block $\left(1^{\times(4j+2)}\right)$ being replaced
by the block $\left(j+1,j,1^{\times(2j+1)}\right)$, and $d_{k}(2j):=q\frac{1+b_{k}(2j)}{4}=\frac{1}{4}(p+q)$.
Then define the class
\[
E\left(b_{k}(2j)\right):=\left(d_{k}(2j),d_{k}(2j);m_{k}(2j)\right).
\]

If $i=2j+1$, set 
\[
m_{k}(2j+1):=qw\left(b_{k}(2j+1)\right)
\]
but with the last block $\left(1^{\times(4j+4)}\right)$ being replaced
by the block $\left((j+1)^{\times2},1^{\times(2j+2)}\right)$, and
$d_{k}(2j+1):=q\frac{1+b_{k}(2j+1)}{4}=\frac{1}{4}(p+q)$. Then define
the class
\[
E\left(b_{k}(2j+1)\right):=\left(d_{k}(2j+1)+\frac{1}{2},d_{k}(2j+1)-\frac{1}{2};m_{k}(2j+1)\right).
\]

\end{defn}
We will now prove that the classes $E\left(b_{k}(i)\right)$ belong
to $\mathcal{E}$.
\begin{lem}
\label{lem:E(b_k(i)) satisfy Diophantine conditions}For all $k\geqslant1,i\geqslant0$
the classes $E\left(b_{k}(i)\right)$ satisfy the Diophantine conditions
of Proposition \ref{prop:characterization of E_M}\,(i)\emph{.}\end{lem}
\begin{proof}
We first treat the case $i=2j$. Abbreviate $b:=b_{k}(2j)=:\frac{p}{q}$,
$m:=m_{k}(2j)$, $w:=w\left(b_{k}(2j)\right)$ and $d:=d_{k}(2j)$.
By Lemma \ref{lem:weight expansion}\,(iii), 
\[
\sum_{l}m_{l}=q\sum_{l}w{}_{l}=q\left(b+1-\frac{1}{q}\right)=4\left(q\frac{1+b}{4}\right)-1=4d-1,
\]
which proves that the first equation holds. For the second equation,
we have
\[
\sum_{l}m_{l}^{2}=q^{2}\left(\sum_{l}w_{l}^{2}\right)-2j-1+(j+1)^{2}+j^{2}=q^{2}b+2j^{2}=pq+2j^{2}.
\]
On the other hand,
\[
2d^{2}+1=2\left(q^{2}\frac{(1+b)^{2}}{16}\right)+1=\frac{1}{8}(p+q)^{2}+1.
\]
By Lemma \ref{lem:check three values}, it suffices that this equals
$pq+2j^{2}$ for three small values of~$j$, which is the case.

Similarly in the case $i=2j+1$, abbreviate $b:=b_{k}(2j+1)=:\frac{p}{q}$,
$m:=m_{k}(2j+1)$, $w:=w\left(b_{k}(2j+1)\right)$ and $d:=d_{k}(2j+1)$.
We then have
\[
\sum_{l}m_{l}=q\sum_{l}w{}_{l}=q\left(b+1-\frac{1}{q}\right)=4\left(q\frac{1+b}{4}\right)-1=4d-1.
\]
For the second equation we have
\[
\sum_{l}m_{l}^{2}=q^{2}\left(\sum_{l}w_{l}^{2}\right)-2j-2+2(j+1)^{2}=q^{2}b+2j^{2}+2j=pq+2j^{2}+2j.
\]
On the other hand,
\[
2d^{2}+\frac{1}{2}=2\left(q^{2}\frac{(1+b)^{2}}{16}\right)+\frac{1}{2}=\frac{1}{8}(p+q)^{2}+\frac{1}{2}.
\]
Now use again Lemma \ref{lem:check three values} to check that this
equals $pq+2j^{2}+2j$ for all $j\geqslant0$.\end{proof}
\begin{lem}
\label{lem:E(b_k(2j)) reduce}The classes $E\left(b_{k}(2j)\right)$
reduce to $\left(0;-1\right)$ for all $k\geqslant1,j\geqslant0$.\end{lem}
\begin{proof}
The proof is by induction over $k$. For the initial step $k=1$,
we have $b_{1}(2j)=\frac{24j+17}{4j+3}=:\frac{p}{q}$, $qw\left(b_{1}(2j)\right)=\left((4j+3)^{\times5},4j+2,1^{\times(4j+2)}\right)$,
and $d_{1}(2j)=7j+5$. Thus
\[
E\left(b_{1}(2j)\right)=\left(7j+5,7j+5;(4j+3)^{\times5},4j+2,j+1,j,1^{\times(2j+1)}\right).
\]
We now show that this class reduces to $\left(0;-1\right)$.

\medskip{}
$\left(7j+5,7j+5;(4j+3)^{\times5},4j+2,j+1,j,1^{\times(2j+1)}\right)\overset{{\scriptstyle \varphi_{*}}}{\longmapsto}$

\medskip{}
$\left(10j+7;(4j+3)^{\times4},4j+2,(3j+2)^{\times2},j+1,j,1^{\times(2j+1)}\right);$

\medskip{}
$\left(8j+5;4j+3,4j+2,(3j+2)^{\times2},(2j+1)^{\times3},j+1,j,1^{\times(2j+1)}\right);$

\medskip{}
$\left(5j+3;3j+2,(2j+1)^{\times3},(j+1)^{\times2},j^{\times2},1^{\times(2j+1)}\right);$

\medskip{}
$\left(3j+2;2j+1,(j+1)^{\times3},j^{\times2},1^{\times(2j+1)}\right);$

\medskip{}
$\left(2j+1;j+1,j^{\times3},1^{\times(2j+1)}\right);$

\medskip{}
$\left(j+1;j,1^{\times(2j+2)}\right).$

\medskip{}
This class reduces to $\left(0;-1\right)$ in $j+1$ steps since a
class of the type $\left(s+1;s,1^{\times t}\right)$ for $s\geqslant1,t\geqslant2$
reduces to $\left(s;s-1,1^{\times(t-2)}\right)$ by a standard Cremona
move.

We turn now to the general case. We will freely use the definitions
of the Pell numbers $P_{n}$ and the Half companion Pell numbers $H_{n}$
given in Definition \ref{def:pell numbers} and the fact that for
all $n\geqslant0$, $H_{n}=P_{n}+P_{n-1}$. Suppose that the class
$E\left(b_{k-1}(2j)\right)$ reduces to $(0;-1)$ and let us show
that the class $E\left(b_{k}(2j)\right)$ also reduces to $(0;-1)$.
We have
\[
\begin{aligned}b_{k}(2j) & =\frac{2jP_{2k+2}+H_{2k+2}}{2jP_{2k}+H_{2k}},\\
d_{k}(2j) & =jH_{2k+1}+P_{2k+1}.
\end{aligned}
\]
The first terms of the class $E\left(b_{k}(2j)\right)$ are

\medskip{}

$\left(jH_{2k+1}+P_{2k+1},jH_{2k+1}+P_{2k+1};\left(2jP_{2k}+H_{2k}\right)^{\times5},4jP_{2k-1}+H_{2k-1},\right.$

$\hphantom{\qquad}\left.\left(2jP_{2k-2}+H_{2k-2}\right)^{\times4},4jP_{2k-3}+2H_{2k-3},(*)\right),$

\medskip{}
where $(*)$ stands for all the next terms. The image of $E\left(b_{k}(2j)\right)$
under~$\varphi_{*}$ is\medskip{}

$\left(2jP_{2k+1}+H_{2k+1};\left(2jP_{2k}+H_{2k}\right)^{\times4},4jP_{2k-1}+2H_{2k-1},\left(jH_{2k}+P_{2k}\right)^{\times2}\right.$

$\hphantom{\qquad}\left.\left(2jP_{2k-2}+H_{2k-2}\right)^{\times4},4jP_{2k-3}+2H_{2k-3},(*)\right).$

\medskip{}
To finish the proof, we will show that $\varphi_{*}\left(E\left(b_{k}(2j)\right)\right)$
reduces to $\varphi_{*}\left(E\left(b_{k-1}(2j)\right)\right)$ in
four steps.

\medskip{}

$\left(2jP_{2k+1}+H_{2k+1};\left(2jP_{2k}+H_{2k}\right)^{\times4},4jP_{2k-1}+2H_{2k-1},\left(jH_{2k}+P_{2k}\right)^{\times2},\right.$

$\hphantom{\qquad}\left.\left(2jP_{2k-2}+H_{2k-2}\right)^{\times4},4jP_{2k-3}+2H_{2k-3},(*)\right);$

\medskip{}

$\left(2j\left(H_{2k}+P_{2k-1}\right)+3H_{2k-1}+2P_{2k-1};2jP_{2k}+H_{2k},4jP_{2k-1}+2H_{2k-1},\right.$

$\hphantom{\qquad}\left(jH_{2k}+P_{2k}\right)^{\times2},\left(2jP_{2k-1}+H_{2k-1}\right)^{\times3},\left(2jP_{2k-2}+H_{2k-2}\right)^{\times4},$

$\hphantom{\qquad}\left.4jP_{2k-3}+2H_{2k-3},(*)\right);$

\medskip{}

$\left(j\left(3H_{2k}-2P_{2k}\right)+2H_{2k-1}+P_{2k-1};jH_{2k}+P_{2k},\left(2jP_{2k-1}+H_{2k-1}\right)^{\times3},\right.$

$\hphantom{\qquad}jH_{2k-1}+P_{2k-1},\left(2jP_{2k-2}+H_{2k-2}\right)^{\times4},4jP_{2k-3}+2H_{2k-3},$

$\hphantom{\qquad}\left.jH_{2k-2}+P_{2k-2},(*)\right);$

\medskip{}

$\left(jH_{2k}+P_{2k};2jP_{2k-1}+H_{2k-1},\left(jH_{2k-1}+P_{2k-1}\right)^{\times2},\left(2jP_{2k-2}+H_{2k-2}\right)^{\times4},\right.$

$\hphantom{\qquad}\left.4jP_{2k-3}+2H_{2k-3},jH_{2k-2}+P_{2k-2},(*)\right);$

\medskip{}

$\left(2jP_{2k-1}+H_{2k-1};\left(2jP_{2k-2}+H_{2k-2}\right)^{\times4},4jP_{2k-3}+2H_{2k-3},\right.$

$\hphantom{\qquad}\left.\left(jH_{2k-2}+P_{2k-2}\right)^{\times2},(*)\right).$

\medskip{}
It is important to note that $(*)$ was left invariant during the
whole reduction process. So the last class is precisely $\varphi_{*}\left(E\left(b_{k-1}(2j)\right)\right)$.\end{proof}
\begin{lem}
\label{lem:E(b_k(2j+1)) reduce}The classes $E\left(b_{k}(2j+1)\right)$
reduce to $\left(0;-1\right)$ for all $k\geqslant1,j\geqslant0$.\end{lem}
\begin{proof}
The proof is again by induction over $k$. For the initial step $k=1$,
we have $b_{1}(2j+1)=\frac{24j+29}{4j+5}=:\frac{p}{q}$, $qw\left(b_{1}(2j+1)\right)=\left((4j+5)^{\times5},4j+4,1^{\times(4j+4)}\right)$,
$d_{1}(2j+1)=7j+\frac{17}{2}$. Thus
\[
E\left(b_{1}(2j+1)\right)=\left(7j+9,7j+8;(4j+5)^{\times5},4j+4,(j+1)^{2},1^{\times(2j+2)}\right).
\]
We show now that this class reduces to $\left(0;-1\right)$.

\medskip{}

$\left(7j+9,7j+8;(4j+5)^{\times5},4j+4,(j+1)^{2},1^{\times(2j+2)}\right)\overset{{\scriptstyle \varphi_{*}}}{\longmapsto}$

\medskip{}

$\left(10j+2;(4j+5)^{\times4},4j+4,3j+4,3j+3,(j+1)^{2},1^{\times(2j+2)}\right);$

\medskip{}

$\left(8j+9;4j+5,4j+4,3j+4,3j+3,(2j+2)^{\times3},(j+1)^{2},1^{\times(2j+2)}\right);$

\medskip{}

$\left(5j+5;3j+3,(2j+2)^{\times3},(j+1)^{\times3},j,1^{\times(2j+2)}\right);$

\medskip{}

$\left(3j+3;2j+2,(j+1)^{\times4},j,1^{\times(2j+2)}\right);$

\medskip{}

$\left(2j+2;(j+1)^{\times3},j,1^{\times(2j+2)}\right);$

\medskip{}

$\left(j+1;j,1^{\times(2j+2)}\right).$

\medskip{}
As seen before, this class reduces to $\left(0;-1\right)$ in $j+1$
steps.

We turn now to the general case. Suppose that the class $E\left(b_{k-1}(2j+1)\right)$
reduces to $(0;-1)$ and let us show that the class $E\left(b_{k}(2j+1)\right)$
also reduces to $(0;-1)$. We have
\[
b_{k}(2j+1)=\frac{2jP_{2k+2}+P_{2k+3}}{2jP_{2k}+P_{2k+1}},
\]
\[
d_{k}(2j+1)=jH_{2k+1}+\frac{1}{2}H_{2k+2}.
\]
The first terms of the class $E\left(b_{k}(2j+1)\right)$ are

\medskip{}

$\left(jH_{2k+1}+\frac{1}{2}H_{2k+2}+\frac{1}{2},jH_{2k+1}+\frac{1}{2}H_{2k+2}-\frac{1}{2};\left(2jP_{2k}+P_{2k+1}\right)^{\times5},\right.$

$\hphantom{\qquad}\left.4jP_{2k-1}+2P_{2k},\left(2jP_{2k-2}+P_{2k-1}\right)^{\times4},4jP_{2k-3}+2P_{2k-2},(*)\right).$

\medskip{}
The image of $E\left(b_{k}(2j+1)\right)$ under $\varphi_{*}$ is

\medskip{}

$\left(2jP_{2k+1}+P_{2k+2};\left(2jP_{2k}+P_{2k+1}\right)^{\times4},4jP_{2k-1}+2P_{2k},jH_{2k}+\frac{1}{2}H_{2k+1}+\frac{1}{2},\right.$

$\hphantom{\qquad}\left.jH_{2k}+\frac{1}{2}H_{2k+1}-\frac{1}{2},\left(2jP_{2k-2}+P_{2k-1}\right)^{\times4},4jP_{2k-3}+2P_{2k-2},(*)\right),$

\medskip{}
To finish the proof, we will show that $\varphi_{*}\left(E\left(b_{k}(2j+1)\right)\right)$
reduces to the vector $\varphi_{*}\left(E\left(b_{k-1}(2j+1)\right)\right)$
in four steps.

\medskip{}

$\left(2jP_{2k+1}+P_{2k+2};\left(2jP_{2k}+P_{2k+1}\right)^{\times4},4jP_{2k-1}+2P_{2k},jH_{2k}+\frac{1}{2}H_{2k+1}+\frac{1}{2},\right.$

$\hphantom{\qquad}\left.jH_{2k}+\frac{1}{2}H_{2k+1}-\frac{1}{2},\left(2jP_{2k-2}+P_{2k-1}\right)^{\times4},4jP_{2k-3}+2P_{2k-2},(*)\right);$

\medskip{}

$\left(j\left(2P_{2k}+4P_{2k-1}\right)+\left(P_{2k+1}+2P_{2k}\right);2jP_{2k}+P_{2k+1},4jP_{2k-1}+2P_{2k},\right.$

$\hphantom{\qquad}jH_{2k}+\frac{1}{2}H_{2k+1}+\frac{1}{2},jH_{2k}+\frac{1}{2}H_{2k+1}-\frac{1}{2},\left(2jP_{2k-1}+P_{2k}\right)^{\times3},$

$\hphantom{\qquad}\left.\left(2jP_{2k-2}+P_{2k-1}\right)^{\times4},4jP_{2k-3}+2P_{2k-2},(*)\right);$

\medskip{}

$\left(j\left(P_{2k}+3P_{2k-1}\right)+\left(\frac{1}{2}P_{2k+1}+\frac{3}{2}P_{2k}-\frac{1}{2}\right);jH_{2k}+\frac{1}{2}H_{2k+1}-\frac{1}{2},\right.$

$\hphantom{\qquad}\left(2jP_{2k-1}+P_{2k}\right)^{\times3},jH_{2k-1}+\frac{1}{2}H_{2k}-\frac{1}{2},\left(2jP_{2k-2}+P_{2k-1}\right)^{\times4},$

$\hphantom{\qquad}\left.4jP_{2k-3}+2P_{2k-2},jH_{2k-2}+\frac{1}{2}H_{2k-1}-\frac{1}{2},(*)\right);$

\medskip{}

$\left(jH_{2k}+\frac{1}{2}H_{2k+1}-\frac{1}{2};2jP_{2k-1}+P_{2k},\left(jH_{2k-1}+\frac{1}{2}H_{2k}-\frac{1}{2}\right)^{\times2},\right.$

$\hphantom{\qquad}\left.\left(2jP_{2k-2}+P_{2k-1}\right)^{\times4},4jP_{2k-3}+2P_{2k-2},jH_{2k-2}+\frac{1}{2}H_{2k-1}-\frac{1}{2},(*)\right);$

\medskip{}

$\left(2jP_{2k-1}-P_{2k};\left(2jP_{2k-2}+P_{2k-1}\right)^{\times4},4jP_{2k-3}+2P_{2k-2},\right.$

$\hphantom{\qquad}\left.jH_{2k-2}+\frac{1}{2}H_{2k-1}+\frac{1}{2},jH_{2k-2}+\frac{1}{2}H_{2k-1}-\frac{1}{2},(*)\right).$

\medskip{}
Since $(*)$ was left invariant during the whole reduction process,
the last class is precisely $\varphi_{*}\left(E\left(b_{k-1}(2j+1)\right)\right)$.\end{proof}
\begin{prop}
For all $k\geqslant1,i\geqslant0$, we have $E\left(b_{k}(i)\right)\in\mathcal{E}$.\end{prop}
\begin{proof}
We have to show that the classes $E\left(b_{k}(i)\right)$ satisfy
the Diophantine conditions of Proposition \ref{prop:characterization of E_M},
which we have done in Lemma \ref{lem:E(b_k(i)) satisfy Diophantine conditions},
and that they reduce to $(0;-1)$ by Cremona moves, which we have
done for $i$ even in Lemma \ref{lem:E(b_k(2j)) reduce} and for $i$
odd in Lemma \ref{lem:E(b_k(2j+1)) reduce}. The proof is thus complete.\end{proof}
\begin{cor}
\label{cor:the classes E(beta_n) belong to E}For all $n\geqslant0$,
the classes $E\left(\beta_{n}\right)$ of Theorem \ref{thm:the classes E(alpha_n), E(beta_n) belong to E}
belong to $\mathcal{E}$.\end{cor}
\begin{proof}
Notice that by Lemma \ref{lem:formulae for c_k, u_k(j), v_k(j)},
for all $k\geqslant0$,
\[
\begin{aligned}b_{k}(0) & =v_{k}(2)=\frac{P_{2k+2}+P_{2k+1}}{P_{2k}+P_{2k-1}}=\frac{H_{2k+2}}{H_{2k}}=\beta_{2k},\\
b_{k}(1) & =v_{k}(4)=\frac{2P_{2k+2}+P_{2k+1}}{2P_{2k}+P_{2k-1}}=\frac{P_{2k+3}}{P_{2k+1}}=\beta_{2k+1}.
\end{aligned}
\]
Hence by Definition \ref{def:definition of b_k(i) and E(b_k(i))},
we see that for all $k\geqslant0$, $E\left(b_{k}(0)\right)=E\left(\beta_{2k}\right)$
and $E\left(b_{k}(1)\right)=E\left(\beta_{2k+1}\right)$. Thus all
the classes $E\left(\beta_{n}\right)$ belong indeed to $\mathcal{E}$.\end{proof}
\begin{cor}
\label{cor:c(b_k(i))=00003D(b_k(i)+1)/4}$c\left(b_{k}(i)\right)=\frac{b_{k}(i)+1}{4}$
for all $k\geqslant1$ and $i\geqslant2$.\end{cor}
\begin{proof}
Since $b_{k}(i)\in\left]\sigma^{2},6\right[$, we can write them as
$b_{k}(i)=5+x$ where $x\in\left[0,1\right[$. Now, $\left(2,2;2,1^{\times5}\right)\in\mathcal{E}$,
thus
\[
\mu\left(2,2;2,1^{\times5}\right)\left(b_{k}(i)\right)=\frac{6+x}{4}=\frac{b_{k}(i)+1}{4}.
\]
So $c\left(b_{k}(i)\right)\geqslant\mu\left(2,2;2,1^{\times5}\right)\left(b_{k}(i)\right)=\frac{b_{k}(i)+1}{4}$.

Let us show the converse inequality. Abbreviate $b:=b_{k}(i)=:\frac{p}{q}$
in lowest terms, $d:=d_{k}(i)$ and $m:=m_{k}(i)$. Then
\[
\mu\left(E(b)\right)(b)=\frac{\left\langle m,w(b)\right\rangle }{2d}=\frac{\left\langle qw(b),w(b)\right\rangle }{2d}=\frac{2b}{1+b}<\sqrt{\frac{b}{2}}<\frac{b+1}{4}
\]
since $b>\sigma^{2}$. Now if $(d',e';m')\in\mathcal{E}$ is a class
different from $E(b)$, we have by Proposition \ref{prop:characterization of E_M}\,(ii)
that $\left\langle m,m'\right\rangle \leqslant d(d'+e')$. Using the
definitions of $d$ and $m$ and the fact that $m$ is written in
decreasing order, we get
\[
q\frac{1+b}{4}(d'+e')\geqslant q\left\langle w(b),m'\right\rangle ,
\]
thus
\[
\mu(d',e';m')(b)=\frac{\left\langle m',w(b)\right\rangle }{d'+e'}\leqslant\frac{b+1}{4}.
\]
The proof is complete.\end{proof}
\begin{thm}
\label{thm:c(a) on [sigma^2,6]}$c(a)=\frac{a+1}{4}$ on $\left[\sigma^{2},6\right]$.\end{thm}
\begin{proof}
By Proposition \ref{prop:if points are regular then c(a)=00003D(a+1)/4}
it suffices to show that the points $c_{2k-1}$ for all $k\geqslant1$,
$u_{k}(j)$ for all $k\geqslant1,j\geqslant2$, and $v_{k}(j)$ for
all $k\geqslant1,j\geqslant6$ are regular. By Lemma \ref{lem:u_k(j) regular}
and Lemma \ref{lem:v_k(j) regular}, the points $u_{k}(j)$ and $v_{k}(j)$
are regular. Moreover, for all $k\geqslant1$, $v_{k}(j)\underset{j\rightarrow\infty}{\longrightarrow}c_{2k-1}$.
But by Corollary \ref{cor:c(b_k(i))=00003D(b_k(i)+1)/4}, $c\left(v_{k}(2+2i)\right)=c\left(b_{k}(i)\right)=\frac{b_{k}(i)+1}{4}$
for $i\geqslant2$. So, by continuity of $c$, we get that $c\left(c_{2k-1}\right)=\frac{c_{2k-1}+1}{4}$,
and the points $c_{2k-1}$ are thus regular. This completes the proof.
\end{proof}

\section{The interval $\left[6,8\right]$}

\subsection{Preliminaries}

We will use the fact that if $\left[l_{0};l_{1},...,l_{N}\right]$
is a continued fraction of a rational number $\frac{p}{q}$ and $\frac{p_{k}}{q_{k}}:=\left[l_{0};l_{1},...,l_{k}\right]$
is its $k$-th convergent written in lowest terms, then for any real
number $x$,
\[
\left[l_{0};l_{1},...,l_{k},x\right]=\frac{xp_{k}+p_{k-1}}{xq_{k}+q_{k-1}},
\]
written in lowest terms. In particular, $q_{k}=l_{k}q_{k-1}+q_{k-2}$.
It is then easy to see that if $L:=\sum_{i}l_{i}$, then 
\begin{equation}
q=q_{N}\geqslant L.\label{eq:q great or eq than L}
\end{equation}
Recall also that we defined the \emph{error vector} of a class $(d,e;m)$
at a point $a$ as the vector $\varepsilon:=\varepsilon\left((d,e;m),a\right)$
defined by the equation
\[
m=\frac{d+e}{\sqrt{2a}}w(a)+\varepsilon.
\]
Now set then $M:=l(a)=L+l_{0}$,
\[
\sigma:=\sum_{i>l_{0}}\varepsilon_{i}^{2},
\]
and 
\[
\sigma':=\sum_{i=l_{0}+1}^{M-l_{N}}\varepsilon_{i}^{2}<\sigma.
\]
Then by Lemma \ref{lem:first properties of mu}.3, $\sigma<1$ for
a class of the form $(d,d;m)$ and $\sigma<\frac{1}{2}$ for a class
of the form $(d+\frac{1}{2},d-\frac{1}{2};m)$.
\begin{lem}
\label{lem:sup born on d}\renewcommand{\labelenumi}{(\roman{enumi})}
Let $(d,e;m)\in\mathcal{E}$ be a class such that there exists $a=:\frac{p}{q}\in\left]\sigma^{2},8\right[$ with $l(a)=l(m)$ and \[ \mu(d,e;m)(a)>\sqrt{\frac{a}{2}}. \] Assume that $y(a):=a+1-2\sqrt{2a}>\frac{1}{q}$, and set $v_{M}:=\frac{d+e}{q\sqrt{2a}}$. Then
\begin{enumerate}
\item $\left|\sum\varepsilon_{i}\right|\leqslant\sqrt{\sigma L}$,
\item \emph{If} $v_{M}<1$, then $\left|\sum\varepsilon_{i}\right|\leqslant\sqrt{\sigma'L}$,
\item If $v_{M}\leqslant\frac{1}{2}$, then $v_{M}>\frac{1}{3}$ and $\sigma'\leqslant\frac{1}{2}$. If $v_{M}\leqslant\frac{3}{4}$, then $\sigma'\leqslant\frac{7}{8}$,
\item Set $\delta:=y(a)-\frac{1}{q}>0$. Then for both types of classes $(d,d;m)$ and $(d+\frac{1}{2},d-\frac{1}{2};m)$ we have \[ d\leqslant\frac{\sqrt{a}}{\sqrt{2}\delta}\left(\sqrt{\sigma L}-1\right)\leqslant\frac{\sqrt{a}}{\sqrt{2}\delta}\left(\sqrt{\sigma q}-1\right)<\frac{\sqrt{a}}{\sqrt{2}\delta}\left(\frac{\sigma}{\delta v_{M}}-1\right). \] If $v_{M}<1$, $\sigma$ can be replaced by $\sigma'$.
\end{enumerate} %
\end{lem}
\begin{proof}
The proofs of (i), (ii) and (iii) are the same as in the proof of
Lemma~5.1.2 in \cite{MS}. To prove (iv), we notice first that $\sum\varepsilon_{i}<0$.
Indeed, by Lemma~\ref{lem:first properties of mu}.4, $-\sum_{i=1}^{M}\varepsilon_{i}=\frac{d+e}{\sqrt{2a}}\left(y(a)-\frac{1}{q}\right)+1$.
Since $y(a)>\frac{1}{q}$ by assumption, we obtain the desired inequality.

\noindent Then, using (\ref{eq:q great or eq than L}) and (i), we
find
\[
\sqrt{\sigma q}\geqslant\sqrt{\sigma L}\geqslant\frac{d+e}{\sqrt{2a}}\left(y(a)-\frac{1}{q}\right)+1=\frac{d+e}{\sqrt{2a}}\delta+1=\delta qv_{M}+1>\delta qv_{M}.
\]
Thus
\[
\sqrt{q}<\frac{\sqrt{\sigma}}{\delta v_{M}}.
\]
For both types of classes $(d,d;m)$ and $(d+\frac{1}{2},d-\frac{1}{2};m)$,
we get
\[
d\leqslant\frac{\sqrt{a}}{\sqrt{2}\delta}\left(\sqrt{\sigma L}-1\right)\leqslant\frac{\sqrt{a}}{\sqrt{2}\delta}\left(\sqrt{\sigma q}-1\right)<\frac{\sqrt{a}}{\sqrt{2}\delta}\left(\frac{\sigma}{\delta v_{M}}-1\right).
\]
If $v_{M}<1$, the same arguments go through, when replacing $\sigma$
by $\sigma'$.
\end{proof}

\subsection{The interval $\left[6,7\right]$}

We start by stating a more precise version of part (ii) of Theorem
\ref{thm:statement of the result}.
\begin{thm}
\label{thm:c(a) on [sigma^2,7+1/32]}On the interval $\left[\sigma^{2},7\frac{1}{32}\right]$,
$c(a)=\sqrt{\frac{a}{2}}$ except on the seven disjoint intervals
$\left]u_{x},v_{x}\right[$ given in the following table. For each
of these intervals, there exist a class $(d,e;m)\in\mathcal{E}$ and
a rational number $x\in\left]u_{x},v_{x}\right[$ with $l(x)=l(m)$
such that $c(z)=\mu(d,e;m)(z)=\frac{1}{d+e}\left(A+Bz\right)$ on
$\left[u_{x},x\right]$, and $c(z)=\mu(d,e;m)(z)=\frac{1}{d+e}\left(A'+B'z\right)$
on $\left[x,v_{x}\right]$. We list all these informations in the
table below as well as the values of $c(x)$ and $\sqrt{\frac{x}{2}}$.\begin{center} \renewcommand{\arraystretch}{1.5} \begin{tabular}{|c|c|c|c|c|c|c|} \hline  $x$ & $(d,e;m)$ & $(A,B)$ & $(A',B')$ & $c(x)$ & $c(x)\cong$ & $\sqrt{\frac{x}{2}}\cong$\tabularnewline \hline  $6$ & $\left(2,2;2,1^{\times5}\right)$ & $(1,1)$ & $(7,0)$ & $\frac{7}{4}$ & $1.75$ & $1.73$\tabularnewline \hline  $6\frac{1}{7}$ & $\left(28,28;16^{\times6},3,2^{\times6}\right)$ & $(6,15)$ & $(92,1)$ & $\frac{687}{392}$ & $1.752551$ & $1.752549$\tabularnewline \hline  $6\frac{1}{6}$ & $\left(14,14;8^{\times6},2,1^{\times5}\right)$ & $(6,7)$ & $(43,1)$ & $\frac{295}{168}$ & $1.75595$ & $1.75594$\tabularnewline \hline  $6\frac{1}{5}$ & $\left(11,10;6^{\times6},1^{\times5}\right)$ & $(6,5)$ & $(37,0)$ & $\frac{37}{21}$ & $1.762$ & $1.761$\tabularnewline \hline  $6\frac{1}{3}$ & $\left(7,7;4^{\times6},1^{\times3}\right)$ & $(6,3)$ & $(25,0)$ & $\frac{25}{14}$ & $1.79$ & $1.78$\tabularnewline \hline  $6\frac{1}{2}$ & $\left(9,9;5^{\times6},3,2\right)$ & $(0,5)$ & $\left(26,1\right)$ & $\frac{65}{36}$ & $1.81$ & $1.80$\tabularnewline \hline  $7$ & $\left(4,4;3,2^{\times6}\right)$ & $(1,2)$ & $(15,0)$ & $\frac{15}{8}$ & $1.88$ & $1.87$\tabularnewline \hline  \end{tabular} \par\end{center}%
\begin{center} \renewcommand{\arraystretch}{1.5} \begin{tabular}{|c|c|c|c|c|} \hline  $x$ & $u_{x}$ & $u_{x}\cong$ & $v_{x}$ & $v_{x}\cong$\tabularnewline \hline  $6$ & $\sigma^{2}=3+2\sqrt{2}$ & $5.83$ & $6\frac{1}{8}=\frac{49}{8}$ & $6.13$\tabularnewline \hline  $6\frac{1}{7}$ & $\frac{2}{225}\left(347+28\sqrt{151}\right)$ & $6.142842$ & $4\left(173-70\sqrt{6}\right)$ & $6.142872$\tabularnewline \hline  $6\frac{1}{6}$ & $\frac{2}{7}\left(11+4\sqrt{7}\right)$ & $6.16657$ & $153-14\sqrt{110}$ & $6.16676$\tabularnewline \hline  $6\frac{1}{5}$ & $\frac{3}{100}\left(107+7\sqrt{201}\right)$ & $6.19$ & $\frac{2738}{441}$ & $6.21$\tabularnewline \hline  $6\frac{1}{3}$ & $\frac{1}{9}\left(31+7\sqrt{13}\right)$ & $6.2488$ & $\frac{625}{98}$ & $6.38$\tabularnewline \hline  $6\frac{1}{2}$ & $\frac{162}{25}$ & $6.48$ & $55-9\sqrt{29}$ & $6.53$\tabularnewline \hline  $7$ & $\frac{1}{2}\left(7+4\sqrt{3}\right)$ & $6.96$ & $7\frac{1}{32}=\frac{225}{32}$ & $7.03$\tabularnewline \hline  \end{tabular} \par\end{center}%
\end{thm}
\begin{lem}
\label{lem:Obstructive classes at 6+1/k for k=00003D2,...,8}The classes
$(d,e;m)\in\mathcal{E}$ such that $\mu(d,e;m)\left(6\frac{1}{k}\right)>\sqrt{\frac{6\frac{1}{k}}{2}}$
and $l(m)=l\left(6\frac{1}{k}\right)$ for some $k=1,\ldots,8$ are
given in the following table.\begin{center} \renewcommand{\arraystretch}{1.5} \begin{tabular}{|c|c|} \hline  $k$ & $(d,e;m)$\tabularnewline \hline  $7$ & $\left(28,28;16^{\times6},3,2^{\times6}\right)$\tabularnewline \hline  $7$ & $\left(196,196;112^{\times5},111,16^{\times7}\right)$\tabularnewline \hline  $6$ & $\left(14,14;8^{\times6},2,1^{\times5}\right)$\tabularnewline \hline  $6$ & $\left(84,84;48^{\times5},47,8^{\times6}\right)$\tabularnewline \hline  $5$ & $\left(11,10;6^{\times6},1^{\times5}\right)$\tabularnewline \hline  $4$ & $\left(28,28;16^{\times5},15,4^{\times4}\right)$\tabularnewline \hline  $3$ & $\left(7,7;4^{\times6},1^{\times3}\right)$\tabularnewline \hline  $2$ & $\left(9,9;5^{\times6},3,2\right)$\tabularnewline \hline  $1$ & $\left(4,4;3,2^{\times6}\right)$\tabularnewline \hline  \end{tabular} \par\end{center}%
\end{lem}
\begin{proof}
In the case $k=1$, since $6\frac{1}{1}=7$, we only have to check
which elements of the finite set $\mathcal{E}_{7}$ are obstructive
at $7$. It turns out that the only obstructive one is $\left(4,4;3,2^{\times6}\right)$.

Let us now treat the cases $k=2,\ldots,8$. Suppose that there is
a class of the form $(d,d;m)$ which is obstructive at some $6\frac{1}{k}$.
Since $l\left(6\frac{1}{k}\right)=6+k$, by Lemma~\ref{lem:parallel blocks}
the vector $m$ has to be of one of the five forms
\[
\begin{array}{ccc}
\left(a^{\times6},b^{\times k}\right), & \left(a+1,a^{\times5},b^{\times k}\right), & \left(a^{\times5},a-1,b^{\times k}\right),\\
\left(a^{\times6},b+1,b^{\times(k-1)}\right), & \left(a^{\times6},b^{\times(k-1)},b-1\right).
\end{array}
\]
Define $\varepsilon_{a}$ and $\varepsilon_{b}$ by
\[
a=\frac{\sqrt{2}d}{\sqrt{6\frac{1}{k}}}+\varepsilon_{a}\quad\textrm{and}\quad b=\frac{\sqrt{2}d}{k\sqrt{6\frac{1}{k}}}+\varepsilon_{b}.
\]

If $m=\left(a^{\times6},b^{\times k}\right)$, then $\left|a-kb\right|=\left|\varepsilon_{a}-k\varepsilon_{b}\right|\leqslant\left|\varepsilon_{a}\right|+k\left|\varepsilon_{b}\right|$.
Since by Lemma \ref{lem:first properties of mu}\,(iii), $\sum\varepsilon_{i}^{2}<1$,
we find $\left|\varepsilon_{a}\right|+k\left|\varepsilon_{b}\right|<\sqrt{k+1}$,
and thus $\left|a-kb\right|\leqslant\left\lceil \sqrt{k+1}-1\right\rceil $.
Hence
\[
s:=a-kb\in\left\{ \begin{array}{ll}
\left\{ 0,\pm1\right\}  & \textrm{if }k\in\left\{ 2,3\right\} ,\\
\left\{ 0,\pm1,\pm2\right\}  & \textrm{if }k\in\left\{ 4,\ldots,8\right\} .
\end{array}\right.
\]
The Diophantine equations of Proposition \ref{prop:characterization of E_M}\,(i)
then become
\[
\begin{array}{rl}
4d= & 6a+kb+1,\\
2d^{2}= & 6a^{2}+kb^{2}-1.
\end{array}
\]
Thus $\left(6a+kb+1\right)^{2}=8\left(6a^{2}+kb^{2}-1\right)$. Replacing
$a$ by $kb+s$, we can solve this equation in $b$ for the values
of $k$ and $s$ given above. We find three solutions to the equation
with $b\geqslant1$, namely when $(k,s,b)$ is equal to $(3,0,3)$,
$(3,1,1)$ or $(3,-1,5)$. This leads to the vectors $\left(16,16;9^{\times6},3^{\times3}\right)$,
$\left(7,7;4^{\times6},1^{\times3}\right)$ and $\left(25,25;14^{\times6},5^{\times3}\right)$,
respectively. Since only $\left(7,7;4^{\times6},1^{\times3}\right)$
reduces to $(0;-1,0,\ldots,0)$ by Cremona moves, this is the only
class of the form $\left(d,d;a^{\times6},b^{\times k}\right)$ potentially
obstructive at some $6\frac{1}{k}$, and it indeed is obstructive
at $6\frac{1}{3}$.

In the case where $m=\left(a+1,a^{\times5},b^{\times k}\right)$,
$\sigma=k\left|\varepsilon_{b}\right|^{2}\leqslant\frac{1}{6}.$ Thus,
$\left|a-kb\right|\leqslant\left|\varepsilon_{a}\right|+k\left|\varepsilon_{b}\right|\leqslant1+\sqrt{\frac{k}{6}}$,
and thus
\[
s:=a-kb\in\left\{ \begin{array}{ll}
\left\{ 0,\pm1\right\}  & \textrm{if }k\in\left\{ 2,\ldots,5\right\} ,\\
\left\{ 0,\pm1,\pm2\right\}  & \textrm{if }k\in\left\{ 6,\ldots,8\right\} .
\end{array}\right.
\]
From the Diophantine equations we obtain $\left(6a+kb+2\right)^{2}=8\left(6a^{2}+2a+kb^{2}\right)$.
Replacing $a$ by $kb+s$, we obtain no solutions with $b\geqslant1$
for the accepted values of $k$ and $s$.

As in the previous case, when $m=\left(a^{\times5},a-1,b^{\times k}\right)$,
we have
\[
s:=a-kb\in\left\{ \begin{array}{ll}
\left\{ 0,\pm1\right\}  & \textrm{if }k\in\left\{ 2,\ldots,5\right\} ,\\
\left\{ 0,\pm1,\pm2\right\}  & \textrm{if }k\in\left\{ 6,\ldots,8\right\} .
\end{array}\right.
\]
The Diophantine equations become $\left(6a+kb\right)^{2}=8\left(6a^{2}-2a+kb^{2}\right)$,
which yields the four solutions with $b\geqslant1$, namely the tuples
$(k,s,b)$ equal to $(2,1,1)$, $(4,0,4)$, $(6,0,8)$ and $(7,0,16)$
which give the vectors $\left(5,5;3^{\times5},2,1^{\times2}\right)$,
$\left(28,28;16^{\times5},15,4^{\times4}\right)$, $\left(84,84;48^{\times5},47,8^{\times6}\right)$
and $\left(196,196;112^{\times5},111,16^{\times7}\right)$, respectively.
These vectors all reduce to $(0;-1,0,\ldots,0)$ by Cremona moves,
but the first one is not obstructive at $6\frac{1}{2}$. So we add
only the three last vectors to our table.

For the case $m=\left(a^{\times6},b+1,b^{\times(k-1)}\right)$ notice
that if $\varepsilon\in\mathbb{R}$ and $k\in\mathbb{N}$ are such
that $(k-1)\varepsilon^{2}+(\varepsilon+1)^{2}\leqslant1$, then $\varepsilon\in\left[-\frac{2}{k},0\right]$.
Thus, $\left|(k-1)\varepsilon+(\varepsilon+1)\right|=\left|k\varepsilon+1\right|\leqslant1$.
Since $\sigma\geqslant\frac{k-1}{k}$, we get
\[
\begin{array}{rl}
\left|a-kb-1\right| & =\left|a-(b+1)-(k-1)b\right|=\left|\varepsilon_{a}-\left(\varepsilon_{b}+1\right)-(k-1)\varepsilon_{b}\right|\\
 & \leqslant\left|\varepsilon_{a}\right|+\left|(k-1)\varepsilon_{b}+\varepsilon_{b}+1\right|\leqslant1+1=2.
\end{array}
\]
Thus
\[
s:=a-kb-1\in\left\{ 0,\pm1,\pm2\right\} .
\]
The Diophantine equations become $\left(6a+kb+2\right)^{2}=8\left(6a^{2}+kb^{2}+2b\right)$,
which when we replace $a$ by $kb+s+1$ gives the three tuples of
solutions $(k,s,b)$ equal to $(2,0,2)$, $(6,1,6)$ and $(7,1,2)$,
which yields the vectors $\left(9,9;5^{\times6},3,2\right)$, $\left(14,14;8^{\times6},2,1^{\times5}\right)$
and, again, $\left(28,28;16^{\times6},3,2^{\times6}\right)$, respectively.
All three vectors reduce to $(0;-1,0,\ldots,0)$ by Cremona moves,
and they are obstructive at $6\frac{1}{k}$ for $k=2,6,7$ respectively.

For the case $m=\left(a^{\times6},b^{\times(k-1)},b-1\right)$ we
find similarly as in the previous case that
\[
s:=a-kb+1\in\left\{ 0,\pm1,\pm2\right\} .
\]
The Diophantine equations become $\left(6a+kb\right)^{2}=8\left(6a^{2}+kb^{2}-2b\right)$,
which when we replace $a$ by $kb+s-1$ gives as only solution with
$b\geqslant2$ the tuple $(k,s,b)=(2,0,3)$. This gives again the
vector $\left(9,9;5^{\times6},3,2\right)$.

The last case we have to treat is the case of an obstructive class
of the form $(d+\frac{1}{2},d-\frac{1}{2};m)$. By Corollary \ref{cor:an obstructive class (d,d-1;m) has the same blcoks},
the only possibility for $m$ is to be of the form $\left(a^{\times6},b^{\times k}\right)$.
We saw earlier that in this case we have
\[
s:=a-kb\in\left\{ \begin{array}{ll}
\left\{ 0,\pm1\right\}  & \textrm{if }k\in\left\{ 2,3\right\} ,\\
\left\{ 0,\pm1,\pm2\right\}  & \textrm{if }k\in\left\{ 4,\ldots,8\right\} .
\end{array}\right.
\]
Now the Diophantine equations are
\[
\begin{array}{rl}
4d= & 6a+kb+1,\\
2d^{2}= & 6a^{2}+kb^{2}-\frac{1}{2}.
\end{array}
\]
This leads to the equation $\frac{1}{8}\left(6a+kb+3\right)^{2}-\frac{1}{2}\left(6a+kb+3\right)+1=6a^{2}+kb^{2}$.
When replacing $a$ by $kb+s$, we obtain as only solution with $b\geqslant1$
the tuple $(5,1,1)$ which gives the vector $\left(11,10;6^{\times6},1^{\times5}\right)$.
This vector reduces to $(0;-1,0,\ldots,0)$ by Cremona moves and is
obstructive at $6\frac{1}{5}$.\end{proof}
\begin{lem}
\label{lem:no other obstructive classes on [6+1/8,7]}The classes
given in Lemma \ref{lem:Obstructive classes at 6+1/k for k=00003D2,...,8}
are the only obstructive classes on the interval $\left[6\frac{1}{8},7\right]$.\end{lem}
\begin{proof}
We claim that it suffices to prove that for all $a\in\left[6\frac{1}{8},7\right]$,
there is no other class $(d,e;m)\in\mathcal{E}$ with $l(m)=l(a)$
that is obstructive at $a$. Indeed, suppose that $\mu(d,e;m)(a)>\sqrt{\frac{a}{2}}$
for some $a\in\left[6\frac{1}{8},7\right]$, and let $I$ be the maximal
nonempty interval containing $a$ on which $\mu(d,e;m)(z)>\sqrt{\frac{z}{2}}$.
Then, by Lemma~\ref{lem:there exists a unique a_0 center of the class},
there exists a unique $a_{0}\in I$ such that $l\left(a_{0}\right)=l(m)$
and $l\left(a_{0}\right)\leqslant l(a)$ for all $a\in I$. Since
for $M\leqslant6$, $\mathcal{E}_{M}$ is finite, explicit calculations
show that none of these classes is obstructive for $a\geqslant6\frac{1}{8}$.
Thus $l(m)>6$, and $a_{0}>6$. This implies that $a_{0}\geqslant6\frac{1}{8}$.
Indeed, $a_{0}<6\frac{1}{8}$ would contradict the fact that $l\left(a_{0}\right)\leqslant l(a)$
for all $a\in I$ since $l(a)>l\left(6\frac{1}{8}\right)$ for all
$a\in\left]6,6\frac{1}{8}\right[$. A similar argument also shows
that $a_{0}\leqslant7$. Thus, $a_{0}\in\left[6\frac{1}{8},7\right]$
and this proves the claim.

We will thus prove that for each $a=6\frac{p}{q}\in\left[6\frac{1}{8},7\right]$
there is no class $(d,e;m)\in\mathcal{E}$ with $l(m)=l(a)$ obstructive
at $a$, and different from those given in Lemma~\ref{lem:Obstructive classes at 6+1/k for k=00003D2,...,8}.
By Lemma \ref{lem:Obstructive classes at 6+1/k for k=00003D2,...,8},
we only have to prove for $\frac{p}{q}\neq\frac{1}{k}$ with $k=1,\ldots,8$.
We will separate the proof in three cases: $3\leqslant q\leqslant8$,
$9\leqslant q\leqslant39$, $q\geqslant40$.\medskip{}

\textbf{Case 1:} $3\leqslant q\leqslant8$: In this case, $2\leqslant p\leqslant q$.
Notice that for all these values of $p$ and $q$, $y\left(6\frac{p}{q}\right)>\frac{1}{q}$.
We can thus apply Lemma \ref{lem:sup born on d}\,(iv). We get that
if $(d,e;m)\in\mathcal{E}$ is obstructive at $6\frac{p}{q}$, then
\begin{equation}
d\leqslant\frac{\sqrt{6\frac{p}{q}}}{\sqrt{2}\left(y\left(6\frac{p}{q}\right)-\frac{1}{q}\right)}\left(\sqrt{q}-1\right)\label{eq:sup_born_d}
\end{equation}
since $\sigma<1$ for an obstructive class. We now use the computer
program $\mathtt{SolLess[a,D]}$ given in the Appendix which computes
for a rational number $a$ and a natural number $D$ all obstructive
classes $(d,e;m)$ at $a$ with $l(m)=l(a)$ and $d\leqslant D$.
The code shows that there are no such classes for $3\leqslant q\leqslant8$.\medskip{}

\textbf{Case 2:} $9\leqslant q\leqslant39$: Since $y\left(6\frac{1}{8}\right)=\frac{1}{8}$
and $y$ is increasing for $a>2$, we have
\[
y(a)-\frac{1}{q}\geqslant\frac{1}{8}-\frac{1}{9}>0
\]
for all $a\in\left[6\frac{1}{8},7\right]$. We can thus again apply
Lemma \ref{lem:sup born on d}\,(iv) and obtain again (\ref{eq:sup_born_d}),
but this time for $1\leqslant p\leqslant q$. Again, the code $\mathtt{SolLess[a,D]}$
shows that for $9\leqslant q\leqslant39$ there are no obstructive
classes $(d,e;m)$ at $a=6\frac{p}{q}$ with $l(m)=l(a)$.\medskip{}

\textbf{Case 3:} $q\geqslant40$: For all $a=6\frac{p}{q}\in\left[6\frac{1}{8},7\right]$,
we have $\delta:=y(a)-\frac{1}{q}\geqslant\frac{1}{8}-\frac{1}{40}=\frac{1}{10}$.
Suppose that $(d,e;m)\in\mathcal{E}$ is obstructive at some $a=6\frac{p}{q}$
with $q\geqslant40$. We distinguish two cases: (i) $m_{1}=m_{6}$,
(ii) $m_{1}\neq m_{6}$.\medskip{}

(i) Notice that by Lemma \ref{lem:sup born on d}\,(iii),

\[
\begin{array}{ll}
\textrm{if }v_{M}\in\left[\frac{1}{3},\frac{1}{2}\right], & \textrm{then }\frac{\sigma'}{v_{M}}\leqslant\frac{1/2}{1/3}=\frac{3}{2},\vphantom{{\displaystyle \frac{3}{2}}}\\
\textrm{if }v_{M}\in\left[\frac{1}{2},\frac{2}{3}\right], & \textrm{then }\frac{\sigma'}{v_{M}}\leqslant\frac{7/8}{1/2}=\frac{7}{4},\vphantom{{\displaystyle \frac{3}{2}}}\\
\textrm{if }v_{M}\geqslant\frac{2}{3}, & \textrm{then }\frac{\sigma}{v_{M}}\leqslant\frac{3}{2}.\vphantom{{\displaystyle \frac{3}{2}}}
\end{array}
\]
By Lemma \ref{lem:sup born on d}\,(iv) we get that if $a=6\frac{p}{q}\in\left]6\frac{1}{k+1},6\frac{1}{k}\right[$
for some $k=1,\ldots,7$ and $q\geqslant40$, then for all obstructive
classes $(d,e;m)$ at $a$ with $m_{1}=m_{6}$
\[
d\leqslant\frac{\sqrt{6\frac{1}{k}}}{\sqrt{2}\left(y\left(6\frac{1}{k+1}\right)-\frac{1}{40}\right)}\left(\frac{1}{y\left(6\frac{1}{k+1}\right)-\frac{1}{40}}\frac{7}{4}-1\right).
\]
Here we used the computer program $\mathtt{InterSolLess1[k,D]}$ given
in the Appendix which gives for $k\in\left\{ 1,\ldots,7\right\} $
and a natural number $D$ a finite list of classes $(d,e;m)$ with
$m_{1}=m_{6}$ and $d\leqslant D$ which can potentially be obstructive
at some $a=6\frac{p}{q}\in\left]6\frac{1}{k+1},6\frac{1}{k}\right[$
with $q\geqslant40$. Applied to our case, the code gives only one
class that reduces to $(0;-1,0,\ldots,0)$ by Cremona moves, namely
$(d,e;m)=\left(99,99;56^{\times6},14^{\times4},1^{\times3}\right)$.
By Lemma \ref{lem:parallel blocks}, the $a$ in question can be $\left[6;3,1,3\right]=6\frac{4}{15}$
or $\left[6,3,1,1,2\right]=6\frac{5}{18}$, and the class turns out
to give no obstruction at these two points.\medskip{}

(ii) Since $m_{1}\neq m_{6}$, we know by Lemma \ref{lem:parallel blocks}
that $\sigma\leqslant\frac{1}{6}$. This implies that $v_{M}\geqslant1-\frac{1}{2\sqrt{3}}$
because the last two weights of $w\left(\frac{p}{q}\right)$ are always
$\frac{1}{q}$. Then by Lemma \ref{lem:sup born on d}\,(iv) we get
that if $a=6\frac{p}{q}\in\left]6\frac{1}{k+1},6\frac{1}{k}\right[$
for some $k=1,\ldots,7$ and $q\geqslant40$, then for all obstructive
classes $(d,e;m)$ at $a$ with $m_{1}\neq m_{6}$ we have
\[
d\leqslant\frac{\sqrt{6\frac{1}{k}}}{\sqrt{2}\left(y\left(6\frac{1}{k+1}\right)-\frac{1}{40}\right)}\left(\frac{1}{y\left(6\frac{1}{k+1}\right)-\frac{1}{40}}\frac{\frac{1}{6}}{\left(1-\frac{1}{2\sqrt{3}}\right)}-1\right).
\]
Here we used the computer program $\mathtt{InterSolLess2[k,D]}$ which
gives for $k\in\left\{ 1,\ldots,7\right\} $ and a natural number
$D$ a finite list of classes $(d,e;m)$ with $m_{1}\neq m_{6}$ and
$d\leqslant D$ which can potentially be obstructive at some $a=6\frac{p}{q}\in\left]6\frac{1}{k+1},6\frac{1}{k}\right[$
with $q\geqslant40$. Applied to our case, the code gives no class
that reduces to $(0;-1,0,\ldots,0)$ by Cremona moves.\end{proof}
\begin{rem}
The programs $\mathtt{SolLess[a,D]}$, $\mathtt{InterSolLess1[k,D]}$
and $\mathtt{InterSolLess2[k,D]}$ give, for a natural number $D$,
solutions $(d,e;m)$ with $d\leqslant D$. But, in the case of classes
of the form $(d+\frac{1}{2},d-\frac{1}{2};m)$, we give estimates
for $d$ in Lemma~\ref{lem:no other obstructive classes on [6+1/8,7]}.
Thus for these classes, we have to add $\frac{1}{2}$ to our estimates
when using the programs.
\end{rem}
\emph{Proof of Theorem \ref{thm:c(a) on [sigma^2,7+1/32]}.} We have
already proven in Theorem \ref{thm:c(a) on [sigma^2,6]} that the
class $\left(2,2;2,1^{\times5}\right)$ gives the constraint $c(a)=\mu\left(2,2;2,1^{\times5}\right)(a)=\frac{a+1}{4}$
on $\left[\sigma^{2},6\right]$. We postpone the proof that $c(a)=\mu\left(4,4;3,2^{\times6}\right)(a)=\frac{15}{8}$
on $\left[7,7\frac{1}{32}\right]$ to Corollary \ref{cor:c(a) on [7,7+1/32]}.

Since by Lemma \ref{lem:no other obstructive classes on [6+1/8,7]},
the only obstructive classes on the interval $\left[6\frac{1}{8},7\right]$
are those of Lemma \ref{lem:Obstructive classes at 6+1/k for k=00003D2,...,8},
$c\left(6\frac{1}{8}\right)=\frac{7}{4}=\sqrt{\frac{6\frac{1}{8}}{2}}$
because an explicit computation shows that none of them is obstructive
at $6\frac{1}{8}$. Hence, $c(a)=\frac{7}{4}$ for all $a\in\left[6,6\frac{1}{8}\right]$
since $c$ is nondecreasing.

In order to determine $c$ on the interval $\left[6\frac{1}{8},7\right]$,
Lemma \ref{lem:no other obstructive classes on [6+1/8,7]} shows that
we only have to work out the constraints given by the classes of Lemma
\ref{lem:Obstructive classes at 6+1/k for k=00003D2,...,8}. Notice
that for $a\in\left]6\frac{1}{k+1},6\frac{1}{k}\right[$, the first
terms of the weight expansion of $a$ are $w(a)=\left(1^{\times6},(a-6)^{\times k},1-k(a-6),\ldots\right)$.
We can thus easily compute the constraints of all the classes. In
the next table, we write the constraints given by the classes of Lemma
\ref{lem:Obstructive classes at 6+1/k for k=00003D2,...,8} that do
not appear in Theorem \ref{thm:c(a) on [sigma^2,7+1/32]} and we then
simply verify that they indeed do not give new obstructions.\begin{center} \renewcommand{\arraystretch}{1.5} \begin{tabular}{|c|c|c|c|c|} \hline  $x$ & $(d,e;m)$ & $(A,B)$ & $(A',B')$ & $\mu(x)$\tabularnewline \hline  $6\frac{1}{7}$ & $\left(196,196;112^{\times5},111,16^{\times7}\right)$ & $(-1,112)$ & $(687,0)$ & $\frac{687}{392}$\tabularnewline \hline  $6\frac{1}{6}$ & $\left(84,84;48^{\times5},47,8^{\times6}\right)$ & $(-1,48)$ & $(295,0)$ & $\frac{295}{168}$\tabularnewline \hline  $6\frac{1}{4}$ & $\left(28,28;16^{\times5},15,4^{\times4}\right)$ & $(-1,16)$ & $(99,0)$ & $\frac{99}{56}$\tabularnewline \hline  \end{tabular} \par\end{center}

\begin{center}
\begin{center} \renewcommand{\arraystretch}{1.5} \begin{tabular}{|c|c|c|c|c|} \hline  $x$ & $u_{x}$ & $u_{x}\cong$ & $v_{x}$ & $v_{x}\cong$\tabularnewline \hline  $6\frac{1}{7}$ & $\frac{1}{112}\left(344+7\sqrt{2415}\right)$ & $6.142844$ & $\frac{471969}{76832}$ & $6.142870$\tabularnewline \hline  $6\frac{1}{6}$ & $\frac{1}{48}\left(148+7\sqrt{447}\right)$ & $6.16660$ & $\frac{87025}{14112}$ & $6.16674$\tabularnewline \hline  $6\frac{1}{4}$ & $\frac{1}{16}\left(50+7\sqrt{51}\right)$ & $6.2494$ & $\frac{9801}{1568}$ & $6.25$\tabularnewline \hline  \end{tabular} \par\end{center}
\par\end{center}

\begin{center}

\par\end{center}

\begin{flushleft}
The proof of Theorem \ref{thm:c(a) on [sigma^2,7+1/32]} (up to Corollary
\ref{cor:c(a) on [7,7+1/32]}) is complete.\hfill{}$\square$
\par\end{flushleft}

\subsection{The interval $\left[7,8\right]$}
\begin{lem}
\label{lem:sup born on d is 13}Assume that there exists a class $(d,e;m)\in\mathcal{E}$
such that $\mu(d,e;m)(a)>\sqrt{\frac{a}{2}}$ for some $a\in\left[7\frac{1}{32},8\right]$
with $l(a)=l(m)$. Then $m_{1}=\ldots=m_{7}$ and $d\leqslant13$.\end{lem}
\begin{proof}
Notice first that
\[
y(a)\geqslant y\left(7\frac{1}{32}\right)=\frac{17}{32}>\frac{1}{q}
\]
for all $q\geqslant2$. We distinguish two cases: $q\geqslant12$
and $q\leqslant11$.\medskip{}

If $q\geqslant12$, then $\delta=y(a)-\frac{1}{q}\geqslant\frac{17}{32}-\frac{1}{12}=\frac{43}{96}$.
Assume by contradiction that $m_{1}\neq m_{7}$. Then by Lemma \ref{lem:parallel blocks},
$\sigma\leqslant\frac{1}{7}$ and so $v_{M}\geqslant\frac{1}{2}$.
Thus,
\[
\frac{\sigma}{v_{M}\delta}\leqslant\frac{192}{301}<1.
\]
But this contradicts Lemma \ref{lem:sup born on d}\,(iv).

To prove that $d\leqslant13$, notice first that by Lemma \ref{lem:sup born on d}\,(iii),
\[
\begin{array}{ll}
\textrm{if }v_{M}\in\left[\frac{1}{3},\frac{1}{2}\right], & \textrm{then }\frac{\sigma'}{v_{M}}\leqslant\frac{1/2}{1/3}=\frac{3}{2},\vphantom{{\displaystyle \frac{3}{2}}}\\
\textrm{if }v_{M}\in\left[\frac{1}{2},\frac{2}{3}\right], & \textrm{then }\frac{\sigma'}{v_{M}}\leqslant\frac{7/8}{1/2}=\frac{7}{4},\vphantom{{\displaystyle \frac{3}{2}}}\\
\textrm{if }v_{M}\geqslant\frac{2}{3}, & \textrm{then }\frac{\sigma}{v_{M}}\leqslant\frac{3}{2}.\vphantom{{\displaystyle \frac{3}{2}}}
\end{array}
\]
Then, since $\sqrt{a}\leqslant2\sqrt{2}$, we get by Lemma \ref{lem:sup born on d}\,(iv)
that
\[
d\leqslant\frac{2\sqrt{2}}{43/96\sqrt{2}}\left(\frac{1}{43/96}\frac{7}{4}-1\right)+\frac{1}{2}<14.
\]
Thus $d\leqslant13$.\medskip{}

Let now $q\leqslant11$. Notice that $a\leqslant7\frac{q-1}{q}$ and
\[
\delta=y(a)-\frac{1}{q}\geqslant y\left(7\frac{1}{q}\right)-\frac{1}{q}.
\]
By Lemma \ref{lem:sup born on d}\,(iv), we have
\[
d\leqslant\frac{\sqrt{a}}{\sqrt{2}\delta}\left(\sqrt{\sigma q}-1\right)+\frac{1}{2}\leqslant\frac{\sqrt{7\frac{q-1}{q}}}{\sqrt{2}\left(y\left(7\frac{1}{q}\right)-\frac{1}{q}\right)}\left(\sqrt{q}-1\right)+\frac{1}{2}.
\]
Since the RHS is strictly smaller than $11$ for all $2\leqslant q\leqslant11$,
we see that $d\leqslant10$.

Assume now by contradiction that $m_{1}\neq m_{7}$. Then $\sigma\leqslant\frac{1}{7}$.
If $2\leqslant q\leqslant7$, then $\sqrt{\sigma q}-1\leqslant0$
which contradicts Lemma \ref{lem:sup born on d}\,(iv). If $8\leqslant q\leqslant11$,
then
\[
v_{M}=\frac{d+e}{q\sqrt{2a}}\leqslant\frac{\sqrt{2}d}{q\sqrt{a}}\leqslant\frac{\sqrt{2}10}{8\sqrt{7}},
\]
and so, by Lemma \ref{lem:parallel blocks},
\[
\left\langle \varepsilon,\varepsilon\right\rangle \geqslant\frac{6}{7}+2\left(1-v_{M}\right)^{2}>1
\]
which contradicts Lemma \ref{lem:first properties of mu}\,(iii).\end{proof}
\begin{prop}
\label{prop:c(a)=00003Dsqrt(a/2) for a>7+1/32}$c(a)=\sqrt{\frac{a}{2}}$
for all $a\in\left[7\frac{1}{32},8\right]$.\end{prop}
\begin{proof}
Suppose by contradiction that there exists $a\geqslant7\frac{1}{32}$
and $(d,e;m)\in\mathcal{E}$ with $\mu(d,e;m)(a)>\sqrt{\frac{a}{2}}$.
Let $I$ be the maximal open interval containing $a$ on which $(d,e;m)$
is obstructive. Then, by Lemma \ref{lem:there exists a unique a_0 center of the class},
there exists $a_{0}\in I$ with $l\left(a_{0}\right)=l(m)$ and $l(a)\geqslant l\left(a_{0}\right)$
for all $a\in I$.

Using the finite list of $\mathcal{E}_{7}$ in Lemma \ref{lem:finite list of elements in E_M for M leq 7}
we check by hand that no class in $\mathcal{E}_{7}$ is obstructive
for $a\geqslant7\frac{1}{32}$. Thus $l(m)>7$ and so $a_{0}>7$.
But then $a_{0}\geqslant7\frac{1}{32}$. Indeed, assume by contradiction
that $a_{0}<7\frac{1}{32}$. Then since $a_{0},a\in I$, $7\frac{1}{32}$
will also belong to $I$. But, for all $z\in\left]7,7\frac{1}{32}\right[$,
$l(z)>l\left(7\frac{1}{32}\right)$, and this contradicts the fact
that $l(a)\geqslant l\left(a_{0}\right)$ for all $a\in I$.

Now by Lemma \ref{lem:sup born on d is 13}, we find that $d\leqslant13$
and $m_{1}=\ldots=m_{7}$. Since there are only finitely many classes
satisfying these conditions, we can compute them explicitly. We find
that there is only one class satisfying the conditions, namely $\left(8,7;4^{\times7},1\right)$,
but this class is not obstructive for $a\geqslant7\frac{1}{32}$.\end{proof}
\begin{cor}
\label{cor:c(a) on [7,7+1/32]}$c(a)=\frac{15}{8}$ for all $a\in\left[7,7\frac{1}{32}\right]$.\end{cor}
\begin{proof}
Since the class $(d,e;m)=\left(4,4;3,2^{\times6}\right)$ gives the
constraint $\mu(d,e;m)(a)=\frac{15}{8}=\sqrt{\frac{7\frac{1}{32}}{2}}$
for all $a\geqslant7$, we see that $c(a)=\frac{15}{8}$ on $\left[7,7\frac{1}{32}\right]$
because $c$ is nondecreasing by Lemma \ref{lem:scaling property of c}.
\end{proof}
\appendix

\section{Computer programs}

\subsection{Computing $c$ at a point $a\in\left[6\frac{1}{8},7\right]$}

We used the computer in Lemma \ref{lem:no other obstructive classes on [6+1/8,7]}
to compute $c$ at points $\frac{p}{q}\in\left[6\frac{1}{8},7\right]$
with $q\leqslant39$. In this section, we explain the code $\mathtt{SolLess[a,D]}$
which computes for a rational number $a$ and a natural number $D$
all classes $(d,e;m)$ obstructive at $a$, with $l(m)=l(a)$ and
$d\leqslant D$. We have just adapted the program $\mathtt{SolLess[a,D]}$
given in the Appendix of \cite{MS} to our case. The modules $\mathtt{W[a]}$,
$\mathtt{P[k]}$, $\mathtt{Difference[M]}$ are exactly the same as
in \cite{MS}. \medskip{}

\texttt{W{[}a\_{]} := Module{[}\{aa = a, M, i = 2, L, u, v\},}

\texttt{\quad{}M = ContinuedFraction{[}aa{]};}

\texttt{\quad{}L = Table{[}1, \{j, M{[}{[}1{]}{]}\}{]};}

\texttt{\quad{}\{u, v\} = \{1, aa - Floor{[}aa{]}\};}

\texttt{\quad{}While{[}i <= Length{[}M{]},}

\texttt{\quad{}\quad{}L = Join{[}L, Table{[}v, \{j, M{[}{[}i{]}{]}\}{]}{]};}

\texttt{\quad{}\quad{}\{u, v\} = \{v, u - M{[}{[}i{]}{]} v\};}

\texttt{\quad{}\quad{}i++{]};}

\texttt{\quad{}Return{[}L{]}{]}\medskip{}
}

\texttt{P{[}k\_{]} := Module{[}\{kk = k, PP, T0, i\},}

\texttt{\quad{}T0 = Table{[}0, \{u, 1, k\}{]};}

\texttt{\quad{}T0p = ReplacePart{[}T0, 1, 1{]};}

\texttt{\quad{}T1 = Table{[}1, \{u, 1, k\}{]};}

\texttt{\quad{}T1m = ReplacePart{[}T1, 0, -1{]};}

\texttt{\quad{}PP = \{T0, T0p, T1, T1m\};}

\texttt{\quad{}Return{[}PP{]}{]}\medskip{}
}

\texttt{Difference{[}M\_{]} := Module{[}\{V = M, vN, V1, l, L = \{\},
D, PP, i,}

\texttt{\quad{}\quad{}j, N\},}

\texttt{\quad{}l = Length{[}V{]};}

\texttt{\quad{}If{[}l == 1,}

\texttt{\quad{}\quad{}L = P{[}V{[}{[}1{]}{]}{]}}

\texttt{\quad{}\quad{}{]};}

\texttt{\quad{}If{[}l > 1,}

\texttt{\quad{}\quad{}vN = V{[}{[}-1{]}{]};}

\texttt{\quad{}\quad{}V1 = Delete{[}V, -1{]};}

\texttt{\quad{}\quad{}D = Difference{[}V1{]};}

\texttt{\quad{}\quad{}PP = P{[}vN{]};}

\texttt{\quad{}\quad{}i = 1;}

\texttt{\quad{}\quad{}While{[}i <= Length{[}D{]},}

\texttt{\quad{}\quad{}\quad{}j = 1;}

\texttt{\quad{}\quad{}\quad{}While{[}j <= Length{[}PP{]},}

\texttt{\quad{}\quad{}\quad{}\quad{}N = Join{[}D{[}{[}i{]}{]},
PP{[}{[}j{]}{]}{]};}

\texttt{\quad{}\quad{}\quad{}\quad{}L = Append{[}L, N{]};}

\texttt{\quad{}\quad{}\quad{}\quad{}j++{]};}

\texttt{\quad{}\quad{}\quad{}i++{]}}

\texttt{\quad{}\quad{}{]};}

\texttt{\quad{}Return{[}L{]}{]}\medskip{}
}

The following module $\mathtt{Sol0[a,d]}$ gives for a rational number
$a$ all vectors of the form $(d,d;m)$ with $l(m)=l(a)$ which satisfy
the Diophantine equations of Proposition \ref{prop:characterization of E_M}\,(i)
and such that $\mu(d,d;m)(a)>\sqrt{\frac{a}{2}}$. The code $\mathtt{Sol1[a,d]}$
does the same thing for a class of the form $(d,d-1;m)$. Note that
both modules do not verify whether the vectors reduce to $(0;-1)$
by repeated Cremona moves. We have just adapted the code $\mathtt{Sol[a,d]}$
of \cite{MS}, using that in our case, the volume constraint is $\sqrt{\frac{a}{2}}$
instead of $\sqrt{a}$ and that for a class of the form $(d,d;m)$
the Diophantine equations become
\begin{flalign*}
\sum m_{i}=4d-1, & \qquad\sum m_{i}^{2}=2d^{2}+1,
\end{flalign*}
and for a class of the form $(d,d-1;m)$, they become
\begin{flalign*}
\sum m_{i}=4d-3, & \qquad\sum m_{i}^{2}=2d^{2}-2d+1.
\end{flalign*}

\texttt{\medskip{}
}

\texttt{Sol0{[}a\_, d\_{]} := Module{[}\{aa = a, dd = d, M, F, D,
i, V, L = \{\}\},}

\texttt{\quad{}M = ContinuedFraction{[}aa{]};}

\texttt{\quad{}F = Floor{[}((2{*}dd)/Sqrt{[}2{*}aa{]}) W{[}aa{]}{]};}

\texttt{\quad{}D = Difference{[}M{]};}

\texttt{\quad{}i = 1;}

\texttt{\quad{}While{[}i <= Length{[}D{]},}

\texttt{\quad{}\quad{}V = Sort{[}F + D{[}{[}i{]}{]}, Greater{]};}

\texttt{\quad{}\quad{}SV = Sum{[}V{[}{[}j{]}{]}, \{j, 1, Length{[}V{]}\}{]};}

\texttt{\quad{}\quad{}If{[}\{SV, V.V\} == \{4{*}dd - 1, 2{*}dd\textasciicircum{}2
+ 1\} \&\& V{[}{[}-1{]}{]} > 0 \&\& }

\texttt{\quad{}\quad{}\quad{}W{[}aa{]}.V/(2{*}dd) >= Sqrt{[}aa/2{]},}

\texttt{\quad{}\quad{}L = Append{[}L, V{]}}

\texttt{\quad{}\quad{}{]};}

\texttt{\quad{}i++{]};}

\texttt{\quad{}Return{[}\{\{dd, dd\}, Union{[}L{]}\}{]}{]}\medskip{}
}

\texttt{Sol1{[}a\_, d\_{]} := Module{[}\{aa = a, dd = d, M, F, D,
i, V, L = \{\}\},}

\texttt{\quad{}M = ContinuedFraction{[}aa{]};}

\texttt{\quad{}F = Floor{[}((2{*}dd - 1)/Sqrt{[}2{*}aa{]}) W{[}aa{]}{]};}

\texttt{\quad{}D = Difference{[}M{]};}

\texttt{\quad{}i = 1;}

\texttt{\quad{}While{[}i <= Length{[}D{]},}

\texttt{\quad{}\quad{}V = Sort{[}F + D{[}{[}i{]}{]}, Greater{]};}

\texttt{\quad{}\quad{}SV = Sum{[}V{[}{[}j{]}{]}, \{j, 1, Length{[}V{]}\}{]};}

\texttt{\quad{}\quad{}If{[}\{SV, V.V\} == \{4{*}dd - 3, 2{*}dd\textasciicircum{}2
- 2{*}d + 1\}}

\texttt{\quad{}\quad{}\quad{}\&\& V{[}{[}-1{]}{]} > 0 \&\& W{[}aa{]}.V/(2{*}dd
- 1) > Sqrt{[}aa/2{]},}

\texttt{\quad{}\quad{}L = Append{[}L, V{]}}

\texttt{\quad{}\quad{}{]};}

\texttt{\quad{}i++{]};}

\texttt{\quad{}Return{[}\{\{dd, dd - 1\}, Union{[}L{]}\}{]}{]}\medskip{}
}

Finally, we collect in the code $\mathtt{SolLess[a,D]}$ the vectors
$(d,e;m)$ with $l(m)=l(a)$ that are obstructive at $a$ and such
that $d\leqslant D$.

\texttt{\medskip{}
}

\texttt{SolLess{[}a\_, D\_{]} := Module{[}\{aa = a, DD = D, d = 1,
Ld, L = \{\}\},}

\texttt{\quad{}While{[}d <= D,}

\texttt{\quad{}\quad{}Ld = Sol0{[}aa, d{]};}

\texttt{\quad{}\quad{}If{[}Length{[}Ld{[}{[}2{]}{]}{]} > 0,}

\texttt{\quad{}\quad{}\quad{}L = Append{[}L, Ld{]}}

\texttt{\quad{}\quad{}\quad{}{]};}

\texttt{\quad{}\quad{}Ld = Sol1{[}aa, d{]};}

\texttt{\quad{}\quad{}If{[}Length{[}Ld{[}{[}2{]}{]}{]} > 0,}

\texttt{\quad{}\quad{}\quad{}L = Append{[}L, Ld{]}}

\texttt{\quad{}\quad{}\quad{}{]};}

\texttt{\quad{}\quad{}d++{]};}

\texttt{\quad{}Return{[}L{]}{]}}

\subsection{Computing $c$ on an interval $\left]6\frac{1}{k+1},6\frac{1}{k}\right[$
with $k\in\left\{ 1,\ldots7\right\} $}

In Lemma \ref{lem:no other obstructive classes on [6+1/8,7]} we used
the programs $\mathtt{InterSolLess1[k,D]}$ and $\mathtt{InterSolLess2[k,D]}$
which give for $k\in\left\{ 1,\ldots,7\right\} $ and a natural number
$D$, a finite list of vectors $(d,e;m)$ with $d\leqslant D$ which
can potentially be obstructive at some $a\in\left]6\frac{1}{k+1},6\frac{1}{k}\right[$.
By Lemma \ref{lem:parallel blocks}, if a class $(d,e;m)\in\mathcal{E}$
is obstructive at some point $a\in\left[6\frac{1}{8},7\right]$, then
we have three possibilities:

\renewcommand{\labelenumi}{(\roman{enumi})}
\begin{enumerate}
\item $m_{1}=\ldots=m_{6},$
\item $m_{1}-1=m_{2}=\ldots=m_{6},$
\item $m_{1}=\ldots=m_{5}=m_{6}+1.$
\end{enumerate} %

The code $\mathtt{InterSolLess1[k,D]}$ treats the case (i) while
the cases (ii) and~(iii) are covered by $\mathtt{InterSolLess2[k,D]}$.
We used the codes $\mathtt{Solutions[a,b]}$ and $\mathtt{sum[L]}$
exactly as they were in \cite{MS}. $\mathtt{Solutions[a,b]}$ gives
for $a,b\in\mathbb{N}$ all vectors $m$ which are solution of the
equations
\begin{flalign*}
\sum m_{i}=a, & \qquad\sum m_{i}^{2}=b,
\end{flalign*}
and $\mathtt{sum[L]}$ computes the sum of the entries of a vector
$L$. \texttt{\medskip{}
}

\texttt{Solutions{[}a\_, b\_{]} := Solutions{[}a, b, Min{[}a, Floor{[}Sqrt{[}b{]}{]}{]}{]}\medskip{}
}

\texttt{Solutions{[}a\_, b\_, c\_{]} := }

\texttt{\quad{}Module{[}\{A = a, B = b, C = c, i, m, K, j, V, L =
\{\}\},}

\texttt{\quad{}\quad{}If{[}A\textasciicircum{}2 < B,}

\texttt{\quad{}\quad{}\quad{}L = \{\}}

\texttt{\quad{}\quad{}\quad{}{]};}

\texttt{\quad{}\quad{}If{[}A\textasciicircum{}2 == B,}

\texttt{\quad{}\quad{}\quad{}If{[}A > C,}

\texttt{\quad{}\quad{}\quad{}\quad{}L = \{\},}

\texttt{\quad{}\quad{}\quad{}\quad{}L = \{\{A\}\}}

\texttt{\quad{}\quad{}\quad{}\quad{}{]}}

\texttt{\quad{}\quad{}\quad{}{]};}

\texttt{\quad{}\quad{}If{[}A\textasciicircum{}2 > B,}

\texttt{\quad{}\quad{}\quad{}i = 1;}

\texttt{\quad{}\quad{}\quad{}m = Min{[}Floor{[}Sqrt{[}B{]}{]},
C{]};}

\texttt{\quad{}\quad{}\quad{}While{[}i <= m,}

\texttt{\quad{}\quad{}\quad{}\quad{}K = Solutions{[}A - i, B -
i\textasciicircum{}2, i{]};}

\texttt{\quad{}\quad{}\quad{}\quad{}j = 1;}

\texttt{\quad{}\quad{}\quad{}\quad{}While{[}j <= Length{[}K{]},}

\texttt{\quad{}\quad{}\quad{}\quad{}\quad{}V = Prepend{[}K{[}{[}j{]}{]},
i{]};}

\texttt{\quad{}\quad{}\quad{}\quad{}\quad{}L = Append{[}L, V{]};}

\texttt{\quad{}\quad{}\quad{}\quad{}\quad{}j++}

\texttt{\quad{}\quad{}\quad{}\quad{}\quad{}{]};}

\texttt{\quad{}\quad{}\quad{}\quad{}i++{]}}

\texttt{\quad{}\quad{}{]};}

\texttt{\quad{}Return{[}Union{[}L{]}{]}{]}\medskip{}
}

\texttt{sum{[}L\_{]} := Sum{[}L{[}{[}j{]}{]}, \{j, 1, Length{[}L{]}\}{]}\medskip{}
}

\subsubsection{Finding obstructive classes $(d,e;m)$ with $m_{1}=\ldots=m_{6}$}

We have adapted the modules $\mathtt{P[k]}$, $\mathtt{Prelist[k,d]}$
from \cite{MS} to the fact that the first six entries of $m$ have
to be equal instead of the first seven entries as it was the case
in \cite{MS}. The module $\mathtt{Prelist[k,d]}$ becomes $\mathtt{Prelist[k,d,c]}$
where $c=0$ in the case of a class of the form $(d,d;m)$ and $c=1$
when the class is of the form $(d,d-1;m)$. As before, we have adapted
the code to take into account that we have another volume constraint
and other Diophantine equations. Note that \cite{MS} used their Lemma
2.1.7 and Lemma 2.1.8 which are also true in our case as stated in
Lemma \ref{lem:parallel blocks} and Lemma \ref{lem:parallel blocks 2}.

\texttt{\medskip{}
}

\texttt{P{[}k\_{]} := Module{[}\{kk = k, PP, T0, i\},}

\texttt{\quad{}T0 = Table{[}0, \{i, 6 + kk\}{]};}

\texttt{\quad{}Tm = ReplacePart{[}T0, -1, -1{]};}

\texttt{\quad{}Tp = ReplacePart{[}T0, 1, 7{]};}

\texttt{\quad{}PP = \{Tm, T0, Tp\};}

\texttt{\quad{}Return{[}PP{]}}

\texttt{\quad{}{]}\medskip{}
}

\texttt{Prelist{[}k\_, d\_, c\_{]} := }

\texttt{\quad{}Module{[}\{kk = k, dd = d, case = c, u, v, m1, M1,
mx, Mx, f, t,}

\texttt{\quad{}\quad{}PP, M, MM,i = 0, j = 0, s = 1, S, T, K, l,
L = \{\}\},}

\texttt{\quad{}u = 1/(kk + 1);}

\texttt{\quad{}v = 1/kk;}

\texttt{\quad{}m1 = Round{[}(Sqrt{[}2{]}{*}dd)/Sqrt{[}6 + v{]}{]};}

\texttt{\quad{}M1 = Round{[}(Sqrt{[}2{]}{*}dd)/Sqrt{[}6 + u{]}{]};}

\texttt{\quad{}mx = Floor{[}(Sqrt{[}2{]}{*}dd)/Sqrt{[}6 + v{]} u{]}
- 1;}

\texttt{\quad{}Mx = Ceiling{[}(Sqrt{[}2{]}{*}dd)/Sqrt{[}6 + u{]}
v{]} + 1;}

\texttt{\quad{}f = Ceiling{[}Sqrt{[}kk + 2{]} - 1{]};}

\texttt{\quad{}t = -f;}

\texttt{\quad{}PP = P{[}kk{]};}

\texttt{\quad{}While{[}i <= M1 - m1,}

\texttt{\quad{}\quad{}While{[}j <= Mx - mx,}

\texttt{\quad{}\quad{}\quad{}While{[}s <= 3,}

\texttt{\quad{}\quad{}\quad{}\quad{}While{[}t <= f,}

\texttt{\quad{}\quad{}\quad{}\quad{}M = Join{[}Table{[}m1 + i,
\{u, 6\}{]}, Table{[}mx + j, \{u, kk\}{]}{]};}

\texttt{\quad{}\quad{}\quad{}\quad{}M = M + PP{[}{[}s{]}{]};}

\texttt{\quad{}\quad{}\quad{}\quad{}S = Sum{[}M{[}{[}u{]}{]},
\{u, 7, 7 + kk - 1\}{]};}

\texttt{\quad{}\quad{}\quad{}\quad{}M = Append{[}M, M{[}{[}6{]}{]}
- S + t{]};}

\texttt{\quad{}\quad{}\quad{}\quad{}T = 1;}

\texttt{\quad{}\quad{}\quad{}\quad{}If{[}M == Sort{[}M, Greater{]}
\&\& M{[}{[}-1{]}{]} > 0,}

\texttt{\quad{}\quad{}\quad{}\quad{}\quad{}T = 1, T = 0{]};}

\texttt{\quad{}\quad{}\quad{}\quad{}S = sum{[}M{]};}

\texttt{\quad{}\quad{}\quad{}\quad{}If{[}case == 0,}

\texttt{\quad{}\quad{}\quad{}\quad{}\quad{}A = 4{*}dd - 1 - S;}

\texttt{\quad{}\quad{}\quad{}\quad{}\quad{}B = 2{*}dd\textasciicircum{}2
+ 1 - M.M;}

\texttt{\quad{}\quad{}\quad{}\quad{}{]};}

\texttt{\quad{}\quad{}\quad{}\quad{}If{[}case == 1,}

\texttt{\quad{}\quad{}\quad{}\quad{}\quad{}A = 4{*}dd - 3 - S;}

\texttt{\quad{}\quad{}\quad{}\quad{}\quad{}B = 2{*}dd{*}(dd -
1) + 1 - M.M;}

\texttt{\quad{}\quad{}\quad{}\quad{}{]};}

\texttt{\quad{}\quad{}\quad{}\quad{}B = 2{*}dd\textasciicircum{}2
+ 1 - M.M;}

\texttt{\quad{}\quad{}\quad{}\quad{}If{[}Min{[}A, B{]} < 0,}

\texttt{\quad{}\quad{}\quad{}\quad{}\quad{}T = 0{]};}

\texttt{\quad{}\quad{}\quad{}\quad{}If{[}T == 1,}

\texttt{\quad{}\quad{}\quad{}\quad{}\quad{}K = Solutions{[}A,
B, M{[}{[}-1{]}{]}{]};}

\texttt{\quad{}\quad{}\quad{}\quad{}\quad{}l = 1;}

\texttt{\quad{}\quad{}\quad{}\quad{}\quad{}While{[}l <= Length{[}K{]},}

\texttt{\quad{}\quad{}\quad{}\quad{}\quad{}\quad{}MM = Join{[}M,
K{[}{[}l{]}{]}{]};}

\texttt{\quad{}\quad{}\quad{}\quad{}\quad{}\quad{}While{[}MM{[}{[}-1{]}{]}
== 0,}

\texttt{\quad{}\quad{}\quad{}\quad{}\quad{}\quad{}\quad{}MM
= Drop{[}MM, -1{]}{]};}

\texttt{\quad{}\quad{}\quad{}\quad{}\quad{}\quad{}L = Append{[}L,
MM{]};}

\texttt{\quad{}\quad{}\quad{}\quad{}\quad{}\quad{}l++{]}}

\texttt{\quad{}\quad{}\quad{}\quad{}\quad{}{]};}

\texttt{\quad{}\quad{}\quad{}\quad{}t++{]};}

\texttt{\quad{}\quad{}\quad{}t = -f;}

\texttt{\quad{}\quad{}\quad{}s++{]};}

\texttt{\quad{}\quad{}s = 1;}

\texttt{\quad{}\quad{}j++{]};}

\texttt{\quad{}j = 0;}

\texttt{\quad{}i++{]};}

\texttt{\quad{}Return{[}\{\{dd, dd - case\}, Union{[}L{]}\}{]}{]}\medskip{}
}

As in \cite{MS}, the module $\mathtt{InterSol[k,d,c]}$ reduces the
number of candidates given by the code $\mathtt{Prelist[k,d,c]}$.
As before, $c=0$ in the case of a class of the form $(d,d;m)$ and
$c=1$ for a class of the form $(d,d-1;m)$.

\texttt{\medskip{}
}

\texttt{InterSol{[}k\_, d\_, c\_{]} := }

\texttt{\quad{}Module{[}\{kk = k, dd = d, case = c, L, M, T, K =
\{\}, i = 1, l,}

\texttt{\quad{}\quad{}rest\},}

\texttt{\quad{}L = Prelist{[}kk, dd, case{]}{[}{[}2{]}{]};}

\texttt{\quad{}While{[}i <= Length{[}L{]},}

\texttt{\quad{}\quad{}M = L{[}{[}i{]}{]};}

\texttt{\quad{}\quad{}l = Length{[}M{]};}

\texttt{\quad{}\quad{}T = 1;}

\texttt{\quad{}\quad{}If{[}l <= 6 + kk + 2, T = 0{]};}

\texttt{\quad{}\quad{}If{[}M{[}{[}-2{]}{]} - M{[}{[}-1{]}{]} > 1,
T = 0{]};}

\texttt{\quad{}\quad{}If{[}M{[}{[}-3{]}{]} > M{[}{[}-2{]}{]} + 1
\&\& Abs{[}M{[}{[}-3{]}{]} - M{[}{[}-2{]}{]} - M{[}{[}-1{]}{]}{]}}

\texttt{\quad{}\quad{}\quad{}> 1,}

\texttt{\quad{}\quad{}\quad{}T = 0{]};}

\texttt{\quad{}\quad{}If{[}kk == 1 \&\& l >= 9, }

\texttt{\quad{}\quad{}\quad{}If{[}M{[}{[}8{]}{]} - M{[}{[}9{]}{]}
> 1 \&\& Abs{[}M{[}{[}7{]}{]} - (M{[}{[}8{]}{]} + M{[}{[}9{]}{]}){]}}

\texttt{\quad{}\quad{}\quad{}> 1,}

\texttt{\quad{}\quad{}\quad{}\quad{}T = 0{]}{]};}

\texttt{\quad{}\quad{}rest = Sum{[}M{[}{[}j{]}{]}, \{j, 7 + kk,
l\}{]};}

\texttt{\quad{}\quad{}If{[}M{[}{[}6 + kk{]}{]} - rest >= Sqrt{[}l
- kk - 5{]}, T = 0{]};}

\texttt{\quad{}\quad{}If{[}T == 1, K = Append{[}K, M{]}{]};}

\texttt{\quad{}\quad{}i++{]};}

\texttt{\quad{}Return{[}\{\{dd, dd - case\}, K\}{]}{]}\medskip{}
}

Finally, we collect all the solutions for $d\leqslant D$ in the module
$\mathtt{InterSolLess1[k,D]}$.

\texttt{\medskip{}
}

\texttt{InterSolLess1{[}k\_, D\_{]} := Module{[}\{kk = k, DD = D,
LL = \{\}, Q,}

\texttt{\quad{}\quad{}d = 1\},}

\texttt{\quad{}While{[}d <= DD,}

\texttt{\quad{}\quad{}Q = InterSol{[}kk, d, 0{]};}

\texttt{\quad{}\quad{}If{[}Length{[}Q{[}{[}2{]}{]}{]} > 0,}

\texttt{\quad{}\quad{}\quad{}LL = Append{[}LL, Q{]}{]};}

\texttt{\quad{}\quad{}Q = InterSol{[}kk, d, 1{]};}

\texttt{\quad{}\quad{}If{[}Length{[}Q{[}{[}2{]}{]}{]} > 0,}

\texttt{\quad{}\quad{}\quad{}LL = Append{[}LL, Q{]}{]};}

\texttt{\quad{}\quad{}d++{]};}

\texttt{\quad{}Return{[}LL{]}{]}}

\texttt{\medskip{}
}

\subsubsection{Finding obstructive classes $(d,e;m)$ with $m_{1}\neq m_{6}$}

The code $\mathtt{InterSolLess2[k,D]}$ gives for $k\in\left\{ 1,\ldots,7\right\} $
and a natural number $D$, a finite list of vectors $(d,e;m)$ with
$d\leqslant D$ and $m_{1}\neq m_{6}$ which can potentially be obstructive
at some $a\in\left]6\frac{1}{k+1},6\frac{1}{k}\right[$. By Lemma
\ref{cor:an obstructive class (d,d-1;m) has the same blcoks}, if
a class $(d,e;m)\in\mathcal{E}$ with $m_{1}\neq m_{6}$ is obstructive
at some $a\in\left[6,7\right[$, then necessarily $d=e$. Moreover,
Lemma~\ref{lem:parallel blocks} shows that either $m_{1}-1=m_{2}=\ldots=m_{6}$
or $m_{1}=\ldots=m_{5}=m_{6}+1$. Notice that the first terms of the
weight expansion of some $a\in\left]6\frac{1}{k+1},6\frac{1}{k}\right[$
are $\left(1^{\times6};(a-6)^{\times k},\ldots\right)$. Thus the
vector $m$ is either of the form $\left(M+1,M^{\times5},m^{\times k},\ldots\right)$
or of the form $\left(M,(M-1)^{\times5},m^{\times k},\ldots\right)$.
To find the vectors $m$ of the form $\left(M+1,M^{\times5},m^{\times k},\ldots\right)$,
we vary $M$ and $m\leqslant M$ as long as $(M+1)+5M+km\leqslant4d-1$
and $(M+1)^{2}+5M^{2}+km^{2}\leqslant2d^{2}+1$ and then use the code
$\mathtt{Solutions[a,b]}$ from \cite{MS} to find the solutions of
the equations
\[
\begin{aligned}\sum m_{i} & =4d-1-\left((M+1)+5M+km\right),\\
\sum m_{i}^{2} & =2d^{2}-1-\left(M+1)^{2}+5M^{2}+km^{2}\right).
\end{aligned}
\]
The case of a solution vector $m$ of the form $\left(M,(M-1)^{\times5},m^{\times k},\ldots\right)$
is then treated similarly.

\medskip{}

\texttt{InterSolLess2{[}kk\_, DD\_{]} := Module{[}\{k = kk, D = DD,
d, M, m,}

\texttt{\quad{}\quad{}Sol,i, j\},}

\texttt{\quad{}For{[}d = 1, d <= D, d++,}

\texttt{\quad{}\quad{}M = 1;}

\texttt{\quad{}\quad{}While{[}6{*}M + 1 <= 4{*}d - 1 \&\& 6{*}M\textasciicircum{}2
+ 2{*}M + 1 <= 2{*}d\textasciicircum{}2 + 1,}

\texttt{\quad{}\quad{}\quad{}m = 1;}

\texttt{\quad{}\quad{}\quad{}While{[}}

\texttt{\quad{}\quad{}\quad{}\quad{}6{*}M + 1 + k{*}m <= 4{*}d
- 1 \&\& }

\texttt{\quad{}\quad{}\quad{}\quad{}\quad{}6{*}M\textasciicircum{}2
+ 2{*}M + 1 + k{*}m\textasciicircum{}2 <= 2{*}d\textasciicircum{}2
+ 1 \&\& m <= M,}

\texttt{\quad{}\quad{}\quad{}\quad{}Sol = }

\texttt{\quad{}\quad{}\quad{}\quad{}\quad{}Solutions{[}4{*}d
- 1 - (6{*}M + 1 + k{*}m), }

\texttt{\quad{}\quad{}\quad{}\quad{}\quad{}\quad{}2{*}d\textasciicircum{}2
+ 1 - (6{*}M\textasciicircum{}2 + 2{*}M + 1 + k{*}m\textasciicircum{}2){]};}

\texttt{\quad{}\quad{}\quad{}\quad{}If {[}Length{[}Sol{]} > 0,}

\texttt{\quad{}\quad{}\quad{}\quad{}\quad{}For{[}i = 1, i <=
Length{[}Sol{]}, i++,}

\texttt{\quad{}\quad{}\quad{}\quad{}\quad{}\quad{}If{[}Sol{[}{[}i{]}{]}{[}{[}1{]}{]}
<= m,}

\texttt{\quad{}\quad{}\quad{}\quad{}\quad{}\quad{}\quad{}For{[}j
= 1, j <= k, j++,}

\texttt{\quad{}\quad{}\quad{}\quad{}\quad{}\quad{}\quad{}\quad{}Sol{[}{[}i{]}{]}
= Prepend{[}Sol{[}{[}i{]}{]}, m{]};}

\texttt{\quad{}\quad{}\quad{}\quad{}\quad{}\quad{}\quad{}\quad{}{]};}

\texttt{\quad{}\quad{}\quad{}\quad{}\quad{}\quad{}\quad{}For{[}j
= 1, j <= 5, j++,}

\texttt{\quad{}\quad{}\quad{}\quad{}\quad{}\quad{}\quad{}\quad{}Sol{[}{[}i{]}{]}
= Prepend{[}Sol{[}{[}i{]}{]}, M{]};}

\texttt{\quad{}\quad{}\quad{}\quad{}\quad{}\quad{}\quad{}\quad{}{]};}

\texttt{\quad{}\quad{}\quad{}\quad{}\quad{}\quad{}\quad{}Sol{[}{[}i{]}{]}
= Prepend{[}Sol{[}{[}i{]}{]}, M + 1{]};}

\texttt{\quad{}\quad{}\quad{}\quad{}\quad{}\quad{}\quad{}Print{[}\{\{d,
d\}, Sol{[}{[}i{]}{]}\}{]};}

\texttt{\quad{}\quad{}\quad{}\quad{}\quad{}\quad{}\quad{}{]}}

\texttt{\quad{}\quad{}\quad{}\quad{}\quad{}\quad{}{]}}

\texttt{\quad{}\quad{}\quad{}\quad{}\quad{}{]};}

\texttt{\quad{}\quad{}\quad{}\quad{}m++;}

\texttt{\quad{}\quad{}\quad{}\quad{}{]};}

\texttt{\quad{}\quad{}\quad{}M++;}

\texttt{\quad{}\quad{}\quad{}{]};}

\texttt{\quad{}\quad{}M = 1;}

\texttt{\quad{}\quad{}While{[}6{*}M - 1 <= 4{*}d - 1 \&\& 6{*}M\textasciicircum{}2
- 2{*}M + 1 <= 2{*}d\textasciicircum{}2 + 1,}

\texttt{\quad{}\quad{}\quad{}m = 1;}

\texttt{\quad{}\quad{}\quad{}While{[}}

\texttt{\quad{}\quad{}\quad{}\quad{}6{*}M - 1 + k{*}m <= 4{*}d
- 1 \&\& }

\texttt{\quad{}\quad{}\quad{}\quad{}\quad{}6{*}M\textasciicircum{}2
- 2{*}M + 1 + k{*}m\textasciicircum{}2 <= 2{*}d\textasciicircum{}2
+ 1 \&\& m <= M,}

\texttt{\quad{}\quad{}\quad{}\quad{}Sol = }

\texttt{\quad{}\quad{}\quad{}\quad{}\quad{}Solutions{[}4{*}d
- 1 - (6{*}M - 1 + k{*}m), }

\texttt{\quad{}\quad{}\quad{}\quad{}\quad{}\quad{}2{*}d\textasciicircum{}2
+ 1 - (6{*}M\textasciicircum{}2 - 2{*}M + 1 + k{*}m\textasciicircum{}2){]};}

\texttt{\quad{}\quad{}\quad{}\quad{}If {[}Length{[}Sol{]} > 0,}

\texttt{\quad{}\quad{}\quad{}\quad{}\quad{}For{[}i = 1, i <=
Length{[}Sol{]}, i++,}

\texttt{\quad{}\quad{}\quad{}\quad{}\quad{}\quad{}If{[}Sol{[}{[}i{]}{]}{[}{[}1{]}{]}
<= m,}

\texttt{\quad{}\quad{}\quad{}\quad{}\quad{}\quad{}\quad{}For{[}j
= 1, j <= k, j++,}

\texttt{\quad{}\quad{}\quad{}\quad{}\quad{}\quad{}\quad{}\quad{}Sol{[}{[}i{]}{]}
= Prepend{[}Sol{[}{[}i{]}{]}, m{]};}

\texttt{\quad{}\quad{}\quad{}\quad{}\quad{}\quad{}\quad{}\quad{}{]};}

\texttt{\quad{}\quad{}\quad{}\quad{}\quad{}\quad{}\quad{}Sol{[}{[}i{]}{]}
= Prepend{[}Sol{[}{[}i{]}{]}, M - 1{]};}

\texttt{\quad{}\quad{}\quad{}\quad{}\quad{}\quad{}\quad{}For{[}j
= 1, j <= 5, j++,}

\texttt{\quad{}\quad{}\quad{}\quad{}\quad{}\quad{}\quad{}\quad{}Sol{[}{[}i{]}{]}
= Prepend{[}Sol{[}{[}i{]}{]}, M{]};}

\texttt{\quad{}\quad{}\quad{}\quad{}\quad{}\quad{}\quad{}\quad{}{]};}

\texttt{\quad{}\quad{}\quad{}\quad{}\quad{}\quad{}\quad{}Print{[}\{\{d,
d\}, Sol{[}{[}i{]}{]}\}{]};}

\texttt{\quad{}\quad{}\quad{}\quad{}\quad{}\quad{}\quad{}{]}}

\texttt{\quad{}\quad{}\quad{}\quad{}\quad{}\quad{}{]}}

\texttt{\quad{}\quad{}\quad{}\quad{}\quad{}{]};}

\texttt{\quad{}\quad{}\quad{}\quad{}m++;}

\texttt{\quad{}\quad{}\quad{}\quad{}{]};}

\texttt{\quad{}\quad{}\quad{}M++;}

\texttt{\quad{}\quad{}\quad{}{]};}

\texttt{\quad{}\quad{}{]}{]}}


\begin{thebibliography}{Bibliography}
\bibitem[B]{B}P. Biran, Symplectic packing in dimension 4, \emph{Geom.
Funct. Anal.}, \textbf{7} (1997), 420-437.

\bibitem[G]{G}M. Gromov, Pseudoholomorphic curves in symplectic manifolds,
\emph{Invent. Math.}, \textbf{82} (1985), 307-347.

\bibitem[H1]{H1}M. Hutchings, Quantitative embedded contact homology,
\emph{J. Differential Geom.}, \textbf{88} (2011), 231-266.

\bibitem[H2]{H2}M. Hutchings, Recent progress on symplectic embedding
problems in four dimensions, \emph{Proc. Natl. Acad. Sci. USA}, \textbf{108}
(2011), 8093-8099.

\bibitem[K]{K}Y. Karshon, Appendix to \cite{MP}, \emph{Invent. Math.},
\textbf{115} (1994), 431-434.

\bibitem[LiLi]{LiLi}B.-H. Li and T.-J. Li, Symplectic genus, minimal
genus and diffeomorphisms, \emph{Asian J. Math.}, \textbf{6} (2002),
123-144.

\bibitem[LiLiu]{LiLiu}T.-J. Li and A. K. Liu, Uniqueness of symplectic
canonical class, surface cone and symplectic cone of $4$-manifolds
with $b^{+}=1$, \emph{J. Differential Geom.}, \textbf{58} (2001),
331-370.

\bibitem[M1]{M1}D. McDuff, From symplectic deformation to isotopy,
Topics in symplectic $4$-manifolds (Irvine, CA, 1996), 85-99, \emph{First
Int. Press Lect. Ser., I}, Int. Press, Cambridge, MA, 1998.

\bibitem[M2]{M2}D. McDuff, Symplectic embeddings of 4-dimensional
ellipsoids, \emph{J. Topol.} \textbf{2} (2009), 1-22.

\bibitem[M3]{M3}D. McDuff, The Hofer conjecture on embedding symplectic
ellipsoids, \emph{J. Differential Geom.}, \textbf{88} (2011), 519-532.

\bibitem[MP]{MP}D. McDuff and L. Polterovich, Symplectic packings
and algebraic geometry, \emph{Invent. Math.}, \textbf{115} (1994),
405-429.

\bibitem[MS]{MS}D. McDuff and F. Schlenk, The embedding capacity
of 4-dimensional symplectic ellipsoids, \emph{Annals of Math.}, \textbf{175}
(2012), 1191-1282.

\bibitem[S1]{S1}F. Schlenk, Volume preserving embeddings of open
subsets of $\mathbb{R}^{n}$ into manifolds, \emph{Proc. Amer. Math.
Soc.}, \textbf{131} (2003), 1925-1929.

\bibitem[S2]{S2}F. Schlenk, \emph{Embedding problems in symplectic
geometry}, de Gruyter Expositions in Mathematics \textbf{40}, Walter
de Gruyter, Berlin, 2005.

\bibitem[T]{T}L. Traynor, Symplectic packing constructions, \emph{J.
Differential Geom.}, \textbf{42} (1995), 411-429.\end{thebibliography}
\end{document}